\documentclass[10pt]{article}
\usepackage{amsmath,amssymb,amsfonts,amsthm,mathtools,xfrac,extarrows,enumitem, bbold}
\usepackage{graphicx, float}
\usepackage{color}  
\usepackage[T1]{fontenc}  
\usepackage{amscd}  
\usepackage[active]{srcltx}    
\usepackage[font={small}]{caption} 
\usepackage{epstopdf}
\usepackage[T1]{fontenc}
\usepackage{amscd} 
\usepackage{xr,xr-hyper}

\numberwithin{equation}{section}  
\newcommand{\N}{\mathbb{N}}
\newcommand{\R}{\mathbb{R}}

\newcommand{\ind}{\text{ind}}

\makeatletter
\renewcommand*\env@matrix[1][*\c@MaxMatrixCols c]{%
  \hskip -\arraycolsep
  \let\@ifnextchar\new@ifnextchar
  \array{#1}}
\makeatother

\def\hline
{
\vspace{-12pt}
\begin{center}
\rule{0.7\linewidth}{0.5pt}
\end{center}
\vspace{12pt}
}

\newtheorem{remark}{Remark}[section]    
\newtheorem{lemma}{Lemma}[section]
\newtheorem{proposition}{Proposition}[section] 
\newtheorem{thm}{Theorem}
\newtheorem{definition}{Definition}[section]

\numberwithin{equation}{section}

\title{Prescribing Morse scalar curvatures: critical points at infinity}

\author
{
Martin Mayer 
\\
\small University Tor Vergata,
Via della Ricerca Scientifica 1,
00133, ITALY
}

\normalsize

\begin{document}

\maketitle

\medskip
   
\begin{abstract} 

  The problem of prescribing conformally the scalar curvature of a closed Riemannian manifold as a given Morse  function 
reduces to solving an elliptic partial differential equation with critical   Sobolev exponent. Two ways of attacking this 
problem consist in   subcritical approximations or negative pseudo gradient flows. 
We show under a mild non degeneracy assumption the equivalence of both approaches with respect to zero weak limits, in particular a one to one correspondence of
 zero weak limit finite energy subcritical blow-up solutions, zero weak limit critical points at infinity of negative  type
 and sets of critical points with negative Laplacian of the function to be  prescribed.

\end{abstract}

\begin{center}
{\it Key Words :  }Conformal geometry, scalar curvature, subcritical approximation,  critical points at infinity    
\end{center}

\tableofcontents


 %
\section{Introduction}
Prescribing conformally the scalar curvature on a manifold as a given function falls into the class of variational problems, which lack compactness, as  the underlying partial differential equation is critical with respect to Sobolev's embedding. In particular the Palais-Smale condition is violated, which in classical variational theory allows the use of deformation lemmata, which in return are a fundamental pillar in the calculus of variations.

To overcome this lack of compactness one may  try to restore compactness or study a hopefully only slightly different, yet compact situation and pass to the limit or return directly  to the deformation lemmata themselves, hence studying non compact flows. 
The first approach is restrictive to  e.g. symmetric situations with improved Sobolev embedding, the second one leads to the idea of compact  approximation and the third one to the theory of critical points at infinity. 

\ 

Let us comment on the corresponding ideas. First and famously in order to restore compactness the positive mass theorem has been used, cf. \cite{Schoen_Using_Mass}. Here the argument is, that a certain sublevel set of the  variational functional is shown to be compact, while, assuming sufficient flatness of $K$ or even $K$ to be constant, the positive mass term becomes dominant in the expansion of the energy of a specific test function pushing its energetic value below the threshold of the sublevel set, i.e. the  test function already lies withing the latter, which is therefore not empty.  Hence one can find a minimizer by direct methods.  

For compact  approximations, cf. \cite{Brezis_Nirenberg},\cite{yy2},\cite{MM2},\cite{MM4} in contrast the underlying equations can be solved classically, whereas the passage to the critical limit then has to be understood in detail. The advantage of this approach is, that one deals with a sequence of solutions to specific equations rather than with arbitrary Palais-Smale sequences. Of course there will be a lack of compactness, i.e. there will be, as we pass to the critical limit, solutions, which do not converge in the variational space. But one may hope to find at least some sequences, which remain compact, thus providing a solution to the critical equation itself.    

Similarly in the context of studying non compact flows, cf. \cite{bc},\cite{Chang_Yang_S2},\cite{may-cv}, i.e. returning to the study of energy deformation, we do not have to study arbitrary Palais-Smale sequences, but flow lines. And the liberty is, that we are not bound to study a specific, but an energy deformation of our choice. In particular given a flow exhibiting non compactness somewhere, we may hope to avoid the latter by adapting the former, as was done in \cite{MM5}.
While in \cite{Chang_Yang_S2} classical min-max schemes are established by excluding certain non compactness scenarios, in \cite{bc} the topological effect of non compact flow lines to sublevels sets is computed. The difference is, that while the first result is based on avoiding non compactness, the second one uses this non compactness by understanding its topological contribution directly, which is a central topic in the theory of critical points at infinity. 

\ 

Evidently in case of compact, for instance subcritical approximations or the study of non compact flow lines one has to understand  and describe the lack of compactness in absence of at least partial compactness as in \cite{Schoen_Using_Mass} qualitatively. A natural question is, whether or not one can expect to find different results by means of subcritical approximation or the study of non compact flows, as the first describes subcritical non compact sequences of solutions and the latter non compact flow lines. 
Comparing Theorems \ref{t:ex-multi} and \ref{thm} this does not seem to be the case.
 
\subsection{Setting} 

Consider a closed Riemannian manifold
\begin{equation*}
M=(M^{n},g_{0}) 
\; \text{ with } \; n\geq 5,
\end{equation*}
 volume measure $\mu_{g_{0}}$ and scalar curvature $R_{g_{0}}$.
We assume the Yamabe invariant 
\begin{equation}\begin{split}\label{Yamabe_invariant}
Y(M,g_{0})
= &
\inf_{\mathcal{A}}
\frac
{\int \left( c_{n}\vert \nabla u \vert_{g_{0}}^{2}+R_{g_{0}}u^{2} \right) d\mu_{g_{0}}}
{(\int u^{\frac{2n}{n-2}}d\mu_{g_{0}})^{\frac{n-2}{n}}}, \; c_{n}=4\frac{n-1}{n-2}, 
\end{split}\end{equation}
where 
$$
\mathcal{A}=
\{
u\in W^{1,2}(M,g_{0})\mid u\geq 0,u\not \equiv 0
\},
$$
to be positive  positive.
As a consequence the conformal Laplacian
\begin{equation*}\begin{split} 
L_{g_{0}}=-c_{n}\Delta_{g_{0}}+R_{g_{0}}
\end{split}\end{equation*}
 is a positive and self-adjoint operator. Without loss of generality 
we assume $R_{g_{0}}>0$ and denote by
$$
G_{g_{0}}:M\times M \setminus \Delta(M)  \longrightarrow\R_{+}
\; \text{ with } \; 
\Delta(M)=\{ (m_{1},m_{2})\in M\times M \; : \; m_{1}=m_{2}\} 
$$
the Green's function of $L_{g_0}$.
Considering a conformal metric $g=g_{u}=u^{\frac{4}{n-2}}g_{0}$ there holds 
\begin{equation*}\begin{split} d\mu_{g_{u}}=u^{\frac{2n}{n-2}}d\mu_{g_{0}}
\;  \text{ and }\; 
R=R_{g_{u}}=u^{-\frac{n+2}{n-2}}(-c_{n} \Delta_{g_{0}} u+R_{g_{0}}u) =
u^{-\frac{n+2}{n-2}}L_{g_{0}}u
\end{split}\end{equation*}
due to conformal covariance of the conformal Laplacian, i.e.
\begin{equation*}
-c_{n}\Delta_{g_{u}}v+R_{g_{u}}v=L_{g_{u}}v=u^{-\frac{n+2}{n-2}}L_{g_{0}}(uv).
\end{equation*}
So prescribing conformally  the scalar curvature $R=K$ as a given function $K$ is equivalent to solving 
\begin{equation}\label{the_equation}
L_{g_{0}}u=Ku^{\frac{n+2}{n-2}}
\end{equation}
and the Green's function $G_{g_{u}}$ for $L_{g_{u}}$ transforms according to
\begin{equation*}
G_{g_{u}}(x,y)=u^{-1}(x)G_{g_{0}}(x,y)u^{-1}(y).
\end{equation*}
Moreover  we may  associate in a unique and smooth way to every  $a \in M$  a suitable conformal metric 
$$
g_{a}=g_{u_{a}}=u_{a}^{\frac{4}{n-2}}g_{0} 
\; \text{ with } \;
u_{a}=1+O(d^{2}_{g_{0}}(a,\cdot))
$$
such, that in a geodesic normal coordinate system for $ g_{a} $, which we call a conformal normal coordinate system for $ g_{0} $,
the volume element is locally euclidean, i.e. 
\begin{equation*}
\begin{split}
d\mu_{g_{a}}=dL^{n} 
\; \text{ close to  } \; a\in M,
\end{split}
\end{equation*}
cf. \cite{lp}.
In particular
\begin{equation*}
\begin{split} 
(exp_{a}^{g_{0}})^{-}\circ \exp_{a}^{g_{a}}(x)
=
x+O(\vert x \vert^{3})
\end{split}
\end{equation*}
for the exponential maps centered at $ a $, which e.g. implies
\begin{equation}\label{def_derivatives_g_0}
\begin{split}
\nabla_{g_{0}}K(a)=\nabla_{g_{a}}K(a),\; 
\nabla^{2}_{g_{0}}K(a)=\nabla^{2}_{g_{a}}K(a), 
\end{split}
\end{equation}
and in case $ \nabla_{g_{0}}K(a)=0 $  also 
$
\nabla^{3}_{g_{0}}K(a)=\nabla^{3}_{g_{a}}K(a).  
$
Then 
$$
G_{a}=G_{g_{a}}(a,\cdot),
$$
i.e. the Green's function $G_{g_{ a }}$ 
with pole at $a \in M$ for the conformal Laplacian 
$$L_{g_{a}}=-c_{n}\Delta_{g_{a}}+R_{g_{a}}$$
expands with $\omega_{n} = |S^{n-1}|$ denoting the unit volume as
\begin{equation*}\begin{split} 
G_{ a }=\frac{1}{4n(n-1)\omega _{n}}(r^{2-n}_{a}+H_{ a }), \; r_{a}=d_{g_{a}}(a, \cdot)
, \; 
H_{ a }=H_{r,a }+H_{s, a }, 
\end{split}\end{equation*}
where
$r_a$ denotes the geodesic distance from $a$ with respect to the metric $g_a$,
\begin{equation*}
\begin{split}
H_{s,a}
=
O
\begin{pmatrix}
r_{a}& \text{ for }\, n=5
\\
\ln r_{a} & \text{ for }\, n=6
\\
r_{a}^{6-n} & \text{ for }\, n\geq 7
\end{pmatrix}
\end{split}
\end{equation*}  
and $H_{r,a }\in C^{2, \alpha}_{loc}$. As due to $ R_{g_{0}}>0 $ 
\begin{align*}
c\Vert u \Vert^{2}_{W^{1,2}(M,g_0)}\leq\int u \, L_{g_{0}}u \, d\mu_{g_{0}}=\int \left( c_{n}\vert \nabla u \vert^{2}_{g_{0}}+R_{g_{0}}u^{2}
\right) d\mu_{g_{0}}\leq C\Vert u \Vert^{2}_{W^{1,2}(M,g_0)}
\end{align*}
with positive constants $ 0<c<C<\infty $, we may define and use
\begin{equation*}
\Vert u \Vert^2 = \Vert u \Vert_{L_{g_0}}^2 = \int u \, L_{g_{0}}u \, d\mu_{g_{0}}
\end{equation*}
as an equivalent norm on $W^{1,2}$. We then wish to study  the scaling invariant functional
\begin{equation*}
J:\mathcal{A}\longrightarrow \R:
u\longrightarrow 
\frac
{\int  L_{g_{0}}uu d\mu_{g_{0}}}
{(\int Ku^{\frac{2n}{n-2}}d\mu_{g_{0}})^{\frac{n-2}{n}}}
\; \text{ for } \;
K>0. 
\end{equation*}
Since the conformal scalar curvature $R=R_{u}$ for $g=g_{u}=u^{\frac{4}{n-2}}g_{0}$ satisfies
\begin{equation*}
r=r_{u}=\int R d\mu_{g_{u}}=\int u L_{g_{0}} ud\mu_{g_{0}},
\end{equation*}
we have 
\begin{equation}\label{the_functional}
J(u)=\frac{r}{k^{\frac{n-2}{n}}}
\;  \text{ with } \; k = k_{u}= \int K \, u^{\frac{2n}{n-2}} d \mu_{g_{0}}. 
\end{equation}
The first  and second order derivatives of the functional are given by 
\begin{equation*}
\partial J(u)v
= 
\frac{2}{k^{\frac{n-2}{n}}}
\big[\int L_{g_{0}}uvd\mu_{g_{0}}-\frac{r}{k}\int Ku^{\frac{n+2}{n-2}}vd\mu_{g_{0}}\big],
\end{equation*} 
hence and in particular  \eqref{the_equation} has variational structure, and
\begin{equation*}
\begin{split}
\partial^{2} J(u) vw
= &
\frac{2}{k^{\frac{n-2}{n}}}
\big[\int L_{g_{0}}vwd\mu_{g_{0}}-\frac{n+2}{n-2}\frac{r}{k}\int Ku^{\frac{4}{n-2}}vwd\mu_{g_{0}}\big] \\
& -
\frac{4}{k^{\frac{n-2}{n}+1}}
\big[
\int L_{g_{0}}uvd\mu_{g_{0}}\int Ku^{\frac{n+2}{n-2}}wd\mu_{g_{0}}
+
\int L_{g_{0}}uwd\mu_{g_{0}}\int Ku^{\frac{n+2}{n-2}}vd\mu_{g_{0}}
\big] \\
& +
\frac{2(\frac{n+2}{n-2}+3)r}{k^{\frac{n-2}{n}+2}}
\int Ku^{\frac{n+2}{n-2}}vd\mu_{g_{0}}\int Ku^{\frac{n+2}{n-2}}wd\mu_{g_{0}}.
\end{split}
\end{equation*}
Note, that  $J$ is of class $C^{2, \alpha}(\mathcal{A} \cap \{c^{-1}<k<c\})$ for every $c>1$ and that  the scalar product
\begin{equation*}
\langle u,w \rangle=\langle u,w \rangle_{L_{g_{0}}}=\int L_{g_{0}}uwd\mu_{g_{0}}
\end{equation*}
induces the gradient $ \nabla J=\nabla_{L_{g_{0}}}J$, i.e. 
$$
\langle \nabla J(u),w \rangle
=
\langle \nabla_{L_{g_{0}}} J(u),w \rangle_{L_{g_{0}}}
=
\partial J(u)w
$$
or in other words $ \nabla_{L_{g_{0}}} J(u)=L^{^{-}}_{g_{0}}\partial J(u) $ with $ L_{g_{0}}^{^{-}} $ denoting the inverse to  
\begin{equation*}
\begin{split}
L_{g_{0}}
:
W^{1,2}(M,g_{0})
\longrightarrow 
W^{-1,2}(M,g_{0})
\end{split}
\end{equation*}
mapping $ W^{1,2}(M,g_{0}) $ to its dual.  Likewise and for the sake of brevity let us also write 
\begin{equation*}
\begin{split}
\nabla^{k} K=\nabla^{k}_{g_{0}}K
\; \text{ for } \; k\geq 1
\; \text{  and } \;
\Delta K=\Delta_{g_{0}} K.
\end{split}
\end{equation*}

\subsection{Sub\,- and criticality}  
Let us review first the subcritical, non degenerate case.
\begin{definition}
We call a positive Morse function $K$ on $M$ non degenerate for  $n\geq 5$, if 
\begin{equation}\label{nd} 
\{\vert \nabla K \vert   = 0 \}\cap \{\Delta K = 0 \}= \emptyset. 
\end{equation}
 \end{definition}
 
We will always assume this non degeneracy, under which in \cite{MM1} and \cite{MM2} we proved for the subcritical approximation to \eqref{the_equation}, i.e.  
\begin{equation}\label{subcritical_equation}
 L_{g_{0}}u=Ku^{\frac{n+2}{n-2}-\tau}\; \text{ with }\;  0<\tau \longrightarrow 0
\end{equation} 
and subcriticality is understood with respect to Sobolev's embedding,
the following uniqueness and existence result. 
Clearly $J=J_{0}$. 
\begin{thm}[\cite{MM1},\cite{MM2}]\label{t:ex-multi}
Let $(M,g)$ be a compact manifold of dimension $n \geq 5$ with positive Yamabe invariant and   
$K : M \longrightarrow \R$ be a positive Morse function satisfying \eqref{nd}. 
Let $x_1, \dots, x_q$ be distinct critical points of $K$ with negative Laplacian.

Then there exists, as $\tau \longrightarrow 0$ and up to scaling,  a unique solution $u_{\tau, x_1, \dots, x_q}$
to \eqref{subcritical_equation} developing a simple bubble  
at each  $x_i$ and converging weakly to zero as $\tau \longrightarrow 0$. 
Moreover and  up to scaling 
$$
m(J_\tau, u_{\tau, x_1, \dots, x_q}) 
=
(q-1) + \sum_{i=1}^q 
(n-m(K,x_i)).
$$ 
Conversely  all blow-up solutions  of uniformly bounded energy and zero weak limit type
are as above.   
\end{thm}

Here  $m=m(\cdot,\cdot)$ denotes the Morse index, while blow-up refers to local concentration
\begin{equation*}
\begin{split}
\forall\; 0<\varepsilon\ll 1 \;\exists\; \lambda_{\tau}\longrightarrow \infty
:
\sup_{x\in M} \int_{B_{\frac{1}{\lambda_{\tau}}}(x)}\vert \nabla u_{\tau}\vert^{2}
\geq \varepsilon  
\end{split}
\end{equation*}
of solutions $ u_{\tau}\in \{ \partial J_{\tau}=0 \}  $, cf. (3.2) of Proposition 3.1 in \cite{MM1},  
and simple bubbling to 
\begin{equation*}
\forall \; i=1,\ldots, q \; 
\exists\; R_{\tau}\xlongrightarrow{\tau\to 0} \infty\;:\;
\Vert 
u_{\tau,x_{1},\ldots,x_{q}} - \alpha_{i} \varphi_{a_{i},\lambda_{i} } 
\Vert_{W^{1,2}(B_{\frac{R_{\tau}}{\lambda_{\tau}}})}
\xrightarrow{\tau \to \infty} 0.
\end{equation*} 
Precisely and with a multiplicative constant $\Theta$
reflecting the scaling invariance of $J_{\tau}$   
\begin{equation}\label{Theta}
\alpha_{j}=\frac{\Theta}{K(x_{j})^{\frac{n-2}{4}}}+o_{\tau}(1)
,\quad  
a_{j} \xrightarrow{\tau\to 0} x_j
\; \text{ and } \;
\lambda_{j} \simeq \tau^{- \frac{1}{2}}
,
\end{equation}
cf. Remark 1.1 in \cite{MM1}.
Finally the functions
\begin{equation}\begin{split}\label{eq:bubbles}
\varphi_{a, \lambda }
= &
u_{ a }\left(\frac{\lambda}{1+\lambda^{2} \gamma_{n}G^{\frac{2}{2-n}}_{ a }}\right)^{\frac{n-2}{2}}  
\; \text{ with }\;  
G_{ a }=G_{g_{ a }}( a, \cdot)
\; \text{ and } \;
\gamma_{n}=(4n(n-1)\omega _{n})^{\frac{2}{2-n}},
\end{split}\end{equation}
to which we refer as bubbles,
are zero weak limit \textit{almost solutions} to \eqref{the_equation}, precisely
$$ 
\varphi_{a,\lambda}\xrightharpoondown{\lambda\to \infty } 0
\; \text{ weakly with } \;
\limsup_{\lambda\to \infty}J(\varphi_{a,\lambda})<\infty
\; \text{ and } \;
\Vert \partial J(\varphi_{a,\lambda}) \Vert \xrightarrow{\lambda\to \infty} 0
$$ 
uniformly, cf. \eqref{def_PS} and Lemma \ref{lem_emergence_of_the_regular_part}, hence also for $\eqref{subcritical_equation}$. We refer to Section \ref{section_preliminaries} for precise statements.

\

The proof of Theorem \ref{t:ex-multi} is  based on considering the variational functional
\begin{equation*}
J_{\tau}(u)
=
\frac
{\int  L_{g_{0}}uu d\mu_{g_{0}}}
{(\int Ku^{p+1}d\mu_{g_{0}})^{\frac{2}{p+1}}}
\; \text{ for }\;
 u \in \mathcal{A}
 \; \text{ and with }\; 
  p=\frac{n+2}{n-2}-\tau
\end{equation*} 
corresponding to \eqref{subcritical_equation}. 
As $ J_{\tau} $ is scaling invariant, we may restrict to
\begin{equation}\label{the_variational_space}
X=\{0\leq u \in W^{1,2}(M) \mid u\neq 0,\; \Vert u \Vert=1\}
\end{equation}
and consider $ X $ as the variational space of $J_{\tau}$. 
A preliminary blow up analysis of zero weak limit Palais-Smale sequences, i.e. 
\begin{equation*}
\begin{split}
X \supset (u_{m})\xrightharpoondown{m\to \infty } 0
\; \text{ weakly with } \;
\limsup_{m\to \infty}J_{\tau}(u_{m})<\infty
\; \text{ and } \;
\Vert \partial J_{\tau}(u_{m}) \Vert \xrightarrow{m\to \infty} 0, 
\end{split}
\end{equation*}
then shows, that  every zero weak limit Palais-Smale sequence has to be of the form of a finite sum
\begin{equation}\label{bubbling}
u=\alpha^{i}\varphi_{a_{i},\lambda_{i}} + v,\;\Vert v \Vert \ll 1 
\end{equation}
with 
\begin{enumerate}[label=(\roman*)]
\item 	scaling parameters $c<\alpha_{i}<C$ 
\item   high concentrations $\lambda_{i}\longrightarrow \infty$
\item   small error $ \Vert v \Vert\longrightarrow 0 $.
\end{enumerate} 
Vice versa such functions induce zero weak limit type Palais-Smale sequences. 
This representation
is  rendered unique by means of a minimisation problem, which provides certain orthogonality relations, when
testing the derivative  
$\partial J_{\tau}$ 
along 
$$
\varphi_{i}
,
\lambda_{i}\partial_{\lambda_{i}}\varphi_{i}
,
\frac{\nabla_{a_{i}}}{\lambda_{i}}\varphi_{i}
\; \text{ and } \;
v.
$$
Then a sophisticated combination of such testings provides a  lower bound on  $\vert \partial J_{\tau}\vert$, which in return reduces to a high degree the possible configurations of the parameters
$\alpha_{i},\lambda_{i}$ and $a_{i}$ for zero weak limit blow-up solutions. In particular and necessarily
\begin{equation}\label{principal_a_i_behaviour}
a_{i}\longrightarrow x_{i}\in \{\vert \nabla K \vert =0\}\cap \{\Delta K <0\}
\; \text{ as } \; 
\tau\longrightarrow 0
\; \text{ and }\; 
x_{i}\neq x_{j} 
\; \text{ for }\; 
i\neq j,
\end{equation} 
thus excluding tower bubbling, i.e. $ x_{i}=x_{j} $ for some $ i\neq j $.     
Finally in \cite{MM2} and  based on calculations of the second derivative $\partial^{2}J$ the sharpness of  \eqref{principal_a_i_behaviour} is established, meaning  that for every 
\begin{equation*}
\{x_{1},\ldots ,x_{q}\}  \subseteq \{\vert \nabla K \vert =0\}\cap \{\Delta K <0\}
\end{equation*}
there exists a unique solution $u\in \{\partial J_{\tau}=0\}$ of type 
$u=\sum_{i}\alpha_{i}\varphi_{a_{i},\lambda_{i}} + v$ with
\begin{equation}\label{subcritical_principal_behaviour}
\frac{1}{\lambda_{i}}
,
d(a_{i},x_{i})
,
\vert 1-\frac{\sum_{j}\alpha_{j}^{2}}{\sum_{j}K(a_{j})\alpha_{j}^{\frac{2n}{n-2}}}K(a_{i})\alpha_{i}^{\frac{4}{n-2}}\vert,
\Vert v \Vert 
\longrightarrow 0
\; \text{ as } \; 
\tau\longrightarrow 0
\end{equation}
and the latter convergence is understood to a high degree
in $\tau$. Hence Theorem \ref{t:ex-multi}.

\

Taking also the scaling invariance of $J$ into account we consider again $ X $, cf. \eqref{the_variational_space}, 
as the variation space of $ J=J_{0} $   
and will in the present work construct a semi-flow
\begin{equation*}
\Phi: \R_{\geq 0}  \times X \longrightarrow X
\end{equation*}
rigorously defined in Section \ref{section_flow_construction},
which decreases the energy $J$, and study its zero weak limit flow lines. 

\

Roughly speaking and in analogy to the subcritical case the corresponding parabolic blow-up analysis and unique representation lead to the description \eqref{bubbling} of zero weak limit non compact flow lines for e.g. the strong gradient flow. In particular non compactness of a flow line corresponds to 
$ \forall_{i}\lambda_{i}\longrightarrow \infty $ for at least a sequence in time.  Then the flow $ \Phi $, while preserving $ X $,  is finely tuned to a careful evaluation and combination of testings of the derivative $ \partial J $ with the scope to increase $ \max_{i} \lambda_{i} $ along $ \Phi $ as little as possible, whenever a flow line is of type \eqref{bubbling},
while moving along the strong gradient flow otherwise.
In particular this allows us  
to show, that in analogy to \eqref{principal_a_i_behaviour} 
along zero weak limit flow lines of $ \Phi $ there necessarily holds 
\begin{equation*}
a_{i}\longrightarrow x_{i} \in \{\vert \nabla K \vert =0\}\cap \{\Delta K <0\}
\; \text{ as} \; 
t\longrightarrow \infty
\; \text{ and } \; 
x_{i}\neq x_{j}\; \text{ for } \; i\neq j.
\end{equation*}
And again in analogy to Theorem \ref{t:ex-multi}
we show, that for every
\begin{equation*}
\{x_{1},\ldots ,x_{q}\}  \subseteq \{\vert \nabla K \vert =0\}\cap \{\Delta K <0\} 
\end{equation*}
there exists a flow line  
$$u(t)=\Phi(t,u_{0})
\; \text{ of type } \;
u=\sum_{i}\alpha_{i}\varphi_{a_{i},\lambda_{i}} + v$$ 
with
\begin{equation*} 
\frac{1}{\lambda_{i}}
,\;
d(a_{i},x_{i})
,\; 
\vert 1-\frac{\sum_{j}\alpha_{j}^{2}}{\sum_{j}K(a_{j})\alpha_{j}^{\frac{2n}{n-2}}}K(a_{i})\alpha_{i}^{\frac{4}{n-2}}\vert,\;
\Vert v \Vert 
\longrightarrow 0
\; \text{ as } \;  
t \longrightarrow \infty
\end{equation*} 
exponentially fast, cf. \eqref{subcritical_principal_behaviour}. 

\ 
 
Consequently zero weak limit flow lines for this energy decreasing flow $\Phi$ and finite energy zero weak limit subcritical blow-up solutions display the same limiting behaviour. And, since from the computation of the second derivative $\partial^{2} J_{\tau}$ at the latter subcritical blow-up solutions the induced change of topology of  sublevel sets is known according to their Morse index, 
the same change of topology is induced by corresponding  \textit{critical points at infinity}, for whose definition we refer to \cite{bab_calc_var} and Section \ref{section_critical_points_at_infinity}.
\begin{thm}\label{thm}
Let $(M,g)$ be a compact manifold of dimension $n \geq 5$ with positive Yamabe invariant  and 
let $K : M \longrightarrow \R$ be a positive Morse function satisfying \eqref{nd}.    
Let $x_1, \dots, x_q$ be distinct critical points of $K$ with negative Laplacian. 

Then there exists up to scaling a unique critical point at infinity  $c=u_{\infty, x_1, \dots, x_q}$ for $ J $ of zero weak limit, energy decreasing type exhibiting  a simple peak 
at each point $x_i$.
Moreover $c=u_{\infty, x_1, \dots, x_q}$ has index  
$$\ind(J, u_{\infty, x_1, \dots, x_q}) = (q-1) + \sum_{i=1}^q 
(n-m(K,x_i)).$$   
Conversely all critical points at infinity of energy decreasing  and zero weak limit type
are as above. 
\end{thm}   

\begin{figure}[h!]
\centering
\includegraphics[scale=0.8]{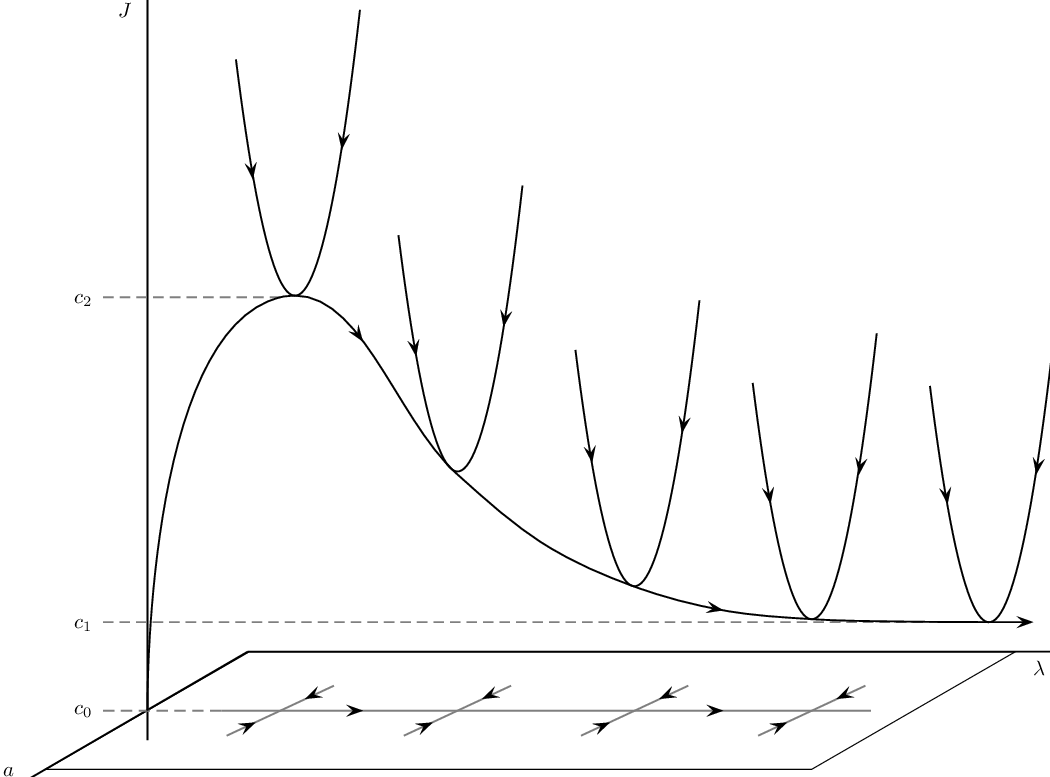} 
\caption{Any point $(a,\lambda) \in J^{c_{2-\delta}}$ with $ \vert a \vert\ll 1 \ll\lambda  $ inevitably flows towards $(0,\infty)$}
\label{Fig_cp_at_infty_pic_3} 
\centering
\end{figure}

Let us discuss the terminology and the practical impact on how to proceed. 
Consider in analogy to the negative gradient flow an 
\textit{energy decreasing} deformation $F $ generated by 
$\partial_{t}u= G(u)$ with
\begin{equation}\label{general_form_of_a_flow}
\begin{split}
G\in C^{0,1}_{loc}(X,TX)
\; \text{ such, that }\;
\partial J(u) G(u) \longrightarrow  0
 \Longrightarrow 
\vert \partial J(u) \vert \longrightarrow 0   
\end{split}
\end{equation}
Independently of a particular choice of $ G $  we then see, that
\begin{enumerate}
\item[(1)]
for every flow line due to  energy consumption necessarily  
\begin{equation}\label{gradient_to_zero}
\vert \partial J(u) \vert \longrightarrow  0
\end{equation}
up to a subsequence in time. Hence and generally one and only one of the three possibilities 
\begin{enumerate}[label=(\roman*)]
\item \quad $\Vert u \Vert\longrightarrow \infty$
\item \quad $ u \xrightharpoondown{\quad} u_{\infty} $ weakly, but 
$ \Vert u-u_{\infty}\Vert \hspace{4pt}\not \hspace{-4pt}\longrightarrow  0$ 
\item \quad
$u\longrightarrow u_{\infty}\in  \{ \partial J=0 \}$ strongly.  
\end{enumerate}
occurs up to a subsequence in time for every flow line. 
In cases (i) or (ii) we say, 
that $ u $ tends to leave the variational space 
or escapes to \textit{infinity}, see Figure \ref{Fig_cp_at_infty_pic_3} for an illustration.
While (i) may occur in case of the two dimensional analogon to the prescribed scalar curvature problem, i.e. the Gaussian one, in our setting  (i) is ruled out, 
as $ X\subset \mathcal{A} $ is bounded. 
However, since energy is decreased, and by virtue of \eqref{gradient_to_zero} every flow line constitutes up to a subsequence in time a Palais-Smale sequence.

\item[(2)] 
on the other hand an analysis of arbitrary Palais-Smale sequences shows, that 
\begin{equation*}
\begin{split}
u
= &
u_{\infty}+\sum_{i=1}^{q}\alpha_{i}\varphi_{a_{i},\lambda_{i}}+v
\; \text{ for some  } \; q\in \N 
\end{split}
\end{equation*}
up to a subsequence in time  eventually, where among other properties
\begin{enumerate}
\item[($0$)] \quad  either $ 0<u_{\infty}\in \{ \partial J=0 \}$ is a solution or $ u_{\infty}=0 $ and $ q\geq 1 $ 
\item[($v$)] \quad $ \Vert v \Vert \longrightarrow 0$ is a vanishing perturbation
\item[($\lambda$)] \quad $ \; \forall_{i} \; \lambda_{i}\longrightarrow \infty $ and $ \varphi_{a_{i},\lambda_{i}} \xrightharpoondown{\lambda_{i}\to \infty} 0 $ weakly.
\end{enumerate}
Note, that in case and only in case $ u_{\infty}=0 $ we have
\begin{equation*}
\begin{split}
u \xrightharpoondown{\quad} 0 
\; \text{ weakly, while} \;
\Vert \alpha^{i}\varphi_{a_{i},\lambda_{i}} \Vert \longrightarrow 0
\; \text{ is impossible, } \;
\end{split}
\end{equation*}
since $ \Vert \cdot \Vert=1 $ on $ X $ by definition. In other words $ u $ is of \textit{zero weak limit }  type and $ u $ tends to leave the variational space $ X $, i.e. escapes to \textit{infinity}, via  
\begin{equation}\label{endconfiguration}
\begin{split} 
\Vert 
u -\alpha^{i}\varphi_{a_{i},\lambda_{i}}\Vert
\longrightarrow 0
\; \text{ with } \;
\alpha_{i}\longrightarrow \alpha_{i,\infty} \in \R_{+},
\;
a_{i}\longrightarrow a_{i,\infty}\in M
\; \text{ and } \;
\lambda_{i}\longrightarrow \infty 
\end{split}
\end{equation} 
up to a subsequence in time at least.

\item[(3)]
based on an energy consumption argument relying on lower bound estimates on $ \vert \partial J \vert $,  if 
\begin{equation*}
\begin{split}
u
= &
\alpha^{i}\varphi_{a_{i},\lambda_{i}}+v 
\; \text{ with } \; c<\alpha_{i}<C,\; \lambda_{i}\longrightarrow \infty
\; \text{ and  } \; \Vert v \Vert\longrightarrow 0
\end{split}
\end{equation*}
up to a subsequence in time, then also eventually, i.e. for $ t\longrightarrow \infty $, and 
\begin{equation*}
\begin{split}
\vert \partial J(u)\vert
\longrightarrow 0
\; \text{ as } \;
t\longrightarrow \infty.
\end{split}
\end{equation*}
This is to say, that $ u $ does not only escape to \textit{infinity}, but also becomes \textit{critical at infinity}. A priori however this does not imply a unique limiting profile of type \eqref{endconfiguration} for $ t\longrightarrow \infty $.  
\end{enumerate}

\begin{remark}\label{rem_on_MM_5}
In any case different flows may produce along their respective flow lines non coinciding sets of end configurations as in \eqref{endconfiguration}. For instance, as a pathological example  in \cite{MM5} shows, while the strong gradient flow exhibits \textit{zero weak limit} non compact flow lines escaping to \textit{infinity}, all of which have one and the same end configuration for $ t\longrightarrow \infty $, a slight variation of this flow is compact, i.e. is convergent or in other words does not have any flow lines escaping to \textit{infinity} at all. 
\end{remark}  

In order to avoid such issues we will  
\begin{enumerate}[label=(\roman*)]
\item define an \textit{energy decreasing} flow $ \Phi $ as in \eqref{general_form_of_a_flow}.
\item prove beyond (3) above, that every \textit{zero weak limit} flow $ u $ line of $ \Phi $ escaping to \textit{infinity}, i.e.
\begin{equation*}
\begin{split}
\partial_{t}u=A(u)
\; \text{ with } \;
u=\alpha^{i}\varphi_{a_{i},\lambda_{i}}+v 
\; \text{ eventually, } \;
\end{split}
\end{equation*}
has a unique end configuration, informally
\begin{equation*}
\begin{split}
u\longrightarrow c=\alpha_{\infty}^{i}\varphi_{a_{i,\infty},\infty}
\; \text{ as } \;
t\longrightarrow \infty
\end{split}
\end{equation*}
and $\varphi_{a_{i,\infty},\infty}=\delta_{i}$ is a Dirac measure for $L^{\frac{2n}{n+2}}$.
\item classify all these attained end configurations as $ a_{i,\infty}=x_{i} $ 
for some 
\begin{equation*}
\begin{split}
x_{i}
\in 
\{ \nabla K=0 \} \cap \{ \Delta K<0 \}  
\; \text{ and } \;
x_{i}\neq x_{j}   
\; \text{ for } \; 
i\neq j,
\end{split}
\end{equation*}
\quad cf. \eqref{principal_a_i_behaviour}, while $ \alpha_{i,\infty}=\frac{\Theta}{K(x_{i})^{\frac{n-2}{4}}} $ 
with a normalizing constant $ \Theta>0 $, cf. \eqref{Theta}.
\item determine homologically for a contractible \textit{neighbourhood} 
\begin{equation*}
\begin{split}
V_{\varepsilon}(c)
= &
V(q,\varepsilon)\cap \{ \vert a_{i}-x_{i} \vert \leq \varepsilon\} 
\end{split}
\end{equation*}
of $ c $, cf. Definition \ref{V_p_e} and \eqref{eq:eij}, the change of topology as
\begin{equation*}
\begin{split}
H_{k}
(
\{ J\leq J_{\infty} -\delta \} \cup V_{\varepsilon}(c)
,
\{ J\leq J_{\infty} -\delta \}
)
=
\delta_{k,m}
\; \text{ for } \;
\delta \ll \varepsilon \ll 1
\; \text{ and } \;
m=\ind(J,u_{\infty}), 
\end{split}
\end{equation*}
cf. Theorem \ref{thm}, while this neighbourhood does not contain any solution, i.e. 
$$ V_{\varepsilon}(c) \cap \{ \partial J=0 \}=\emptyset .$$ 

\item show, that none of these end configurations can be avoided as an obstacle to energetic deformation, i.e. they are critical points at infinity,  cf. Definition \ref{critical_points_at_infinity}.
\end{enumerate}
 
Whereas (i)-(iii) are seen by analysing one specific flow $ \Phi $, the index in (iv) is justified from the subcritical approximation, while the minimality condition (v) follows from a Morse lemma at infinity, i.e. a faithful expansion of
$ J $ on $ V_{\varepsilon}(c) $ of Morse type. 
And this is, how to prove Theorem \ref{thm}.

\

Theorem \ref{thm} as a result,  in particular and foremost the exclusion of tower bubbles along a suitable flow,  is  not new,  we refer to  Appendix 2 in \cite{bab1} for the case of the sphere. While in the latter work the most  important arguments are nicely displayed, there is an inaccuracy, which we shall discuss after the proof of Theorem \ref{thm} at the end of this work, whose motivation besides is threefold
\begin{enumerate}[label=(\roman*)]   
 \item[(1)] the discourse fits well into the language and notation of 
 \cite{MM1},\cite{MM2},\cite{MM4},\cite{may-cv},\cite{MM6},\cite{MM5},\cite{may_ndiaye_riemmanian_mapping} and the result demonstrates a natural equivalence of subcritical approximation versus critical points at infinity of energy decreasing type.
 \item[(2)] the flow, we study,  is in contrast to previous explicit constructions, cf. \cite{bab},\cite{bab1},\cite{[BCCH4]},\cite{[BCCHhigh]},\cite{BenAyed_Ahmeou}, norm and positivity preserving, hence provides a natural deformation of energy sublevels as subsets of the variational space $X$ for the variational functional $J$ on $X$. 
Conversely these properties hold true for Yamabe type, i.e. weak  $L^{2}$-pseudo gradient flows, cf. \cite{Brendle},\cite{may-cv},\cite{Struwe_Malchiodi_Q_curvaure}, whose analysis relies on higher $L^{p}$ curvature norm controls, hence are not easy to adapt at \textit{infinity} to exclude  tower bubbles. 
 \item[(3)] the construction of the flow as in  Section \ref{section_flow_construction} is explicit and keeps track of all the relevant quantities. In particular we move the blow-up points $a_{i}$ exactly along the stable  manifolds of $K$, which will prove helpful for adaptations  to describe the flow outside $V(q,\varepsilon)$, but still in a concentrated regime.  
\end{enumerate}         

\section{Critical points at infinity}\label{section_critical_points_at_infinity}

While (i)-(v) above identify the critical points at infinity according to \cite{bab_calc_var}, their definition
as in \cite{bab_calc_var} is related to a pseudo gradient or more generally to a flow of type \eqref{general_form_of_a_flow}. And therefore \textit{this} notion of a critical point at infinity is not intrinsic to the variational problem. On the other hand in some situations, cf. \cite{MM5} and Remark \ref{rem_on_MM_5}, it is counter intuitive to identify a critical point at infinity with a non compact flow line of a specific flow, if any non compactness can be avoided by considering a different flow. 
      
\

We wish to take a different view. Let us first define various objects related to Palais-Smale sequences. 
Strictly speaking Proposition \ref{blow_up_analysis} and Remark \ref{remark_blow_up_analysis} describe the possible Palais-Smale end configurations as elements of \eqref{PS_set} below, but in fact  each such configuration can be easily obtained as a natural limit of a Palais-Smale sequences.
\begin{definition}
Let $\overline{(PS)}$ denote the blow-up profiles arising from Palais-Smale sequences on $X$, i.e.
\begin{equation}
\begin{split}
\overline{(PS)}\label{PS_set}
= 
\{ 
 u_{\infty} + \sum_{i} \alpha_{i_{\infty}}\delta_{a_{i_{\infty}}}
\; : \; &
a_{i_{\infty}}\in M  \; \text{ and for some } \;  \kappa_{\infty}>0 
\\ & \; L_{g_{0}}u_{\infty}=\frac{K}{\kappa_{\infty}}u_{\infty}^{\frac{n+2}{n-2}} ,\;
\alpha_{i_{\infty}}=\sqrt[\frac{4}{n-2}\;\;\,]{\frac{4n(n-1)\kappa_{\infty}}{K(a_{i_{\infty}})}}
\;\}.
\end{split}
\end{equation}
We also  
\begin{enumerate}[label=(\roman*)]
\item 	
denote for $A\subseteq X$ by
\begin{equation*} 
\begin{split}
\overline{(PS)}^{A}
= 
\{ 
c\in \overline{(PS)} \; \mid \; \; \exists \; (u_{n})\subseteq A 
\;\text{ Palais-Smale with } \;
c \;\text{ as limiting profile } \;
\} 
\end{split}
\end{equation*}
the blow-up profiles arising from  Palais-Smale sequences in $A\subseteq X$. 
\item 
denote for $  c\in \overline{(PS)} $ with corresponding $\kappa_{\infty}>0$, cf. \eqref{PS_set}, by 
\begin{equation}\label{limiting_energy}
\begin{split}
J(c)=\lim_{k \to \infty } J(u_{k})=\kappa_{\infty}^{\frac{2-n}{n}}
\end{split}
\end{equation}
the unique limiting energy 
 of any Palais-Smale sequence $ (u_{k})\subset X $ with $ c $ as limiting profile. 
\item
call for $c\in \overline{(PS)}$ an open set $U=U_{c}\subseteq X$ an \textbf{open neighbourhood} of $c$, if
\begin{equation*}
\begin{split}
\; \forall \; (u_{k})\subset X \; \text{ Palais-Smale with limiting profile } \;  c
\; : \; 
u_{k} \in U \; \text{ eventually }  \; \wedge \; d(u_{k},\partial U)\neq o(1),
\end{split}
\end{equation*}
and call an open set $V=V_{A}\subset X$ an \textbf{open neighbourhood} of $A\subseteq \overline{(PS)}$, if
\begin{equation*}
\begin{split}
\; \forall \; c\in A\; : \; V_{A}\; \text{ is a neighbourhood of } \; c.
\end{split}
\end{equation*}
\item
call $O\subseteq \overline{(PS)}$ \textbf{open}, if 
\begin{equation*}
\begin{split}
\; \forall \; c\in O \; \exists \; U=U_{c} \; \text{ a neighbourhood of } \; c
\; : \; 
\overline{(PS)}^{U_{c}} \subseteq O,
\end{split}
\end{equation*}
and $C\subseteq \overline{(PS)}$ \textbf{closed}, if 
$\overline{(PS)}\setminus C$ is \textbf{open}. 
\end{enumerate}

\end{definition}

\begin{remark}
We remark, that
\begin{enumerate}[label=(\roman*)]
\item
as a fundamental property 
\begin{equation}\label{sequential_compactness} 
\begin{split}
\; \forall \; L>0
\; : \; 
\{ c\in \overline{(PS)}\; : \; J(c)\leq L \} 
\; \text{ is sequentially compact, } \;
\end{split}
\end{equation}
cf. \eqref{limiting_energy}, i.e. for any $(c_{k})\subset \overline{(PS)}$ with $J(c_{k})\leq L$ and up to a subsequence
\begin{equation*}
\begin{split}
c_{k}=u_{\infty,k}+\sum^{q_{k}}_{i=1}\alpha_{i_{\infty},k}\delta_{a_{i_{\infty},k}}
\xlongrightarrow{k\to \infty}
c_{\infty}=\tilde u_{\infty}+ \sum^{q_{\infty}}_{i=0}\alpha_{i_{\infty}}\delta_{a_{i,\infty}}
\end{split}
\end{equation*}
in the sense of distributions. In fact by \eqref{limiting_energy_explicit} the number of diracs $q_{k}$ is bounded. And either 
$u_{\infty,k}=0$ up to a subsequence or 
by \eqref{limiting_energy_explicit} and \eqref{normalisation_to_restriction} the sequence
$(u_{\infty,k})$ of solutions constitutes a Palais-Smale sequence of bounded norm and energy, whence 
Proposition \eqref{blow_up_analysis} is applicable. 

\item 	finite intersections and arbitrary unions of open subsets of $\overline{(PS)}$ are again open.

\item for
$c=u_{\infty}+\alpha^{i}_{\infty}\delta_{a_{i_{\infty}}}\in \overline{(PS)}$ and $\varepsilon>0$
\begin{equation*}
\begin{split}
U=U_{c}^{\varepsilon}
= &
\{ 
\frac
{u_{\infty}+\alpha^{i}\varphi_{a_{i},\lambda_{i}}+v }
{\Vert u_{\infty}+\alpha^{i}\varphi_{a_{i},\lambda_{i}}+v  \Vert}
\; : \; \mid
\sum_{i,j}\varepsilon_{i,j}+\vert \alpha_{i}-\alpha_{i_{\infty}}\vert + \frac{1}{\lambda_{i}} + d(a_{i},a_{i_{\infty}})
+
\Vert v \Vert <\varepsilon 
\} 
\end{split}
\end{equation*}
is a natural neighbourhood of $c$ in $X$.
\end{enumerate}
\end{remark}

\

\noindent
In this way the Palais-Smale closure of $X$, i.e.
$
E
= 
X \cup \overline{(PS)}
$
with the topology
\begin{equation*}
\begin{split}
\text{Top}(E)
= &
\{ V\cup \; ( \;\overline{(PS)}^{V} \setminus \overline{(PS)}^{\partial V} \;) 
\; : \; V\subseteq X \; \text{ open }\;\} 
\end{split}
\end{equation*}
becomes a separable Hausdorff space and we identify a neighbourhood 
$O=O_{C}\in Top(E)$ of $C\subseteq E$ with its part 
$U=O\cap X$ in the variational space $X$, cf. (iii) in Definition \ref{PS_set}. 

\begin{definition}
Let
$$ 
\Pi F 
=
\{ \Psi= \Pi^{k}_{i=1}F_{i}\circ \tau_{i} 
\; : \; 
F_{i},\tau_{i} 
\; \text{ as in (1)\,-\,(4) below }   \; \}$$  
denote the set of  consecutive, energy decreasing deformations $\Psi$, for which
\begin{enumerate}[label=(\roman*)]
\item[(1)] $\forall\;1\leq i \leq k\;:\;F_{i}$ as in \eqref{general_form_of_a_flow}	
\item[(2)] $ \forall\; 1\leq i \leq k-1\;:\;\tau_{i}:X\longrightarrow \R_{\geq 0} $ Lipschitz
\item[(3)] during $[0,k-1]$ we deform along 
$$(\Psi\lfloor_{X\times[i-1,i]})(v_{0},t)=F_{i}(v_{0},\tau_{i}(x)\cdot (t-i+1)),$$
i.e. solving for $1\leq i \leq k-1$ consecutively the initial value problems 
$$
\left\{
\begin{matrix*}[l]
v(t_{0})=v_{0}  &  \;\text{ for }\;  & t_{0}=i-1   \\
\partial_{t}v=\tau_{i}(v_{0})\cdot G_{i}(v) &  \;\text{ for }\;  & i-1\leq t \leq i  \\
\end{matrix*}
\right.
$$
\item[(4)] and finally 
$$ (\Psi\lfloor_{X\times [k-1,\infty)})(v_{0},t) = F(v_{0}, t-k+1),$$
i.e. solving the initial value problems 
$$
\left\{
\begin{matrix*}[l]
v(t_{0})=v_{0}  &  \;\text{ for }\;  & t_{0}=k-1   \\
\partial_{t}v= G_{i}(v) &  \;\text{ for }\;  & t\geq k-1  \\
\end{matrix*}
\right. .
$$ 
\end{enumerate}
\end{definition}	
Note, that every
$\Psi(\cdot,t)=\Pi_{i=1}^{k}F_{i}(\cdot,\tau_{i}(\cdot)t)\in \Pi F$ 
acts as a family of diffeomorphisms  
and along each flow line
$u=\Psi(u_{0},\cdot)$
there holds almost always
$\partial_{t}J(u)\leq 0$ and, cf. \eqref{general_form_of_a_flow},
$$
\partial_{t}J(u)\longrightarrow 0 \Longleftrightarrow \vert \partial J(u) \vert \longrightarrow 0.
$$
Since ultimately critical points at infinity will be related to an obstacle to energetic deformation below a certain energy $\sigma$, we introduce the subsequent notions.

\begin{definition}\label{reducible}
For $\sigma \in \R $ we call a closed subset $W\subseteq X$   $\sigma$-reducible, if
\begin{equation*}
\; \forall \; \varepsilon>0 \; \exists \; \Psi \in \Pi F  \; \wedge \; T \geq 0
\; \forall \; w\in W
\; : \; 
J(\Psi(w,T))\leq \sigma + \varepsilon.
\end{equation*}
\end{definition}
Clearly every closed subset of  a $\sigma$-reducible set is $\sigma$-reducible, $\{ J\leq \sigma \} $ is $\sigma$-reducible and, if 
$W$ is $\sigma_{1}$-reducible and $\sigma_{1}\leq \sigma_{2}$, then $W$ is $\sigma_{2}$-reducible as well. 

\begin{definition}\label{capturing}
For $\sigma \in \R $ we call a closed subset $C\subseteq E_{\sigma}=E\cap \{ J=\sigma \}$ $\sigma$-capturing, if
\begin{equation*}
\begin{split}
\; \forall \; U=U_{C} \; \exists \; \varepsilon > 0
\; \forall \; W \; \text{ $\sigma$-reducible } \; 
\; \exists \; \Psi \in \Pi F  \; \wedge \;  T\geq 0
\; \forall \; t\geq 0
\; : \; 
\Psi(W,T+t)\subset \{ J<\sigma-\varepsilon  \} \cup U.
\end{split}
\end{equation*}
\end{definition}	
To clarify this definition 
\begin{enumerate}[label=(\roman*)]
\item consider the case, that $C=\emptyset$ is $\sigma$-capturing. Then clearly 
$E_{\sigma}=\{ J=\sigma \}\cap E$ as an energy level 
is not an obstacle to energetic deformation. 
\item consider with $D=[-1,1]$ the stretched maximum
\begin{equation*}
J
:
\R\longrightarrow \R
:
x\longrightarrow 
\left\{
\begin{matrix*}[c] 
-(x+1)^{2} &  \;\text{for}\;  &   x \leq -1\\
0 &  \;\text{for}\;  &  -1\leq x \leq 1 \\
-(x-1)^{2} &  \;\text{for}\;  &  x\geq 1  \\
\end{matrix*}
\right..
\end{equation*}
In this case $C=D$ is $\sigma$-capturing, while no subset of $D$ is $\sigma$-capturing. In fact
suppose, that  some $C \subsetneq D$ was $\sigma$-capturing. Then there exists $U=U_{C}$ and
\begin{enumerate}
\item[(1)] 	
$d\in D\setminus C, \epsilon>0$
such, that 
$
B_{2\epsilon}(d)\cap U_{C}=\emptyset
$, since $C$ is closed.
\item[(2)]
some $0<\varepsilon <0$ such, that for every $\sigma$-reducible $W$ we find $\Psi \in \Pi F, T\geq 0$ such, that
$$
\Psi(W,T)\subset \{ J<\sigma-\varepsilon  \} \cup U.
$$
\end{enumerate}
Combining then $\Psi$ with a flow $\Phi \in \Pi F$, along which for $t>T$
$$
\partial_{t}u
=
 G(u)
=
\left\{
\begin{matrix*}[l]
1  &  \;\text{for}\;  &  x\geq d+2\epsilon \\
0 &  \;\text{for}\;  &  d-\epsilon \leq x \leq d+
\epsilon \\
-1 & \;\text{for}\; & x\leq d-2\epsilon 
\end{matrix*}
\right.,
$$ 
to a flow $\Theta$, we then find
$
\Theta(W,T+3)\subset \{ J<\varepsilon  \},
$
while this is readily impossible, when choosing 
$$
W=\{J\leq \sigma\} =\R.
$$

\item consider the function as in Figure \ref{fig:some_function} and observe, that 
\begin{figure}[h]
\centering
\includegraphics[scale=0.8]{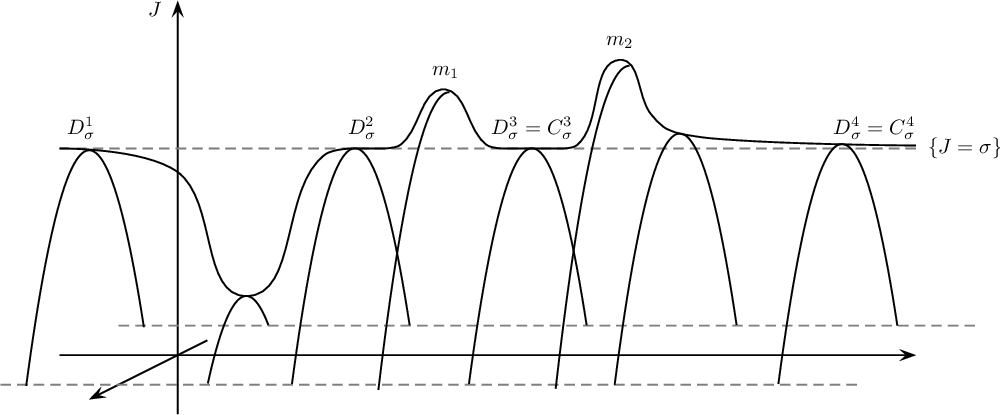}
\caption{Some function on $\R^{2}$}
  \label{fig:some_function}
\end{figure}

\begin{enumerate}[label=(\roman*)]
\item[(1)] every $\sigma$-reducible $W$, which by definition is closed, is a subset of some complement  
$$
X\setminus (B_{\epsilon }(m_{1})\cup B_{\epsilon }(m_{2})),\; X=\R^{2}
$$

\item[(2)] evidently 
$D_{\sigma}=\cup_{i=1}^{4}D^{i}_{\sigma}=\overline{(PS)}_{\sigma}$
and  for every $U=U_{D}$ and $\varepsilon >0$ sufficiently small we may deform
every $\sigma$-reducible $W$ along some $\Psi\in \Pi F$ onto
$
\{ J<\sigma-\varepsilon  \}\cup U  
$,
i.e. $D_{\sigma}$ is $\sigma$-capturing.
\item[(3)] we may deform along some $\Psi \in \Pi F$ suitably small neighbourhoods
$
U_{D^{i}_{\sigma}}
$
of
$
D^{i}_{\sigma} 
$
in such a way, that $\Psi$ leaves  
$U_{D^{i}_{\sigma}}\cup \{ J<\sigma-\varepsilon  \} $ invariant, while for some $T\geq 0$
$$
\Psi(U_{D^{1}_{\sigma}},T),\Psi(U_{D^{2}_{\sigma}},T)\subseteq \{ J<\sigma-\varepsilon  \} 
$$
\item[(4)] as a consequence also
$C=D^{3}_{\sigma}\cup D^{4}_{\sigma}$ is 
$\sigma$-capturing.

\item[(5)] in fact $C$ is minimal in the sense, that $C$ is $\sigma$-capturing and for every $\sigma$-capturing $D\subseteq E_{\sigma}$ also
$C\cap D$ is $\sigma$-capturing, cf. Definition \ref{strongly_critical}.
\end{enumerate}
We remark, that any flow, by which we push a neighbourhood of $D_{\sigma}^{1}$ below a certain energy $\sigma-\varepsilon $, requires diverging speed
towards $D_{\sigma}^{1}=-\infty$, cf. (i) of Remark \ref{rem_infinite_index}. 
A possible choice is $\partial_{t}x=x^{2}$ near 
$D_{\sigma}^{1}$
\end{enumerate}
While the set $D$ in the aforegoing examples is naturally critical, the set $C$ is the one of variational interest, i.e. the obstacle to energetic deformation, and correctly identified as the unique and minimal strongly critical set as defined below.
\begin{definition}\label{strongly_critical}
We call $C\subseteq E_{\sigma}$ strongly critical, if 
$C$ is $\sigma$-capturing and 
\begin{equation*}
\begin{split}
\; \forall \; D\subseteq E_{\sigma} \; \text{$\sigma$-capturing} 
\; : \; 
C\cap D \; \text{ $\sigma$-capturing.} \;
\end{split}
\end{equation*}
\end{definition}	
Note, that this definition does not exclude the case, that $C=\emptyset$ is strongly critical. But if so, the situation is variationally trivial. 
\begin{proposition}\label{PS_strongly_critical}
$\overline{(PS)}_{\sigma}=\overline{(PS)}\cap \{ J=\sigma \}\subseteq E_{\sigma}$
is strongly $\sigma$-critical. 
\end{proposition}	
\begin{proof}
Let $C=\overline{(PS)}_{\sigma}$ and 
$U=U_{C}$ arbitrary. Then using \eqref{sequential_compactness}
\begin{equation*}
\vert \partial J \vert >\gamma 
\; \text{ on } \; 	
\{ \sigma-\epsilon< J < \sigma+\epsilon \}\setminus U
\; \text{ for some } \; \epsilon,\gamma>0 
\end{equation*}
and we consider for some $\delta>0$ open subneighbourhoods of $C$ satisfying
\begin{equation*}
U= V_{3} \supset V_{2} \supset V_{1}
\; \text{ with } \;
d(V_{i},X\setminus V_{i+1})>\delta. 
\end{equation*}
Let $\varepsilon >0$ such, that
$
\gamma \delta > 4\varepsilon.
$
Note, that to travel a distance $\delta>0$ along a negative gradient flow line  
$$
u
:
[0,t]
\longrightarrow 
\{ \sigma-\delta < J < \sigma+\epsilon \}\setminus U
$$
comes at an energetic cost $\Delta J\geq\gamma \delta$, since
\begin{equation*}
\delta
=
d(u(t),u(0))
\leq
\int^{t}_{0} \vert \partial J(u(\tau)) \vert  \,d\tau
\leq
\gamma^{-1}
\int^{t}_{0} \vert \partial J(u(\tau)) \vert^{2}  \,d\tau
=
\frac{J(u(0))-J(u(t))}{\gamma}.
\end{equation*}
We therefore consider some arbitrary $\sigma$-reducible $W$
and choose $\Psi \in \Pi F$ such, that
$$\sup_{w\in W}J(\Psi(w,T_{0}))<\sigma +\varepsilon$$
and $\Phi \in \Pi F$ given by $\Psi$ during $[0,T_{0}]$
and the negative gradient flow for $t>T_{0}$. We then show
$$
\; \exists \; T\geq T_{0} \; \forall \; t\geq T\; : \; 
\Phi(W,T+t) \subset \{ J<\sigma-\varepsilon \}\cup U 
$$
in order to verify, that $C$ is $\sigma$-capturing. Hence consider some
$$
u_{0}\in \Psi(W,T_{0})=\Phi(W,T_{0})
\subset \{ J<\sigma+\varepsilon \} 
$$
as an initial data for the negative gradient flow line $u$. Then
\begin{enumerate}[label=(\roman*)]
\item in case $u_{0}\in X\setminus V_3$, the flow line $u$ can never reach 
$V_{2}\cap \{ J\geq \sigma-\varepsilon \} $, 
since otherwise $u$ would have to travel through $V_3$ bridging a distance $\delta>0$, which comes at an energetic cost $\gamma \delta > 4\varepsilon$, while we have only an energetic gap of $2\varepsilon$ at disposition.  As a consequence $u$ enters $\{ J<\sigma-\varepsilon \} $ and does so in some finite time $T_1\geq T_{0}$, which is uniformly upper bounded for all $u_{0}\in X\setminus V_3$.

\item in case $u_{0}\in V_{3}\setminus V_2$ and by the same argument as above, the flow line can never reach $V_1 \cap \{ J\geq \sigma -\varepsilon \} $ and thus enters $\{ J<\sigma-\varepsilon  \} $ in some uniformly upper bounded time $T_{2}$. 
 
\item in case $u_{0}\in V_{2}$, then the flow line $u$ can never leave $V_{3}\cup \{ J<\sigma-\varepsilon  \} $ again by energy consumption.
\end{enumerate}
As a consequence for $T=\max\{ T_{1},T_{2} \}\geq T_{0} $ we find
$$
\; \forall \; t\geq T\; : \; 
\Phi(W,T+t)\subset \{ J<\sigma+\varepsilon  \} \cup V_{3}.
$$
Recalling $U=V_{3}$, this shows, that $C=\overline{(PS)}_{\sigma}$
is $\sigma$-capturing. To prove, that $C$ is even strongly 
$\sigma$-critical, we consider some arbitrary $\sigma$-capturing $D$ and show, that $D\cap C$ is also $\sigma$-capturing. Again this follows from energy consumption flowing by the negative gradient flow away from $C=\overline{(PS)}_{\sigma}$.
\end{proof}
\begin{proposition}\label{minimality}
There exists a minimal, strongly critical $M\subseteq \overline{(PS)}_{\sigma}$.
\end{proposition}	
\begin{proof}
We may assume, that $\emptyset \subseteq E_{\sigma}$ is not strongly critical. In particular and necessarily 
$\overline{(PS)}_{\sigma}\neq \emptyset,$
 since otherwise for some $\gamma,\epsilon>0$
by \eqref{sequential_compactness}
$$\vert \partial J \vert>\gamma 
\; \text{ on } \; 
\{ \sigma-\epsilon < J < \sigma + \epsilon\}
$$  
and this implies, that every $\sigma$-reducible $W$ can be brought down into $\{ J<\sigma-\varepsilon \}$ for any $0<2\varepsilon<\epsilon$ in finite time along the negative gradient flow.
Hence $\overline{(PS)}_{\sigma}\neq \emptyset$ and by virtue of
Proposition \ref{PS_strongly_critical} 
we may consider 
\begin{equation*}
P=\{ D\subseteq \overline{(PS)}_{\sigma}\; : \; D \; \text{ strongly critical } \;  \} 
\neq \emptyset
\end{equation*}
as a by inclusion partially ordered set. Let us denote by 
$
C_{0}\supseteq C_{1} \supseteq C_{2} \supseteq \ldots 
$
an arbitrary chain in $P$. 
Then the assertion follows from Zorn's Lemma, provided
$$C=\cap_{i}C_{i}\subseteq \overline{(PS)}_{\sigma}$$ 
is strongly critical, i.e. a lower bound for this chain in $P$.
To see the latter we have to show, that 
\begin{enumerate}[label=(\roman*)]
\item 	$C$ is $\sigma$-capturing and
\item   $C\cap D$ is $\sigma$-capturing, whenever some $D\subseteq E_{\sigma}$ is $\sigma$-capturing.
\end{enumerate}

To prove (i) consider an arbitrary $U=U_{C}$.  
Then, as $C_{k}\subseteq \overline{(PS)}_{\sigma}$ is closed and $\overline{(PS)}_{\sigma}$ is sequentially compact, cf. \eqref{sequential_compactness}, there exists $k\in \N$ such, that
$U=U_{C_{k}}$ is a neighbourhood of $C_{k}$. Moreover, since $C_{k}$
is strongly critical, $C_{k}$ is in particular $\sigma$-capturing. Hence according to Definition \ref{capturing}
we find $\varepsilon >0$, such that we may capture every $\sigma$-reducible $W$ in $\{ J<\sigma-\varepsilon  \} \cup U$ by some $\Psi \in \Pi F$ as desired. Therefore and, since $U=U_{C}$ is arbitrary, $C$ is $\sigma$-capturing itself. 

\

To prove (ii) consider an arbitrary $\sigma$-capturing $D\subseteq E_{\sigma}$. Since $C_{k}$ is strongly critical, by definition $D\cap C_{k}\subseteq \overline{(PS)}_{\sigma}$ is $\sigma$-capturing for every $k\in \N$. Arguing as for (i) we then find, that 
$D\cap C$ is $\sigma$-capturing as well. 
\end{proof}
While Zorn's Lemma, as we have seen, guarantees the existence of minimal strongly $\sigma$-critical sets, we have to show uniqueness of the latter separately.
\begin{lemma}\label{uniqueness}
There exists a unique, minimal strongly critical 
$M_{\sigma}\subseteq \overline{(PS)}_{\sigma}$.
\end{lemma}	
\begin{proof} 
By Proposition \ref{minimality} there exists some minimal, strongly critical 
$M_{1}\subseteq \overline{(PS)}_{\sigma}$. 
Suppose, there exists another minimal strongly critical $M_{1}\neq M_{2} \subseteq \overline{(PS)}_{\sigma}$. 

Then, since $M_{2}$ is strongly $\sigma$-critical, $M_{2}$ is by definition $\sigma$-capturing. And, since $M_{1}$ is strongly critical, we deduce, that $M_{1}\cap M_{2}$ is $\sigma$-capturing as well. 

Moreover consider some arbitrary $\sigma$-capturing $D\subset E_{\sigma}$. Then, since $M_{2}$ is strongly critical, also $D\cap M_{2}$ is $\sigma$-capturing. And, since $M_{1}$ is strongly critical, also $D\cap M_{1} \cap M_{2}$ is $\sigma$-capturing. 

We conclude, that $M_{1}\cap M_{2}$ is strongly critical, which contradicts the minimality of $M_{1}$ and $M_{2}$.
\end{proof}

With Lemma \ref{uniqueness} at hand we then define critical points at infinity as follows.
\begin{definition}\label{critical_points_at_infinity}
We call $c\in E$ a critical point at infinity, if 
$c\in \cup_{\sigma}M_{\sigma}\setminus X$.
\end{definition}	
Note, that the definition of
$M_{\sigma}\subseteq \overline{(PS)}_{\sigma}\subseteq E_{\sigma}$ 
does \textit{only} depend on $J$ and the space $\Pi F$ of admissible deformations, in particular does not depend on a specific flow $\Phi \in \Pi F$ or for instance a presumed Morse structure around elements of 
$\overline{(PS)}_{\sigma}$. 

\

Let us show, that the unique, minimal strongly critical sets $M_{\sigma}$ are generically meaningful. 
\begin{proposition}\label{non_degenerate_implies_strongly_critical_1}
Let $c\in \{ J=\sigma \}$ be a non degenerate critical point of finite index. Then $c\in M_{\sigma}$.
\end{proposition} 
\begin{proof}
Arguing by contradiction, we suppose
$c\not \in M_{\sigma}$. Then
$$
\; \exists \; 
U_{M_{\sigma}}
 \; \wedge \; 
U_{c}
 \; : \; 
U_{M_{\sigma}}
\cap
U_{c}
= 
\emptyset.
$$
Consider in a Morse chart around 
$c = 0$, e.g.
$$
J(u)=\sigma+\vert x^{+} \vert^{2}-\vert x^{-} \vert^{2} 
\; \text{ on } \;
V=V_{c}=B_{\epsilon}(0)\subset U_{c},
$$
a sequence 
$$
V\cap X^{+}
\supset
(x^{+}_{k})\longrightarrow c
\; \text{ as } \; k\longrightarrow \infty,
$$
to which for arbitrarily small 
$\sqrt{\varepsilon }\ll \delta \ll  \epsilon$ we attach 
$m=ind(J,0)$-dimensional
disks 
$$
B^{-}_{2\delta}(x^{+}_{k}) 
=
\{ 
(x^{+},x^{-})\in V
\; : \; 
x^{+}=x^{+}_{k}  
\; \wedge \; 
\vert x^{-} \vert \leq 2\delta 
\} 
\subset V.
$$
with boundary
$
\partial B^{-}_{2\delta}(x^{+}_{k})
\subset 
\{ J\leq\sigma -\varepsilon  \}
$
Then for degree reasons, see below, 
\begin{equation}\label{stable_manifold_intersection}
\; \forall \; 
\Psi \in \Pi F  \; \wedge \; t\geq 0
\; \exists \; 
x^{+}_{k,t} \in \Psi(B^{-}_{2\delta}(x^{+}_{k}),t)
\; : \; 
x^{+}_{k,t}
\in  
V \cap X^{+} 
 \cap 
 \{\sigma \leq J\leq J(x^{+}_{k}) \}.
\end{equation}
However, 
since 
each $B^{-}_{2\delta}(x^{+}_{k})$ is $\sigma$-reducible and
$M_{\sigma}$ is $\sigma$-capturing, by definition we find
\begin{equation}\label{no_stable_manifold_intersection}
\Psi(B_{2\delta}(x^{+}_{k}),T) \subset \{ J<\sigma -\gamma \} \cup U_{M_{\sigma}}.
\end{equation}
for suitable $T,\gamma>0$ and $\Psi \in \Pi F$.
Then \eqref{stable_manifold_intersection} and
\eqref{no_stable_manifold_intersection}
lead to the obvious contradiction
$$
x^{+}_{k,t} 
\in 
V
\cap
U_{M_{\sigma}}
\subseteq
U_{c}\cap U_{M_{\sigma}}
=\emptyset.
$$
Hence we are left with proving \eqref{stable_manifold_intersection}.
On the Morse chart $V=B_{\epsilon}(0)$
consider the continuous map
\begin{equation*} 
\begin{split}
(x^{+},x^{-}) 
\longrightarrow 
x^{-}
\longrightarrow 
(
\frac{\delta}{\vert x^{-} \vert } \mathbb{1}_
{
\{\vert x^{-} \vert \geq \delta \}
}
+
\mathbb{1}_
{
\{\vert x^{-} \vert <\delta \} 
}
)
x^{-}
\longrightarrow 
\sfrac
{
(
\frac{\delta}{\vert x^{-} \vert }
\mathbb{1}_
{
\{\vert x^{-} \vert \geq \delta \}
}
+
\mathbb{1}_
{
\{\vert x^{-} \vert <\delta \} 
}
)
x^{-}
}
{\sim},
\end{split} 
\end{equation*}
with $ \sim $ denoting the natural identification
of the disk $ B^{m}_{\delta}(0) $ via
$$x\sim y \Longleftrightarrow \Vert x \Vert =\Vert y \Vert=\delta 
\; \text{ for } \; 
x,y\in B^{m}_{\delta}(0)
$$ 
with the sphere $ S^{m}_{\delta}$ with south pole $ S=0 $.
After rescaling we hence obtain a continuous map
\begin{equation*}
\begin{split}
\theta:B_{\epsilon}(0)\longrightarrow & 
S^{m}_{1}.
\end{split}
\end{equation*}
Moreover for 
$ 
x=(x^{+},x^{-})
\in 
\partial B_{\epsilon}(0)
\cap 
\{ J\leq \sigma + \varepsilon \} 
$ 
and recalling 
$ 0<\sqrt{\epsilon}\ll \delta \ll \epsilon $
we have
\begin{equation*}
\vert x^{+} \vert^{2}
+
\vert x^{-} \vert^{2}
=  
\epsilon^{2}
\; \text{ and } \;
\vert x^{+} \vert ^{2}
\leq \vert x^{-} \vert^{2}
+\varepsilon,
\end{equation*}
whence we may assume
\begin{equation*}
\begin{split}
\vert x_{-} \vert^{2} 
\geq &
\frac{\epsilon^{2}-\varepsilon}{2}
>2\delta^{2}.
\end{split}
\end{equation*}
Consequently and with $ N $ denoting the north pole of $ S_{1}^{m} $ 
$$ \theta
(
\partial B_{\epsilon}(0)
\cap 
\{ J\leq \sigma + \varepsilon \} 
)
=
\{ N \},$$ 
whence we may extend continuously and restrict
$
\theta :  
\{ J\leq \sigma +\varepsilon \} 
\longrightarrow 
S_{1}^{m}
$
by putting
\begin{equation}\label{sigma_extension}
\begin{split}
\theta\equiv N \; \text{ on } \;  
\{ J\leq \sigma +\varepsilon \} 
\setminus B_{\epsilon}(0).   
\end{split}
\end{equation}
We also find, that for all $k$ sufficiently large 
\begin{equation}\label{boundary_retracts_to_point}
\begin{split}
\forall\; t\geq 0
\;:\; 
\theta(\partial \Psi(B^{-}_{2\delta}(x_{k}^{+}),t)=\{ N \} 
\end{split}
\end{equation}
for any flow $ \Psi\in \Pi F $. In fact let
$
y_{0}=(y_{0}^{+},y_{0}^{-})
\in
\partial B_{2\delta}^{-}(x_{k}^{+})
\subset B_{\epsilon}(0).
$
Then by construction 
$$
\theta(y_{0})=N
\; \text{ and } \; 
J(y_{0})=\sigma + \vert x_{k}^{+} \vert^{2} - \vert y_{0}^{-} \vert^{2}  
=
\sigma - 4\delta^{2}+o_{\frac{1}{k}}(1).
$$
Let $y=\Psi(y_{0},\cdot)$ and suppose 
$\theta(y(t))\neq N$ for some $t>0$. Then by \eqref{sigma_extension} necessarily
$y(t)\in B_{\epsilon}(0)$ and
$$
\sigma-4\delta^{2}+o_{\frac{1}{k}}(1)
=
J(y_{0})
\geq
J(y)
=
\sigma+\vert y^{+} \vert^{2}-\vert y^{-} \vert^{2}
\geq
\sigma - \vert y^{-} \vert^{2}  
\geq
\sigma - \delta^{2},
$$
leading to a contradiction for $k$ sufficiently large.  
Then  \eqref{boundary_retracts_to_point} implies, that
for the natural embedding 
$$
i
:
B_{1}^{m}(0)
\xlongrightarrow{\simeq} 
B_{2\delta}^{-}(x_{k}^{+})$$
and for every $ t\geq 0 $ the composition
$\theta\circ \Psi(\cdot,t) \circ i$
factorizes  to a map  
$$
\theta \circ \Psi(\cdot,t)\circ \tilde \theta:S_{1}^{m}\longrightarrow S_{1}^{m}.
$$  
Since 
$\theta \circ \tilde \theta \simeq id_{S_{1}^{m}}$
as a homotopy equivalence, 
and by continuity and constancy of the degree on $ S^{m}_{1} $  
\begin{equation*}
\begin{split}
\forall \; t\geq 0 \; \exists\; s\in S^{m}_{1}
\;:\;
\theta(\Psi(\tilde \theta(s),t))=S.
\end{split}
\end{equation*}
Consequently
$  
\;\forall\; t\geq 0
\;\exists\; x\in B_{2\delta}^{-}(x_{k}^{+}) 
\; :\;  
\Psi(x,t)=(x_{k,t}^{+},0).
$
From this \eqref{stable_manifold_intersection} readily follows.
\end{proof}  
Analogous arguments then show, that a finite index Morse structure at infinity leads to the same conclusion. The spaces following are real Banach.
\begin{lemma}\label{non_degenerate_implies_strongly_critical_2}
Let $c\in \overline{(PS)}_{\sigma}$ and suppose, that for a neighbourhood $U_{c}$ of $c$ we may parameterise
\begin{enumerate}[label=(\roman*)]
\item with spaces $X^{\pm}$ and a neighbourhood $V=V_{\mathbb 0}$ of $\mathbb 0\in X^{+}\times X^{-}$
$$
\overline{(PS)}^{U}
\simeq
\{ 
(x^{+},x^{-}) \in X^{+}\times X^{-}
\mid
(x^{+},x^{-})\in V\}.
$$
\item 
$
U_{C}
\simeq
U \subset Y^{+}\times Y^{-}
$
open with spaces 
$Y^{\pm}=\Lambda^{\pm}\times X^{\pm} \times V^{\pm}$ and  
$$
\overline{(PS)}^{U}
\simeq
\{ 
(\lambda^{+},\lambda^{-},x^{+},x^{-},v^{+},v^{-})
\; : \; 
\lambda_{i}^{\pm}=\infty,\;
(x^{+},x^{-})\in V,\;
 v_{i}^{\pm}=0 
 \} 
$$
\item
$
J(u)
=
\sigma
+
\vert \lambda^{+} \vert^{-2}
-
\vert \lambda^{-} \vert^{-2}  
+
\vert x^{+} \vert^{2}
-
\vert x^{-} \vert^{2} 
+
\vert v^{+} \vert^{2}
-
\vert v^{-} \vert^{2}  
$
\end{enumerate}
Then 
\begin{enumerate}[label=(\roman*)]
\item[(1)] 
$c\in M_{\sigma} \Longrightarrow \Lambda^{-}= \mathbb{0}$ 	
\item[(2)] 
$
\Lambda^{-}=\mathbb{0}  \; \wedge \; \dim(Y^{-})<\infty 
\Longrightarrow
c\in M_{\sigma}. 
$
\end{enumerate}
\end{lemma}	
\begin{proof}
As for (1) suppose $\Lambda^{-}\neq \mathbb{0}$. We then decrease energy within $U\simeq U_{c}$ via
\begin{enumerate}[label=(\roman*)]
\item[$(\alpha)$]  	
decreasing 
$\vert x^{+} \vert$ and $\vert v^{+} \vert $ 
until 
$\vert x^{+} \vert^{2}+\vert v^{+} \vert^{2} <\varepsilon  $
to find 
$$
J(u)
<
\sigma
-
\vert \lambda^{-} \vert^{-2}
+
\vert \lambda^{+} \vert^{-2}  
-
\vert x^{-} \vert^{2} 
-
\vert v^{-} \vert^{2} 
+
\varepsilon 
\leq
\sigma
-
\vert \lambda^{-} \vert^{-2}
+
\vert \lambda^{+} \vert^{-2}  
+
\varepsilon 
$$
\item[$(\beta)$] increasing $\vert \lambda^{+} \vert $ until
$\vert \lambda^{+} \vert^{2}<\varepsilon  $ to find
$
J(u)
<
\sigma
-
\vert \lambda^{-} \vert^{-2}
+
2\varepsilon  
$
\item[$(\gamma)$] decreasing $\vert \lambda^{-} \vert $ until 
$\vert \lambda^{-} \vert^{-2}>3\varepsilon $ to find
$
J(u)
<
\sigma
-
\varepsilon.
$
\end{enumerate}
Hence and by minimality of $M_{\sigma}$ necessarily $c\not \in M_{\sigma}$, proving (1).
(2) then follows exactly as Proposition \ref{non_degenerate_implies_strongly_critical_1} upon replacing
$X^{+}$ by $Y^{+}$ and
$X^{-}$ by $Y^{-}=X^{-}\times V^{-}$.
\end{proof}
Let us comment on these Morse structure results.
\begin{remark}\label{rem_infinite_index}
\begin{enumerate}[label=(\roman*)]
\item 	
Suppose, that as in Lemma \ref{non_degenerate_implies_strongly_critical_2} 
on a neighbourhood
$U=U_{c}$
of 
$c\in \overline{(PS)}_{\sigma}\setminus X$  
we have a diffeomorphism
$$
\Phi:U_{c} \longrightarrow  \Phi(U_{c})=U\subset Y^{+}\times Y^{-}
$$
such, that the functional takes form
\begin{equation*}
\begin{split}
J(u)
= \; &
J(\Phi^{-}(y^{+},y^{-})) \\
= \; &
\sigma
+
\vert \lambda^{+} \vert^{-2} 
-
\vert \lambda^{-} \vert^{-2}  
+
\vert x^{+} \vert^{2}
-
\vert x^{-} \vert^{2}
+
\vert v^{+} \vert^{2}
-
\vert v^{-} \vert^{2}
=
F(\lambda^{-},\lambda^{+},x^{+},x^{-}).
\end{split}
\end{equation*}
Since $\lambda^{\pm}_{i}=\infty,\; v^{\pm}_{i}=0$ corresponds to the Palais-Smale limit, clearly 
$$
\vert \partial J \vert(u) 
=
o_{\sum_{i}(\vert \lambda^{-}_{i} \vert^{-1}+\vert \lambda^{+}_{i} \vert^{-1}
+
\vert v^{+}_{i} \vert
+
 \vert v^{-}_{i} \vert
) 
} (1).
$$
But evidently for $\vert x^{+} \vert+\vert x^{-} \vert \neq 0  $
$$
\vert \partial F \vert(\lambda^{-},\lambda^{+},x^{+},x^{-})
\neq 
o_{\sum_{i}(\vert \lambda^{-}_{i} \vert^{-1}+\vert \lambda^{+}_{i} \vert^{-1}
+
\vert v^{+}_{i} \vert
+
 \vert v^{-}_{i} \vert
) 
} (1).
$$
As a consequence $d\Phi$ must be degenerating and, while $J$ has a clear Morse structure at infinity, its derivative will not relate in a trivial way to that of the Morse representation.  For instance consider
\begin{equation*}
\begin{split}
\Phi
:
\R^{n}\times \R_{+}
\longrightarrow 
\Phi(\R^{n}\times \R_{+})
\subset
W^{1,2}(\R^{n})
:
(a,\lambda)
\longrightarrow 
\delta_{a,\lambda}
\end{split}
\end{equation*}
for $n\geq 5$ and the functional 
$J(\delta_{a,\lambda})=\int K\delta_{a,\lambda}^{\frac{2n}{n-2}}$, which expands as
$$
J(\delta_{a,\lambda})
=
c_{1}K(a)+c_{2}\frac{\Delta K(a)}{\lambda^{2}}+o(\frac{1}{\lambda^{2}}).
$$
The tangential space is given by 
$
T_{\delta_{a,\lambda}}\Phi(\R^{n}\times \R_{+})
=
\langle \frac{\nabla_{a}}{\lambda}\delta_{a,\lambda},\lambda\partial_{\lambda}\delta_{a,\lambda}\rangle,
$
whence for the derivative
$$
\vert \partial J(\delta_{a,\lambda}) \vert 
\simeq 
\vert \frac{\nabla K(a)}{\lambda} \vert 
+
\vert \frac{\Delta K(a)}{\lambda^{2}} \vert 
+
o(\frac{1}{\lambda^{3}}).
$$
In particular $(\infty,a)\in \overline{(PS)}$ for every $a\in \R^{n}$. On the other hand $J(\delta_{a,\lambda})$ has under \eqref{nd} readily a Morse structure for $a \in \{ \nabla K=0 \} $ and 
$\vert \nabla_{a}\Phi(a)\vert=\Vert \nabla_{a}\delta_{a,\lambda}\Vert \simeq \lambda$.
\item
The arguments leading to Proposition \ref{non_degenerate_implies_strongly_critical_1} 
and then to Lemma 
\ref{non_degenerate_implies_strongly_critical_2} 
do rely on a Morse structure of the \textit{functional}, but 
not on a corresponding structure of the derivative, cf. (i) above.
\item
Concerning the  infinite index case, 
consider for instance the functional
$$
J(u)=\sigma - \vert y \vert^{2}
\; \text{ on a Hilbertspace  } \; Y
\; \text{ with } \; \dim(Y)=\infty.
$$
Then $ c=0 \in M_{\sigma}$. In fact either
$M_{\sigma}=\emptyset$ or $M_{\sigma}=\{ c \} $, 
since necessarily $M_{\sigma} \subseteq \overline{(PS)}_{\sigma}$. If $M_{\sigma}=\emptyset$, then by definition and with $W=\{ J\leq \sigma \} $ 
$$
\; \exists \; T,\varepsilon >0 \wedge \Psi \in \Pi F
\; : \; 
\Psi(W,T) \subset \{ J<\sigma - \varepsilon  \}. 
$$
In particular $\Psi(0,t)=0$ for all times is impossible, whence there exists $t\geq 0$ such, that
$$
\Psi(0,t)=0 
\; \text{ and } \; 
\partial_{t} \Psi(0,t)=G(\Psi(0,t))=G(0) \neq 0.
$$
Let $y_{0}=-\epsilon G(0)$ and compute
$
y(t)
=
y_{0}+tG(y_{0})+o_{t}(t).
$
We then find
\begin{equation*}
\begin{split}
\partial_{t}J(y(t))
%
= &
\vert G(0) \vert^{2}
(\epsilon-t)
+
o_{t+\epsilon}(t+\epsilon).
\end{split}
\end{equation*}
Hence for $\epsilon>0$ sufficiently small 
the energy $J$ is increasing along $y$ for a short time. This of course contradicts \eqref{general_form_of_a_flow}, since $y_{0}\not \in \{ \partial J=0 \}$.  
\item The case 
$\dim(Y^{-})=\infty$ in Lemma \ref{non_degenerate_implies_strongly_critical_2} and hence the question, whether or not a critical point at infinity can have an infinite index, is more delicate.  For instance does 
  $\lambda=\infty,\, y=0$  for 
$$
J(u)=\sigma + \frac{1}{\lambda^{2}} - \vert y \vert^{2}
\; \text{ on a Hilbertspace  } \; Y
\; \text{ with } \; \dim(Y)=\infty
$$
represent an obstacle to energetic deformation, i.e. for any energy decreasing flow of type $\Psi \in \Pi F$? 
We conjecture, that the answer is no. In any case infinite indices do not occur in our framework.
\end{enumerate}
\end{remark}

We finally characterize $M_{\sigma}$ as an obstacle to energetic deformation as follows.
\begin{proposition}\label{obstacling}
Let $\sigma_{1}\leq \sigma_{2}$. Then
every $\sigma_{2}$-reducible $W$ is also $\sigma_{1}$-reducible, if and only if
$$
\; \forall \; \sigma_{1}<\sigma\leq \sigma_{2}
\; : \; 
M_{\sigma}=\emptyset.
$$
\end{proposition}	
\begin{proof}
The case $\sigma_{1}=\sigma_{2}$ trivially holds true. Hence let $\sigma_{1}<\sigma_{2}$.

Suppose, that every $\sigma_{2}$-reducible $W$ is also $\sigma_{1}$-reducible. Since for $\sigma_{1}<\sigma\leq \sigma_{2}$  trivially every $\sigma$-reducible $W$ is also $\sigma_{2}$-reducible, we find, that every $\sigma$-reducible $W$ is also $\sigma_{1}$-reducible. Consider hence for $\sigma_{1}<\sigma\leq \sigma_{2}$ an arbitrary $\sigma$-reducible $W$ and choose $\varepsilon >0$ such, that 
$\sigma_{1}+\varepsilon<\sigma - \varepsilon $. Since $W$ is also $\sigma_{1}$-reducible, we find $\Psi \in \Pi F$ and $T\geq 0$ such, that $J(\Psi(W,T))\leq \sigma_{1}+\varepsilon$, cf. Definition \ref{reducible}. As a consequence and, since $\sigma_{1}+\varepsilon <\sigma-\varepsilon$, the empty set 
$\emptyset \subseteq E_{\sigma}$ is $\sigma$-capturing, cf. Definition \ref{capturing}, and then trivially strongly $\sigma$-critical as well, cf. Definition \ref{strongly_critical}. By uniqueness of a minimal, strongly $\sigma$-critical $M_{\sigma}\subseteq E_{\sigma}$ we conclude 
$M_{\sigma}=\emptyset$. 

Vice versa suppose, that 
$
\; \forall \; \sigma_{1}<\sigma\leq \sigma_{2}
\; : \; 
M_{\sigma}=\emptyset
$
and consider
$$
s
=
\inf
\{ 
\sigma_{1}\leq \sigma \leq \sigma_{2} 
\; \mid \; 
\; \forall \; W \; \sigma_{2}\text{-reducible}\; : \; 
W\; \sigma\text{-reducible}.
\} 
$$
In view of Definition \ref{reducible} we then have to show $s=\sigma_{1}$. Arguing by contradiction we assume
$\sigma_{1}<s\leq \sigma_{2}$ and find, that every $\sigma_{2}$-reducible $W$ is also $s$-reducible. Since $M_{s}=\emptyset$ 
is strongly $s$-critical and by Definition \eqref{strongly_critical} also $s$-capturing, we find 
$\varepsilon>0$ and for every $s$-reducible $W$ some $\Psi\in \Pi F$ and $T\geq 0$ such, that
$\Psi(W,T)\subseteq \{ J< s-2\varepsilon \}$. But this implies, that every $s$-reducible $W$ is also $(s-\varepsilon)$-reducible
and therefore every $\sigma_2$-reducible $W$ is also $(s-\varepsilon )$-reducible. This contradicts the minimality of $s$ leading to the desired contradiction. 
\end{proof}
From Proposition \ref{obstacling} we recover the classical deformation lemma. 
\begin{lemma}
Let $\sigma_{1} < \sigma_{2}$ and suppose, 
that 
$$
\; \forall \; \sigma_{1}\leq \sigma \leq \sigma_{2}
\; : \; 
M_{\sigma}=\emptyset.
$$
Then 
$
\{ J\leq \sigma_{2} \} 
\xhookrightarrow{\;\;\;\;\;} 
\{ J\leq \sigma_{1} \}  
$
as a weak deformation retract. 
\end{lemma}	
\begin{proof}
Clearly $\{ J\leq \sigma_{2} \}$ is $\sigma_{2}$-reducible and by virtue of Proposition \ref{obstacling} also $\sigma_{1}$-reducible. Since $M_{\sigma_{1}}=\emptyset$ is strongly $\sigma_{1}$-critical and hence $\sigma_{1}$-capturing, cf. Definitions \ref{capturing}, \ref{strongly_critical}, we find
$\varepsilon >0$ and for $W=\{ J\leq \sigma_{2} \} $ a deformation $\Psi\in \Pi F$ and $T\geq 0$ such, that 
$\Psi(W,T)\subseteq \{ J<\sigma_{1}-\varepsilon \}$. And, since the flow $\Psi$ does not increase energy, clearly
$
\Psi : 
 \{ J\leq \sigma_{1} \} \times [0,T]
\longrightarrow 
\{ J\leq \sigma_{1} \}  
$.
\end{proof}

\section{Preliminaries}\label{section_preliminaries}
Let us start with a quantification of the deficit for some $\varphi_{a,\lambda}$ from solving \eqref{the_equation}.
\begin{lemma}\label{lem_emergence_of_the_regular_part}
There holds 
$
L_{g_{0}}\varphi_{a, \lambda}
=
O
(
\varphi_{a, \lambda}^{\frac{n+2}{n-2}}
).
$
More precisely on a geodesic ball $B_{\alpha}( a )$ for $\alpha>0$ small 
\begin{equation*}
L_{g_{0}}\varphi_{a, \lambda}
= 
4n(n-1)\varphi_{a, \lambda}^{\frac{n+2}{n-2}}
-
2nc_{n}
r_{a}^{n-2}((n-1)H_{a}+r_{a}\partial_{r_{a}}H_{a}) \varphi_{a, \lambda}^{\frac{n+2}{n-2}} 
+
\frac{u_{a}^{\frac{2}{n-2}}R_{g_{a}}}{\lambda}\varphi_{a, \lambda}^{\frac{n}{n-2}}
+
o(r_{a}^{n-2})\varphi_{a, \lambda}^{\frac{n+2}{n-2}},
\end{equation*}
where  $r_{a}=d_{g_{a}}(a, \cdot)$. In particular
\begin{enumerate}[label=(\roman*)]
 \item \quad 
$
L_{g_{0}}\varphi_{a, \lambda}
= 
4n(n-1)
[1
-
\frac{c_{n}}{2}r_{a}^{n-2}(
H_{a}(a) + n \nabla H_{a}(a)x
)
]
\varphi_{a, \lambda}^{\frac{n+2}{n-2}} 
 +
O
\begin{pmatrix}
\lambda^{-2}\varphi_{a,\lambda}
\end{pmatrix}
\; \text{ for }\; n=5
;
$
\item \quad 
$L_{g_{0}}\varphi_{a,\lambda}=
4n(n-1)\varphi_{a,\lambda}^{\frac{n+2}{n-2}}
=
4n(n-1)[1+\frac{c_{n}}{2}W(a)\ln r]\varphi_{a,\lambda}^{\frac{n+2}{n-2}}
+
O
(\lambda^{-2}\varphi_{a,\lambda}) \; \text{ for } n = 6;$
\item \quad 
$L_{g_{0}}\varphi_{a,\lambda}=
4n(n-1)\varphi_{a,\lambda}^{\frac{n+2}{n-2}}
=
O
(\lambda^{-2}\varphi_{a,\lambda}) 
\; \text{ for }\; n\geq 7.
$
\end{enumerate}
The expansions stated above persist upon taking $\lambda \partial_{\lambda}$ and $\frac{\nabla_{a}}{\lambda}$ derivatives.
\end{lemma}
\begin{proof}
Cf. Lemma \ref{I-lem_emergence_of_the_regular_part} in \cite{MM1}.
\end{proof}

Thereby we  may describe the blow-up behaviour of Palais-Smale sequences for 
\eqref{the_functional}.
 \begin{proposition}\label{blow_up_analysis}
 Let $(u_{m})_{m} \subset W^{1,2}(M,g_0)$ be a sequence with $u_m \geq 0$ and $k_{m}=k_{u_{m}}=1$  
 satisfying 
 \begin{equation}\label{def_PS}
 J(u_{m})=r_{u_{m}}\longrightarrow r_{\infty} \; \text{ and }\;  \partial J(u_{m})\longrightarrow 0
 \;\text{ in }\; W^{-1,2}(M,g_0)
 \end{equation}
 for some $ r_{\infty}>0$.  Then  up to a subsequence there exist 
 $$
  u_\infty : M \longrightarrow [0,\infty)
 \; \text{ smooth with  } \;
 \partial J(u_{\infty})=0,
 $$  
 $q \in \mathbb{N}_{0}$ and for $i=1, \ldots,q$ sequences
 \begin{equation*}\begin{split}
\alpha_{i,m}\subset \R_{+}
\; \text{ and } \;
 M \supset (a_{i,m})\longrightarrow a_{i_{\infty}}
, \quad \R_{+}\supset \lambda_{i,m} \longrightarrow \infty
 \; \text{ as }\; m\longrightarrow  \infty
 \end{split}\end{equation*}
for some $ a_{i_{\infty}}\in M $  such, that 
 $u_{m}
 =
u_{\infty }
 +
\sum_{i}\alpha_{i,m}\varphi_{a_{i,m},\lambda_{i,m}}+v_{m}$
 and
 \begin{equation}\label{alpha_i_m_condition}
 \Vert v_{m} \Vert \longrightarrow 0 
 \; \text{ and }\; \
 \frac{r_{u_{m}}K(a_{i,m})\alpha_{i,m}^{\frac{4}{n-2}}}{4n(n-1)}\longrightarrow 1
 \end{equation} 
 and  for each pair $1\leq i<j\leq q$  there holds 
\begin{equation}\label{eq:eij}
\varepsilon_{i,j}
=
(
\frac{\lambda_{j}}{\lambda_{i}}
+
\frac{\lambda_{i}}{\lambda_{j}}
+
\lambda_{i}\lambda_{j}\gamma_{n}G_{g_{0}}^{\frac{2}{2-n}}(a _{i},a _{j})
)^{\frac{2-n}{2}}
,\; 
 \gamma_{n}=(4n(n-1))^{\frac{2}{2-n}}.
\end{equation}
\end{proposition}
\begin{proof}
Cf. Proposition \ref{I-blow_up_analysis} in \cite{MM1}.
\end{proof}
\begin{remark}\label{remark_blow_up_analysis} We remark, that 
\begin{enumerate}[label=(\roman*)]
\item  $u_{m} \xrightharpoondown{\quad} u_{\infty}$ weakly implies, that necessarily and in addition to 
\eqref{alpha_i_m_condition} there holds
\begin{equation}\label{u_infinity_equation}
\begin{split}
L_{g_{0}}u_{\infty}
= &
r_{u_{m}}K u_{\infty}^{\frac{n+2}{n-2}}+o(1),
\end{split}
\end{equation}
i.e. $L_{g_{0}}u_{\infty}=r_{\infty}Ku_{\infty}^{\frac{n+2}{n-2}}$ and therefore
$
J(u_{\infty})
=
r_{\infty}(\int Ku_{\infty}^{\frac{2n}{n-2}}d\mu_{g_{0}})^{\frac{2}{n}}.
$
\item for the limiting energy $J_{\infty}=\lim_{m \to \infty }J(u_{m})$ we then obtain
\begin{equation}\label{limiting_energy_implicit}
\begin{split}
J_{\infty}
= &
\frac
{
\int L_{g_{0}}u_{\infty}u_{\infty}d\mu_{g_{0}}+c_{n}\sum_{i}\alpha_{i_{\infty}}^{2}
}
{
(
\int Ku_{\infty}^{\frac{2n}{n-2}}+\frac{c_{n}}{4n(n-1)}\sum_{i}K(a_{i_{\infty}})\alpha_{i_{\infty}}^{\frac{2n}{n-2}}
)^{\frac{n-2}{n}}
},
\end{split}
\end{equation}
where 
$$
c_{n}
=
\lim_{\lambda \to \infty }\int L_{g_{0}}\varphi_{a,\lambda}\varphi_{a,\lambda}d\mu_{g_{0}}
=
4n(n-1)\int_{\R^{n}}\frac{1}{(1+r^{2})^{n}},
$$
cf. \eqref{eq:bubbles} and Lemma \ref{lem_emergence_of_the_regular_part}, and 
\begin{equation}\label{alpha_i_infty_condition}
\begin{split}
\alpha_{i_{\infty}}
=
(\frac{4n(n-1)}{r_{\infty}K(a_{i_{\infty}})})^{\frac{n-2}{4}},
\end{split}
\end{equation} 
cf. \eqref{alpha_i_m_condition}. Inserting \eqref{alpha_i_infty_condition} into \eqref{limiting_energy_implicit}
we conclude
\begin{equation}\label{limiting_energy_explicit}
\begin{split}
J_{\infty}
= &
(J(u_{\infty})+c_{n}\sum_{i}(\frac{4n(n-1)}{K(a_{i_{\infty}})^{\frac{n-2}{2}}}))^{\frac{2}{n}}.
\end{split}
\end{equation}
\item restricting to $X=\{ \Vert \cdot \Vert =1\} $ instead of normalising to $k=1$ is the same up to
\begin{enumerate}[label=(\roman*)]
\item[(1)] $J(u_{m})=k_{u_{m}}^{\frac{2-n}{n}}\longrightarrow k_{\infty}^{\frac{n}{2-n}}$, cf. \eqref{def_PS}
\item[(2)] $\frac{K(a_{i,m})\alpha_{i,m}^{\frac{4}{n-2}}}{4n(n-1)k_{u_{m}}}= 1+o(1)$, cf. 	\eqref{alpha_i_m_condition}
\item[(3)] $L_{g_{0}}u_{\infty}=\frac{K}{k_{u_{m}}}u_{\infty}^{\frac{n+2}{n-2}}+o(1)$, cf. \eqref{u_infinity_equation}.
\end{enumerate}
and $\Vert \cdot \Vert=1$ necessitates 
\begin{equation}\label{normalisation_to_restriction}
\begin{split}
1=c_{n}\sum_{i}\alpha_{i}^{2}+\int L_{g_{0}}u_{\infty}u_{\infty}d\mu_{g_{0}}.
\end{split}
\end{equation}

\item as a consequence of \eqref{limiting_energy_explicit} the number of bubbles and
as a consequence of  \eqref{limiting_energy_explicit} and \eqref{normalisation_to_restriction} 
norm and  energy of the weak limit  $u_{\infty}$ of a Palais-Smale sequence are bounded.
\end{enumerate}
\end{remark}

Proposition \ref{blow_up_analysis} justifies to consider the following subset of peaked function and look for zero weak limit Palais-Smale sequences thereon only.

\begin{definition}\label{V_p_e}
For $\varepsilon>0,\; q\in \N$ and $ u\in W^{1,2}(M,g_{0})$ let
\begin{enumerate}[label=(\roman*)]
\item  
$ _{} $ \vspace{-21pt}
\begin{equation*}
\begin{split}
\hspace{-58pt}
A_{u}(q, \varepsilon)
 =
 \{ 
 ( \alpha^{i}, \lambda_{i}, a_{i})\in \R^{q}\times \R^{q}_{+} \times M^{q} \mid 
& 
\underset{i\neq j}{\forall\;}\;
  \lambda_{i}^{-1}, \lambda_{j}^{-1}, \varepsilon_{i,j}, \\
&
 \vert 1-\frac{r\alpha_{i}^{\frac{4}{n-2}}K(a_{i})}{4n(n-1)k}\vert,
 \Vert u-\alpha^{i}\varphi_{a_{i}, \lambda_{i}}\Vert
 \leq\varepsilon
 \}
\end{split}
\end{equation*}
 \item \quad
 $ 
 V(q, \varepsilon)
 = 
 \{
 u\in W^{1,2}(M,g_{0})  
 \mid
 A_{u}(q, \varepsilon)\neq \emptyset
 \} 
 $
\end{enumerate} 
\end{definition}
However, 
for a precise analysis of $J$ on $V(q,\varepsilon)$  it is convenient to make the representation of its elements unique.
 \begin{proposition}\label{prop_optimal_choice} 
 For every $\varepsilon_{0}>0$  there exists $\varepsilon_{1}>0$ such, that  for $u\in V( q, \varepsilon)$ 
 with $\varepsilon<\varepsilon_{1}$
 \begin{equation*}\begin{split}
 \inf_
 {
 (\tilde\alpha_{i}, \tilde a_{i}, \tilde\lambda_{i})\in A_{u}(q,2\varepsilon_{0}) 
 }
 \Vert
 u
 -
 \tilde\alpha^{i}\varphi_{\tilde a_{i}, \tilde \lambda_{i}}
 \Vert^{2} 
 \end{split}\end{equation*}
 admits a unique minimizer $(\alpha_{i},a_{i}, \lambda_{i})\in A_{u}(q, \varepsilon_{0})$ 
 depending smoothly on $u$
 and we set 
 \begin{equation}\begin{split} \label{eq:v}
 \varphi_{i}=\varphi_{a_{i}, \lambda_{i}}, \quad v=u-\alpha^{i}\varphi_{i}, \quad K_{i}=K(a_{i}).
 \end{split}\end{equation}
 \end{proposition} 
 \begin{proof}
Cf. Appendix A in \cite{bc}.
 \end{proof}
 
For the sake of brevity and recalling  \eqref{def_derivatives_g_0} we denote e.g.
$$K_{i}=K(a_{i}),\;\nabla K_{i}=\nabla K(a_{i}),\;\Delta K_{i}=\Delta K(a_{i})$$
for a  set of points $\{a_i\}_i \subset M$  and for $k,l=1,2,3$ and $ \lambda_{i} >0, \, a _{i}\in M, \,i= 1, \ldots,q$ we let
\begin{enumerate}[label=(\roman*)]
 \item \quad 
$\varphi_{i}=\varphi_{a_{i}, \lambda_{i}}$ and $(d_{1,i},d_{2,i},d_{3,i})=(1,-\lambda_{i}\partial_{\lambda_{i}}, \frac{1}{\lambda_{i}}\nabla_{a_{i}})$
 \item \quad
$\phi_{1,i}=\varphi_{i}, \;\phi_{2,i}=-\lambda_{i} \partial_{\lambda_{i}}\varphi_{i}, \;\phi_{3,i}= \frac{1}{\lambda_{i}} \nabla_{ a _{i}}\varphi_{i}
,\; \text{ in particular }\;
\phi_{k,i}=d_{k,i}\varphi_{i}
$
 \item \quad 
$
\langle \phi_{k,i}\rangle
=
span\{\phi_{k,i}\;:\; k=1,2,3\; \text{ and }\; 1\leq i \leq q\},
$
\end{enumerate}
in particular
$\nabla_{a_{i}}\varphi_{i}
= 
(\nabla_{g_{0}})_{a_{i}}\varphi_{a_{i},\lambda_{i}}
$
pointwise, i.e.
\begin{equation*}
\begin{split}
\forall\; x\in M \; \text{ and } \; w\in T_{\cdot}M\;:\;
\langle \nabla_{a_{i}}\varphi_{i}(x),w(a_{i})\rangle_{g_{0}}
=
d_{a_{i}}\varphi_{a_{i},\lambda_{i}}(x)w(a_{i}).
\end{split}
\end{equation*}
With this notation the term  
$$v=u-\alpha^{i}\varphi_{i}$$
from Proposition \ref{prop_optimal_choice}
is orthogonal to
$
 \langle \phi_{k,i}\rangle 
$ 
with respect to the scalar product 
 \begin{equation*}\begin{split}
 \langle \cdot, \cdot\rangle_{L_{g_{0}}}
 =
 \langle L_{g_{0}}\cdot,\cdot\rangle_{L^{2}_{g_{0}}}
 \end{split}\end{equation*}
and we define for  $u\in V( q, \varepsilon)$ its complement
 \begin{equation*} 
 H_{u}( q, \varepsilon)
 =
 \langle 
\phi_{k,i} \rangle
 ^{\perp_{L_{g_{0}}}}.
 \end{equation*}
A precise  analysis of $J$ on $V(q,\varepsilon)$ was performed in \cite{MM1} by testing the variation $\partial J$ separately with the bubbles 
$\varphi_{i}$ and their derivatives $-\lambda_{i}\partial_{\lambda_{i}}\varphi_{i}, \frac{1}{\lambda_{i}}\nabla_{a_{i}}\varphi_{i}$ on the one hand 
and orthogonally to them, i.e. with elements of $H_{u}(q,\varepsilon)$ on the other. 

\

Recalling \eqref{eq:eij}  we collect below some principal interactions over various integrals involving $ \phi_{k,i} $, which clearly appear in the gradient testing or expansion of the energy $J$ itself.
Note, that $ \Vert \phi_{k,i}\Vert \simeq 1 $. 
\begin{lemma}\label{lem_interactions}
For $k,l=1,2,3$ and $i,j = 1, \ldots,q$ and we have with constants 
$b_{k},c_{k}>0$ 
\begin{enumerate}[label=(\roman*)]
 \item \quad
$ \vert \phi_{k,i}\vert, 
  \vert \lambda_{i}\partial_{\lambda_{i}}\phi_{k,i}\vert,
  \vert \frac{1}{\lambda_{i}}\nabla_{a_{i}} \phi_{k,i}\vert
  \leq 
  C \varphi_{i}
$
 \item \quad
$ 
\int \varphi_{i}^{\frac{4}{n-2}} \phi_{k,i}\phi_{k,i}d\mu_{g_{0}}
=
c_{k}\cdot id
+
O(\frac{1}{\lambda_{i}^{2}}), \;c_{k}>0
$
\item \quad 
for  $i\neq j$
\begin{equation*}
\int \varphi_{i}^{\frac{n+2}{n-2}}\phi_{k,j}d\mu_{g_{0}}
= 
b_{k}d_{k,i}\varepsilon_{i,j}
=
\int \varphi_{i}d_{k,j}\varphi_{j}^{\frac{n+2}{n-2}}  d\mu_{g_{0}}
\end{equation*}
 \item \quad 
$
\int \varphi_{i}^{\frac{4}{n-2}} \phi_{k,i}\phi_{l,i}d\mu_{g_{0}}
= 
O(\frac{1}{\lambda_{i}^{2}})$
for $k\neq l$ and
$
\int \varphi_{i}^{\frac{n+2}{n-2}} \phi_{k,i}d\mu_{g_{0}}
=
O
\begin{pmatrix}
\lambda_{i}^{2-n} & \text{for } n\leq 5 \\
\frac{\ln \lambda_{i}}{\lambda_{i}^{4}} & \text{for } n=6 \\
\lambda_{i}^{4} &  \text{for } n\geq 7
\end{pmatrix} 
$
 for $k=2,3$
 \item \quad
$
\int \varphi_{i}^{\alpha }\varphi_{j}^{\beta} d\mu_{g_{0}}
=
O(\varepsilon_{i,j}^{\beta})
$
for $i\neq j,\;\alpha +\beta=\frac{2n}{n-2}, \; \alpha>\frac{n}{n-2}>\beta\geq 1$
\item \quad
$
\int \varphi_{i}^{\frac{n}{n-2}}\varphi_{j}^{\frac{n}{n-2}} d\mu_{g_{0}}
=
O(\varepsilon^{\frac{n}{n-2}}_{i,j}\ln \varepsilon_{i,j}) \; \text{ for }\; i\neq j
$
 \item \quad 
$
(1, \lambda_{i}\partial_{\lambda_{i}}, \frac{1}{\lambda_{i}}\nabla_{a_{i}})\varepsilon_{i,j}=O(\varepsilon_{i,j})
 \; \text{ for }\;
 i\neq j$. 
\end{enumerate}  
\end{lemma}
\begin{proof}
Cf. Lemma 3.4  in \cite{may-cv} or Lemma \ref{I-lem_interactions} in  \cite{MM1}.
\end{proof}

Let us comment on the following  lemmata, which describe the testing of $\partial J$. First a testing in an orthogonal direction is due to orthogonalities small.
 \begin{lemma} 
 \label{lem_testing_with_v} 
 For $u\in V(q,\varepsilon)$ with $k=1$ and $\nu\in H_{u}(q,\varepsilon)$ there holds
 \begin{equation*}
 \partial J(\alpha^{i}\varphi_{i})\nu
 = 
 O\bigg(
 \bigg[
 \sum_{r}\frac{\vert \nabla K_{r}\vert}{\lambda_{r}}
 +
 \sum_{r}\frac{1}{\lambda_{r}^{2}}
 +
 \sum_{r\neq s}\varepsilon_{r,s}^{\frac{n+2}{2n}}
 \bigg]
 \Vert \nu \Vert \bigg).  
 \end{equation*}
 \end{lemma}
 \begin{proof}
Cf. Proposition 4.4 in \cite{may-cv} or Lemma \ref{I-lem_testing_with_v} in \cite{MM1}.
 \end{proof}

In combination with the well known  uniform positivity of the second variation on the orthogonal space
$\langle \phi_{k,i}\rangle^{\perp_{L_{g_{0}}}}$,
cf. \cite{Rey},  
this allows us  to estimate $v$ itself in terms of the aforegoing quantities. 
 \begin{lemma} 
 \label{lem_v_part_gradient} 
 For $u\in V(q,\varepsilon)$ with $k=1$ and $v$ is as in \eqref{eq:v} there holds
 \begin{equation*}
 \Vert v \Vert
 =
 O\bigg(
 \sum_{r}\frac{\vert \nabla K_{r}\vert}{\lambda_{r}}
 +
 \sum_{r}\frac{1}{\lambda_{r}^{2}}
 +
 \sum_{r\neq s}\varepsilon_{r,s}^{\frac{n+2}{2n}}
 +
 \vert \partial J(u)\vert
 \bigg).
 \end{equation*}
  \end{lemma}
  \begin{proof}
Cf. Corollary 4.6 in \cite{may-cv} or Lemma \ref{I-lem_v_part_gradient} in  \cite{MM1}.
 \end{proof}

 The latter smallness  estimate will turn out to be sufficient to consider $v$ as a negligible quantity in the sense, that $v$ is not responsible for a blow-up. 

\

Let us turn to the testing in the directions of the bubbles and their derivative
as had been performed carefully in low dimensions $n=3,4,5$ in Section 4 of \cite{may-cv}.
We note, that for each bubble we have three quantities associated, namely $\alpha,a$ and $\lambda$. The $\alpha$-direction then corresponds to a testing with a bubble itself, since
$\alpha \partial_{\alpha}(\alpha \varphi_{a,\lambda})=\alpha \varphi_{a,\lambda}$.
Again for the sake of brevity let us define the quantities
\begin{equation}\label{alpha_and_alpha_K}
\alpha^{2}=\sum_{i}\alpha_{i}^{2}
\; \text{ and }\; 
\alpha_{K}^{\frac{2n}{n-2}}=\sum_{i}K_{i}\alpha_{i}^{\frac{2n}{n-2}},
\end{equation}
which are the principal terms in the nominator and denominator of $J$, cf. \eqref{the_functional}.

 \begin{lemma} 
  \label{lem_alpha_derivatives_at_infinity} 
  For $u\in V(q,\varepsilon)$ and $\varepsilon>0$ sufficiently small the three quantities 
  $\partial J(u)  \phi_{1,j}$, 
  $\partial J(\alpha^{i}\varphi_{i})\phi_{1,j}$, 
  $\partial_{\alpha_{j}}J(\alpha^{i}\varphi_{i})$ can be written as 
  \begin{equation*}
  \begin{split}  
  \frac{\alpha_{j}}
  {(\alpha_{K}^{\frac{2n}{n-2}})^{\frac{n-2}{n}}}
  \bigg( 
  &
  \grave c_{0}\big(
  1
  -
  \frac{\alpha^{2}}{\alpha_{K}^{\frac{2n}{n-2}}}K_{j}\alpha_{j}^{\frac{4}{n-2}}
 \big)
  -
  \grave c_{2}
  \big(
  \frac{\Delta K_{j}}{K_{j}\lambda_{j}^{2}}
  -
  \sum_{k}\frac{\Delta K_{k}}{K_{k}\lambda_{k}^{2}}
  \frac{\alpha_{k}^{2}}{\alpha^{2}}
  \big)
  \\ & 
  +
  \grave b_{1} \bigg(
  \sum_{k\neq l}
  \frac{\alpha_{k}\alpha_{l}}{\alpha^{2}}
  \varepsilon_{k,l}
  -
  \sum_{j\neq i}\frac{\alpha_{i}}{\alpha_{j}}\varepsilon_{i,j}
  \bigg)
  -
  \grave d_{1}
  \begin{pmatrix}
  \frac{H_{j}}{\lambda_{j}^{3}} -\sum_{k}\frac{\alpha_{k}^{2}}{\alpha^{2}}\frac{H_{k}}{\lambda_{k}^{3 }} &\text{for } n=5\\
  \frac{W_{j}\ln \lambda_{j}}{\lambda_{i}^{4}}-\sum_{k}\frac{\alpha_{k}^{2}}{\alpha^{2}}\frac{W_{k}\ln \lambda_{k}}{\lambda_{k}^{4}} &\text{for } n=6\\
  0 &\text{for }n\geq 7
  \end{pmatrix} 
  \bigg)
  \end{split}
  \end{equation*}
  with positive constants $\grave c_{0},\grave c_{2},\grave b_{1},\grave d_{1}$ and  up to some
$ 
  O
  \big(
  \sum_{r\neq s} 
  \frac{\vert \nabla K_{r}\vert^{2}}{\lambda_{r}^{2}}
  +
  \frac{1}{\lambda_{r}^{4}}
  +
  \varepsilon_{r,s}^{\frac{n+2}{n}}
  +
  \vert \partial J(u)\vert^{2}
  \big).
$ 
   In particular 
  \begin{equation*}
\forall \;j\;:\;  \frac{\alpha^{2}}{\alpha_{K}^{\frac{2n}{n-2}}}K_{j}\alpha_{j}^{\frac{4}{n-2}}
  =
  1
  +
  O 
  \big(
  \sum_{r\neq s} 
  \frac{1}{\lambda_{r}^{2}}
  +
  \varepsilon_{r,s}
  +
  \vert \partial J(u)\vert
  \big).
  \end{equation*}  
  \end{lemma}
  \begin{proof}
Cf. Lemma \ref{I-lem_alpha_derivatives_at_infinity} in  \cite{MM1},  for instance
see also Lemma A.4.3 in \cite{bab1} or Proposition 5.1 in \cite{ BenAyed_Ahmeou}.
 \end{proof}
     
Evidently the principal term due to largeness of the concentration parameters $\lambda_{i}$ and smallness of the interaction terms $\varepsilon_{i,j}$  in the above expansion is the one related to $\grave{c_{0}}$ forcing $\alpha_{j}$ into a certain regime. 

\begin{lemma} 
  \label{lem_lambda_derivatives_at_infinity} 
  For $u\in V(q,\varepsilon)$ and $\varepsilon>0$ sufficiently small the three quantities 
  $\partial  J  (u) \phi_{2,j}$,  
  $\partial J(\alpha^{i}\varphi_{i})\phi_{2,j}$ and 
  $ \frac{\lambda_{j}}{\alpha_{j}}\partial_{\lambda_{j}}J(\alpha^{i}\varphi_{i})$
  can be written as 
  \begin{equation*}
  \begin{split}
  \frac{\alpha_{j}}{(\alpha_{K}^{\frac{2n}{n-2}})^{\frac{n-2}{n}}}
  \bigg(
  \tilde c_{2}\frac{\Delta K_{j}}{K_{j}\lambda_{j}^{2}} 
  -
  \tilde b_{2}\sum_{j\neq i }\frac{\alpha_{i}}{\alpha_{j}}\lambda_{j}\partial_{\lambda_{j}}\varepsilon_{i,j} 
 +
  \tilde  d_{1}
  \begin{pmatrix}
  \frac{H_{j}}{\lambda_{j}^{3 }} &\text{for }\; n=5\\
  \frac{W_{j}\ln \lambda_{j}}{\lambda_{j}^{4}} &\text{for }\; n=6\\
  0 &\text{for }\; n\geq 7
  \end{pmatrix} 
  \bigg)
  \end{split}
  \end{equation*}
  with positive constants $\tilde c_{1},\tilde c_{2},\tilde d_{1},\tilde b_{2}$ and up to some 
$ 
  O
  \big(
  \sum_{r\neq s}
  \frac{\vert \nabla K_{r}\vert^{2}}{\lambda_{r}^{2}}
  +
  \frac{1}{\lambda_{r}^{4}}
  +
  \varepsilon_{r,s}^{\frac{n+2}{n}}
  +
  \vert \partial J(u)\vert^{2}
  \big).
$ 
  \end{lemma}
  \begin{proof}
Cf. Lemma \ref{I-lem_lambda_derivatives_at_infinity} in \cite{MM1}, for instance see also Lemma A.4.3 in \cite{bab1} or Proposition 5.1 in \cite{ BenAyed_Ahmeou}.
  \end{proof} 

  Here at least in high dimensions the principal terms are the ones related to $\tilde c_{2}$ and $\tilde b_{2}$.  
  The first one turns out to be responsible for a potential diverging flow line within $V(q,\varepsilon)$ depending on  the sign of $\Delta K$, the latter one, measuring  interactions,  may be relatively strong or weak depending on, whether the corresponding $a_{j}$ are close to $a_{i}$ or not. 
  In any case these interaction terms  will turn out to be responsible for excluding tower bubbling, i.e. multiple bubbles concentrating at the same point along a flow line, just as they prevent tower bubbling in the subcritical case, cf. \cite{MM2}.
  The location of a bubble $\varphi_{a,\lambda}$ on $M$ in the sense of the centre $a$  is principally determined from the $a$-testing below.
  \begin{lemma} 
  \label{lem_a_derivatives_at_infinity} 
  For $u\in V(q,\varepsilon)$ and $\varepsilon>0$ sufficiently small the three quantities 
  $\partial J(u)\phi_{3,j} $, 
  $\partial J(\alpha^{i}\varphi_{i})\phi_{3,j}$ and 
  $\frac{\nabla_{a_{j}}}{\alpha_{j}\lambda_{j}}J(\alpha^{i}\varphi_{i})$ can be written as 
  \begin{equation*}
  \begin{split}
  -\frac{\alpha_{j}}{(\alpha_{K}^{\frac{2n}{n-2}})^{\frac{n-2}{n}}}
  \left(
  \check{c}_{3}\frac{\nabla K_{j}}{K_{j}\lambda_{j}}
  +
  \check{c}_{4} \frac{\nabla \Delta K_j}{K_j \lambda_{j}^3}
  +
  \check{b}_{3}
  \sum_{j\neq i}
  \frac{\alpha_{i}}{\alpha_{j}}
  \frac{\nabla_{a_{j}}}{\lambda_{j}}\varepsilon_{i,j}
  \right)
  \end{split}
  \end{equation*}
  with positive constants $\check{c}_{3}, \check{c}_{4} , \check{b}_{3}$  and up to some 
$ 
  O
  \big(
  \sum_{r\neq s}
  \frac{\vert \nabla K_{r}\vert^{2}}{\lambda_{r}^{2}}
  +
  \frac{1}{\lambda_{r}^{4}}
  +
  \varepsilon_{r,s}^{\frac{n+2}{n}}
  +
  \vert \partial J(u)\vert^{2}
  \big).
$ 
  \end{lemma}
  \begin{proof}
Cf. Lemma \ref{I-lem_a_derivatives_at_infinity} in \cite{MM1}, for instance see also Lemma A.4.3 in \cite{bab1} or Proposition 5.1 in \cite{ BenAyed_Ahmeou}.
 \end{proof}
 
  Evidently the principal terms are  the one related to $\check{c}_{3}$,  trying to force the centres of concentration to be close to critical points of $K$, and the one related to $\check{b}_{3}$.
  
  \

  Of course the source of delicacy is, that the principle terms above are related by their error terms.
  
  \begin{proposition}\label{prop_gradient_bounds}
 For $\varepsilon>0$ sufficiently small there holds uniformly on
 $
 V(q,\varepsilon) \; \cap \; \{k=1\}
 $ 
 \begin{equation*}
\sum_{r\neq s}\frac{\vert \nabla K_{r}\vert}{\lambda_{r}}+\frac{1}{\lambda_{r}^{2}}
 +
 \vert
 1
 -
 \frac{\alpha^{2} K_{r}\alpha_{r}^{\frac{4}{n-2}} }{\alpha_{K}^{\frac{2n}{n-2}}}
 \vert
 +\varepsilon_{r,s}
 \lesssim
 \vert \partial J\vert
  \lesssim 
 \sum_{r\neq s}\frac{\vert \nabla K_{r}\vert}{\lambda_{r}}
 +
 \frac{1}{\lambda_{r}^{2}}
 +
 \vert 
 1
 -
 \frac{\alpha^{2}K_{r}\alpha_{r}^{\frac{4}{n-2}} }{\alpha_{K}^{\frac{2n}{n-2}}}
\vert
 +
 \varepsilon_{r,s}^{\frac{n+2}{2n}}
 +
 \Vert v \Vert
.
 \end{equation*}
  \end{proposition}
  \begin{proof}
The lower bound  is due to  Theorem \ref{I-lem_top_down_cascade} in \cite{MM1},  the upper bound   due to  Lemma \ref{I-lem_upper_bound} in \cite{MM1}. 
\end{proof}    

Here and later on we use for two functions $a,b\geq 0$ the shorthand notation 
$$ a\lesssim b :\Longleftrightarrow \;\exists\; C>0\;:\; a\leq Cb.$$ 
We note, that the latter gradient estimates evidently prevent the existence of a solution in $V(q,\varepsilon)$ and allow us to compare the quantities appearing to  $\vert \partial J\vert $ and vice versa.
Finally we may perform an expansion of the energy itself on $V(q,\varepsilon)$, which reads as

  \begin{lemma}
  For $u=\alpha^{i}\varphi_{i}+v\in V(q,\varepsilon)$ and $\varepsilon>0$, 
  both $J(u)$ and $J ( \alpha^{i}  \varphi_{i}) $ can be written as 
  \begin{equation*}
  \frac
  {\hat c_{0}\alpha^{2}}
  {(\alpha_{K}^{\frac{2n}{n-2}})^{\frac{n-2}{n}}}
  \Bigg(
  1
  -
  \hat c_{2}\sum_{i}\frac{\Delta K_{i}}{K_{i}\lambda_{i}^{2}} \frac{\alpha_{i}^{2} }{\alpha^{2}} 
  -
  \hat b_{1}\sum_{i\neq j}\frac{\alpha_{i}\alpha_{j}}{\alpha^{2}}
  \varepsilon_{i,j}
  -
  \hat d_{1} \sum_{i}
  \frac{\alpha_{i}^{2}}{\alpha^{2}}
  \begin{pmatrix}
  \frac{H_{i}}{\lambda_{i}^{3}} & \text{for }\; n=5\\
  \frac{W_{i}\ln \lambda_{i}}{\lambda_{i}^{4}} & \text{for }\; n=6\\
  0 & \text{for }\; n\geq 7
  \end{pmatrix}
  \Bigg) 
  \end{equation*}
  with positive constants $\hat c_{0},\hat c_{1},\hat c_{2},\hat b_{1},\hat d_{1}$ 
  and up to  some
$ 
  O
  (
  \sum_{r}\frac{\vert \nabla K_{r}\vert^{2}}{\lambda_{r}^{2}}
  +
  \frac{1}{\lambda_{r}^{4}}
  +
  \sum_{r\neq s}\varepsilon_{r,s}^{\frac{n+2}{n}}
  +
  \vert \partial J(u)\vert^{2}
$ 
  \end{lemma}
  \begin{proof}
Cf. Proposition \ref{I-prop_functional_at_infinity} in \cite{MM1}, see also   Proposition 5.6 in \cite{bab}
  \end{proof}
In the following we will work with the normalisation to the unit sphere, i.e. on $\{\Vert \cdot \Vert=1\}$, whereas 
in \cite{MM1} and \cite{MM2} we have been restricting to $\{k=1\}$, i.e. to the unit sphere with respect to the 
conformal $K$-volume, cf. \eqref{the_functional}. However, along an energy decreasing flow line $u$ we have 
\begin{equation*}
 0<c<J(u)=\frac{r}{k^{\frac{n-2}{n}}}=\frac{\Vert u \Vert^{2}}{k^{\frac{n-2}{n}}}<J(u_{0})
\end{equation*}
thanks to the positivity of the Yamabe invariant, cf. \eqref{Yamabe_invariant},
and hence a control of $k$ via $\Vert u \Vert$ and vice versa.  Moreover on $V(q,\varepsilon)$ there holds 
\begin{equation*}
J(u)=\frac{\alpha^{2}}{(\alpha_{K}^{\frac{2n}{n-2}})^{\frac{n-2}{n}}} +o(1)
\; \text{ and }\; 
\frac{\alpha^{2}}{\alpha_{K}^{\frac{2n}{n-2}}}K_{i}\alpha_{i}^{\frac{4}{n-2}}=1+o(1),
\end{equation*}
cf. \eqref{alpha_and_alpha_K},
whence, as an easy computation shows, we have uniform energy control on each $V(q,\varepsilon)$ via
\begin{equation*}
J(u)=\sum_{i}(K_{i}^{\frac{2-n}{2}})^{\frac{2}{n}}+o(1) \; \text{ for }\; u\in V(q,\varepsilon).
\end{equation*}
In particular the aforegoing Lemmata are still applicable, when working on $\{\Vert \cdot \Vert=1\}$
instead of $\{k=1\}$.

\section{Flow construction}\label{section_flow_construction}
\begin{lemma} 
\label{lem_shadow_flow}
For  bounded 
$\nu \in \langle \varphi_{i},\lambda_{i}\partial_{\lambda_{i}}\varphi_{i},\frac{\nabla_{a_{i}}}{\lambda_{i}}\varphi_{i}\rangle^{\perp_{L_{g_{0}}}}$
and
$$\beta_{\alpha,a,\lambda}=(\beta_{\alpha_{i}},\beta_{a_{i}},\beta_{\lambda_{i}})
\in C^{0}(V(q,\varepsilon))$$ 
there exist 
$b_{\alpha},b_{v} \in C^{0}(V(q,\varepsilon))$
with
\begin{enumerate}[label=(\roman*)]
 \item \quad 
 $
 b_{\alpha}
=
-
\frac
{
\sum_{i}\beta_{\alpha_{i}}\alpha_{i}^{2}\Vert \varphi_{i}\Vert^{2}
}
{
\sum_{i}\alpha_{i}^{2}\Vert \varphi_{i}\Vert^{2}
}
+
O(\sum_{r \neq s}\frac{1}{\lambda_{r}^{2}}+\varepsilon_{r,s}+\Vert v \Vert^{2} + \Vert \nu \Vert^{2})
;$
\item \quad 
$
b_{v} 
=
O(\Vert v \Vert (\sum_{i}\vert \beta_{\lambda_{i}}\vert + \vert \beta_{a_{i}}\vert ) \sum_{r} \varphi_{r})\in \langle \phi_{k,i}\rangle
$
\end{enumerate}
such, that moving $u=\alpha^{i}\varphi_{i}+v \in V(q,\varepsilon) \cap \{\Vert \cdot \Vert=1\}$ along
\begin{equation*}
\frac{\dot \alpha_{i}}{\alpha_{i}}=\beta_{\alpha_{i}}+b_{\alpha}
,\;
\frac{\dot \lambda_{i}}{\lambda_{i}}=\beta_{\lambda_{i}}
,\;
\lambda_{i}\dot a_{i}=\beta_{a_{i}}
\; \text{ and }\;
\partial_{t}v=b_{v}+\nu
\end{equation*}
there holds 
$
\partial_{t}\Vert u \Vert^{2}=0
\; \text{ and }\; 
v\in  \langle \varphi_{i},\lambda_{i}\partial_{\lambda_{i}}\varphi_{i},\frac{\nabla_{a_{i}}}{\lambda_{i}}\varphi_{i}\rangle^{\perp_{L_{g_{0}}}}.
$ 
\end{lemma}
\begin{remark}
Lemma \ref{lem_shadow_flow} simply tells us, that for every principal movement 
$$
(\alpha_{i},a_{i},\lambda_{i}) 
\; \text{ and } \; v
\; \text{ along } \quad 
(\beta_{\alpha_{i}},\beta_{a_{i}},\beta_{\lambda_{i}})
\; \text{ and } \quad 
\nu
$$  
we may
\begin{enumerate}[label=(\roman*)]
\item 
preserve the movement of $ (a_{i},\lambda_{i}) $ along
$ (\beta_{a_{i}},\beta_{\lambda_{i}}) $

\item modify the movement of $ (\alpha_{i},v) $
along $ (\beta_{\alpha_{i}},\nu) $ only slightly by 
$ (b_{\alpha},b_{v}) $ such, that we

\item ensure, that writing
$$
u
=
\alpha^{i}\varphi_{a_{i},\lambda_{i}}+v
=
\alpha^{i}\varphi_{i}+v
$$ 
remains compatible with the representation on 
$ V(q,\varepsilon) $  in the sense of proposition  \ref{prop_optimal_choice} and

\item preserve the norm 
$ \Vert u \Vert=\Vert u \Vert_{L_{g_{0}}} $
of 
$ u=\alpha^{i}\varphi_{i}+v$.
\end{enumerate}
\end{remark}
\begin{proof}
Let us suppose a $\langle v, \langle \phi_{k,i}\rangle\rangle_{L_{g_{0}}}=0$ orthogonality preserving evolution
\begin{equation*}
\begin{split}
\partial_{t}u
= &
\partial_{t}(\alpha^{i}\varphi_{i})+\partial_{t}v
= 
\beta_{\alpha_{i}}\alpha^{i}\varphi_{i}
+
b_{\alpha}\alpha^{i}\varphi_{i}
+
\beta_{\lambda_{i}}\alpha^{i}\lambda_{i}\partial_{\lambda_{i}}\varphi_{i}
+
\beta_{a_{i}}\alpha^{i}\frac{\nabla_{a_{i}}}{\lambda_{i}}\varphi_{i}
+
b_{v}
+
\nu \\
= &
\beta^{k,i}_{\alpha,a,\lambda}\phi_{k,i}+b_{\alpha}\alpha^{i}\varphi_{i}+b_{v}+\nu
\end{split}
\end{equation*}
exists. Then the preservation of orthogonality, i.e. 
$ 
\partial_{t}\langle v, \phi_{k,i}\rangle
=
0
$ 
for 
$\langle \cdot,\cdot \rangle = \langle \cdot,\cdot \rangle_{L_{g_{0}}}$,
necessitates
\begin{enumerate}[label=(\roman*)]
 \item in the $\alpha$-variable \quad 
$
0
\overset{!}{=}
\partial_{t}\langle v, \varphi_{i}\rangle
= 
\langle b_{v}+\nu,\varphi_{i}\rangle 
+ 
\langle 
v
,
\beta_{\lambda_{i}}\lambda_{i}\partial_{\lambda_{i}}\varphi_{i}
+
\beta_{a_{i}}\frac{\nabla_{a_{i}}}{\lambda_{i}}\varphi_{i}
\rangle 
= 
\langle b_v,\varphi_{i}\rangle 
;
$
\item in the $\lambda$-variable
 \begin{equation*}
\begin{split}
0
\overset{!}{=} &
\partial_{t}\langle v, \lambda_{i}\partial_{\lambda_{i}}\varphi_{i}\rangle
=
\langle b_{v}+\nu,\lambda_{i}\partial_{\lambda_{i}}\varphi_{i}\rangle
+
\langle
v
,
\beta_{\lambda_{i}}(\lambda_{i}\partial_{\lambda_{i}})^{2}\varphi_{i}
+
\beta_{a_{i}}\nabla_{a_{i}}\partial_{\lambda_{i}}\varphi_{i}
\rangle \\
= &
\langle b_{v},\lambda_{i}\partial_{\lambda_{i}}\varphi_{i}\rangle
+
\beta_{\lambda_{i}}
\langle
v
,
(\lambda_{i}\partial_{\lambda_{i}})^{2}\varphi_{i}
\rangle
+
\beta_{a_{i}}
\langle 
v,\nabla_{a_{i}}\partial_{\lambda_{i}}\varphi_{i}
\rangle
;
\end{split}
\end{equation*}
\item in the $a$-variable
 \begin{equation*}
\begin{split}
0
\overset{!}{=}&
\partial_{t}\langle v, \frac{\nabla_{a_{i}}}{\lambda_{i}}\varphi_{i}\rangle
=
\langle b_{v}+\nu,\frac{\nabla_{a_{i}}}{\lambda_{i}}\varphi_{i}\rangle
+
\langle
v
,
\beta_{\lambda_{i}}\lambda_{i}\partial_{\lambda_{i}}\frac{\nabla_{a_{i}}}{\lambda_{i}}\varphi_{i})
+
\beta_{a_{i}}(\frac{\nabla_{a_{i}}}{\lambda_{i}})^{2}\varphi_{i}
\rangle \\
= &
\langle b_{v},\frac{\nabla_{a_{i}}}{\lambda_{i}}\varphi_{i}\rangle
+
\beta_{\lambda_{i}}
\langle
v
,
\lambda_{i}\partial_{\lambda_{i}}\frac{\nabla_{a_{i}}}{\lambda_{i}}\varphi_{i}
\rangle
+
\beta_{a_{i}}
\langle 
v,(\frac{\nabla_{a_{i}}}{\lambda_{i}})^{2}\varphi_{i}
\rangle .
\end{split}
\end{equation*}
\end{enumerate}
Since $\vert \phi_{k,i}\vert \leq C\varphi_{i}$ pointwise and, as follows from Lemmata 
 \ref{lem_emergence_of_the_regular_part} and \ref{lem_interactions},
\begin{equation*}
\langle \phi_{k,i},\phi_{l,j} \rangle
=
c_{k,i,l,j}\delta_{k,l}\delta_{i,j}
+
O(\sum_{r\neq s}\frac{1}{\lambda_{r}^{2}}+\varepsilon_{r,s})
\; \text{ for some }\; 
c<c_{k,i,l,j}<C,
\end{equation*}
we may choose some $O(\varphi_{j})=\phi_{l,j}^{*}\in \langle \phi_{k,i} \rangle $ dual to 
$\phi_{k,i}=(\varphi_{i},\lambda_{i}\partial_{\lambda_{i}}\varphi_{i},\frac{\nabla_{a_{i}}}{\lambda_{i}}\varphi_{k,i})$ such, that 
\begin{equation}\label{phi_k_i_star_estimate}
\langle \phi_{l,j}^{*},\phi_{k,i}\rangle=\delta_{k,l}\delta_{i,j}
\; \text{ and } \; 
\Vert \phi_{k,i}^{*}-\phi_{k,i}\Vert =
O(\sum_{r\neq s}\frac{1}{\lambda_{r}^{2}}+\varepsilon_{r,s}).
\end{equation}
We then may solve $(i)$-$(iii)$ via
$
b_{v}=b^{k,i}\phi_{k,i}^{*},\; b_{1,i}=0 $ and  for $k=2,3$ 
\begin{equation}\label{precise_form_of_b}
b_{k,i}=
-
\beta_{\lambda_{i}}
\langle
v
,
\lambda_{i}\partial_{\lambda_{i}}\varphi_{k,i}
\rangle
-
\beta_{a_{i}}
\langle 
v,\frac{\nabla_{a_{i}}}{\lambda_{i}}\phi_{k,i}
\rangle.
\end{equation}
Conversely solving 
$\partial_{t}v=b_{v}+\nu$
with the above choice of $b_{v}$
we find
\begin{equation*}
\begin{split}
\langle v,\phi_{k,i}\rangle_{t=\tau}
= &
\langle v,\phi_{k,i}\rangle_{t=0}
+
\int^{\tau}_{0}
\partial_{t}\langle v,\phi_{k,i}\rangle 
=
\int^{\tau}_{0}\langle b_{v},\phi_{k,i}\rangle + \langle v,\partial_{t}\phi_{k,i}\rangle \\
= &
\int^{\tau}_{0}b_{k,i} 
+
\beta_{\lambda_{i}}
\langle v,\lambda_{i}\partial_{\lambda_{i}}\phi_{k,i}\rangle
+
\beta_{a_{i}}\langle v,\frac{\nabla_{a_{i}}}{\lambda_{i}}\phi_{k,i}\rangle 
=0.
\end{split}
\end{equation*}
Hence the statement on $b_{v}$ follows.  Therefore and in particular due to $b_{1,i}=0$  we find as well
\begin{equation*}
\begin{split}
\partial_{t}\Vert u \Vert^{2}
= & 
2
\langle 
\alpha^{i}\varphi_{i}+v
,
\frac{\dot\alpha_{i}}{\alpha_{i}}\alpha^{i}\varphi_{i}
+
\frac{\dot \lambda_{i}}{\lambda_{i}}\alpha^{i}\lambda_{i}\partial_{\lambda_{i}}\varphi_{i}
+
\lambda_{i}\dot a_{i}\alpha^{i}\frac{\nabla_{a_{i}}}{\lambda_{i}}\varphi_{i}
+
b_{v}+\nu\rangle \\
= &
2\sum_{i}\alpha_{i}^{2}(\beta_{\alpha_{i}}+b_{\alpha})\Vert \varphi_{i}\Vert^{2}
+
O(\sum_{r \neq s}\frac{1}{\lambda_{r}^{2}}+\varepsilon_{r,s}+\Vert v \Vert^{2} +\Vert \nu \Vert^{2}) ,
\end{split}
\end{equation*}
where $\Vert \cdot \Vert=\Vert \cdot \Vert_{L_{g_{0}}}$,
whence $\partial_{t}\Vert u \Vert^{2}=0$ is equivalent to putting 
\begin{equation*}
b_{\alpha}
=
-
\frac
{
\sum_{i}\beta_{\alpha_{i}}\alpha_{i}^{2}\Vert \varphi_{i}\Vert^{2}
}
{
\sum_{i}\alpha_{i}^{2}\Vert \varphi_{i}\Vert^{2}
}
+
O(\sum_{r \neq s}\frac{1}{\lambda_{r}^{2}}+\varepsilon_{r,s}+\Vert v \Vert^{2}+ \Vert \nu \Vert^{2}),
\end{equation*}
noticing that $c<\alpha_{i},\Vert \varphi_{i}\Vert<C$ on $V(q,\varepsilon)\cap \{ \Vert \cdot \Vert=1\}$.
\end{proof}

\begin{definition}\label{def_flow}
Let $ q\in \N $. Then for arbitrary constants 
\begin{equation}\label{constants}
\begin{split}
0<\varepsilon=\varepsilon_{q} \ll 1
\; \text{ and } \;
\kappa_{\alpha}=\kappa_{\alpha,q}
,\;
\kappa_{\lambda}=\kappa_{\lambda,q}
,\;
\kappa_{a}=\kappa_{a,q}
,\;
\kappa_{v}=\kappa_{v,q}\gg 1
\end{split}
\end{equation}
consider on $V(q,2\varepsilon)\cap \{\, \Vert \cdot \Vert=1\}$ the subsets
\begin{enumerate}[label=(\roman*)]
\item \quad 
$
A_{v,\kappa_{v}}=
\{
\Vert v \Vert 
\geq 
\kappa_{v}
(
\sum_{r\neq s} \frac{\vert \nabla K_{r}\vert}{\lambda_{r}}+\frac{1}{\lambda_{r}^{2}}
+
\vert 1
-
\frac{\alpha^{2}}{\alpha_{K}^{\frac{2n}{n-2}}}
K_{r}\alpha_{r}^{\frac{4}{n-2}} \vert
+
\varepsilon_{r,s}^{\frac{n+2}{2n}}
)\}
;$
 \item \quad 
 $
 A_{\alpha,\kappa_{\alpha}}
 =
 \{
\sum_{j}\vert 1-\frac{\alpha^{2}}{\alpha_{K}^{\frac{2n}{n-2}}}K_{j}\alpha_{j}^{\frac{4}{n-2}}\vert 
> \kappa_{\alpha}
(
\sum_{r\neq s}\frac{\vert \nabla K_{r}\vert}{\lambda_{r}}+\frac{1}{\lambda_{r}^{2}}
+
\varepsilon_{r,s}^{\frac{n+2}{2n}}
)
 \}
; $
\item \quad 
$
A_{a_{j},\kappa_{a}}
=
\{
\frac{\vert \nabla K_{j}\vert}{\lambda_{j}}
\geq 
\frac{\kappa_{a}}{\lambda_{j}^{2}}
\}
;$
\item \quad 
$
A_{\lambda_{j},\kappa_{\lambda}}^{\geq}
=
\{\sum_{r\neq s}\varepsilon_{r,s} \geq \frac{\kappa_{\lambda}}{\lambda_{j}^{2}}\}
;$ 
\item \quad  
$
A_{\lambda_{j},\kappa_{\lambda}}^{\leq }
=
\{\sum_{r\neq s}\varepsilon_{r,s} \leq \frac{\kappa_{\lambda}^{-1}}{\lambda_{j}^{2}}\}
$
\end{enumerate}
and to each of these subsets a  corresponding cut-off functions 
$$\eta_{\;\cdot\;}\in C^{\infty}(V(q,\varepsilon),[0,1])$$ 
satisfying 
\begin{equation*}
\begin{array}{clllllll}
(i) \quad &
\eta_{v}\lfloor_{A_{v,\kappa_{v}}} 
& = & 
1
& \; \text{ and }\; &
\eta_{v}\lfloor_{(A_{v,\sqrt{\kappa_{v}}})^{c}}
& = & 
0;
\\[1em]
(ii) \quad &  
\eta_{\alpha}\lfloor_{A_{\alpha,\kappa_{\alpha}}} 
& = & 
1
& \; \text{ and } \;  &
\eta_{\alpha}\lfloor_{(A_{\alpha,\sqrt{\kappa_{\alpha}}})^{c}}
& =& 
0;
\\[1em]
(iii) \quad &  
\eta_{a_{j}}\lfloor_{A_{a_{j},\kappa_{a}}} 
& = & 
1
& \; \text{ and } \; &
\eta_{a_{j}}\lfloor_{(A_{a_{j},\sqrt{\kappa_{a}}})^{c}}
& = & 
0;
\\[1em]
(iv) \quad &
\eta^{\geq}_{\lambda_{j}}\lfloor_{A_{\lambda_{j},\kappa_{\lambda}}^{\geq}} 
& = & 
1
& \; \text{ and } \; &
\eta^{\geq}_{\lambda_{j} }\lfloor_{(A_{\lambda_{j},\sqrt{\kappa_{\lambda}}}^{\geq})^{c}}
& = & 
0
;
\\[1em]
(v) \quad &
\eta^{\leq}_{\lambda_{j}}\lfloor_{A_{\lambda_{j},\kappa_{\lambda}}^{\leq}}
& = & 
1
& \; \text{ and } \;  &
\eta^{\leq}_{\lambda_{j} }\lfloor_{(A_{\lambda_{j},\sqrt{\kappa_{\lambda}}}^{\leq})^{c}}
& = &
0
.
\end{array}
\end{equation*}
Moreover for some monotone cut-off function 
$\eta\in C^{\infty}(\R,[0,1])$ 
with 
$$
\eta\lfloor_{(-\infty,\frac{1}{2}]}=0,
\;
\eta\lfloor_{[1,\infty)}=1
$$
define 
\begin{equation*}
m_{\lambda_{j},\lambda_{i}}=\eta(\frac{\lambda_{j}}{\lambda_{i}})\text{ and }
m_{\lambda_{j}}
=
\kappa^{\sum_{j\neq i} m_{\lambda_{j},\lambda_{i}}},\; \kappa \gg 1.
\end{equation*}
As set out in Lemma \ref{lem_shadow_flow}
we then evolve on $ V(q,2\varepsilon) $ according to 
\begin{equation*}
\begin{array}{clllllll}
(\alpha)\quad &
 \frac{\dot\alpha_{j}}{\alpha_{j}} 
& = &
\beta_{\alpha_{j}} & + & b_{\alpha}
& = &
 -
 \eta_{\alpha} (1-\eta_{v}) \Vert \varphi_{j}\Vert^{-2}
 (
 1
 -
 \frac{\alpha^{2}}{\alpha_{K}^{\frac{2n}{n-2}}}K_{j}\alpha_{j}^{\frac{4}{n-2}}
 )
 +
 b_{ \alpha};
\\[1em]
(a)\quad &
\lambda_{j}\dot a_{j}
& = &
\beta_{a_{j}} &&
& = &
(1-\eta_{\alpha})(1-\eta_{v})\eta_{a_{j}}
\frac{\nabla K_{j}}{\vert \nabla K_{j}\vert};
\\[1em]
(\lambda)\quad &
\frac{\dot \lambda_{j}}{\lambda_{j}}
& = &
\beta_{\lambda_{j}} &&
& = &
 -
 (1-\eta_{\alpha})(1-\eta_{v})
 (
\eta_{\lambda_{j}}^{\leq} \Pi_{i}(1-\eta_{a_{i}})^{m_{\lambda_{j},\lambda_{i}}}\frac{\Delta K_{j}}{\vert \Delta K_{j}\vert}
 +
 \eta_{\lambda_{j}}^{\geq}m_{\lambda_{j}}
 );
\\[1em]
(v)\quad &
\partial_{t}v
& = & 
b_{v} & + &\nu
& = &
b_{v}-C_{v}v
\\
\end{array}
\end{equation*}
with a constant $ C_{v}=C_{v,q}\gg 1 $, i.e. 
$$ 
\partial_{t}u
=
\partial_{t}(\alpha^{i}\varphi_{a_{i},\lambda_{i}})
+
\partial_{t}v
=
A_{q}(\alpha,a,\lambda,v), 
$$ 
where
\begin{equation}\label{A_alpha_a_lambda_v}
\begin{split}
A_{q}(\alpha,a,\lambda,v)
= &
- \eta_{\alpha}(1-\eta_{v})\Vert \varphi_{j}\Vert^{-2}\alpha^{j} 
(
 1
 -
 \frac{\alpha^{2}}{\alpha_{K}^{\frac{2n}{n-2}}}K_{j}\alpha_{j}^{\frac{4}{n-2}}
)
\varphi_{j} 
+
b_{\alpha}\alpha^{i}\varphi_{i}
\\
& +
(1-\eta_{\alpha})(1-\eta_{v}) \eta_{a_{j}}\alpha^{j}\langle\frac{\nabla K_{j}}{\vert \nabla K_{j}\vert},\frac{\nabla_{a_{j}}}{\lambda_{j}}\varphi_{j}\rangle \\
& - 
(1-\eta_{\alpha})(1-\eta_{v})
 \alpha^{j} 
( 
\eta_{\lambda_{j}}^{\leq}\Pi_{i}(1-\eta_{a_{i}})^{m_{\lambda_{j},\lambda_{i}}}\frac{\Delta K_{j}}{\vert \Delta K_{j}\vert}
 + 
 \eta_{\lambda_{j}}^{\geq }m_{\lambda_{j}}
 )
 \lambda_{j}\partial_{\lambda_{j}}\varphi_{j} 
 \\
 & +
 b_{v}-C_{v}v
 .
\end{split}
\end{equation}
\end{definition}

\begin{remark}\label{rem_flow_first_properties} 
\begin{enumerate}[label=(\roman*)] 
\item There holds
\begin{equation}\label{b_estimates}
b_{\alpha}
=
o(\Vert v \Vert)
+
O(\sum_{r \neq s}\frac{1}{\lambda_{r}^{2}}+\varepsilon_{r,s}),  
\end{equation}
cf. Lemma \ref{lem_shadow_flow}, since from 
\begin{equation*}
 \frac{\dot\alpha_{j}}{\alpha_{j}} 
 =
 -
 \eta_{\alpha} (1-\eta_{v}) \Vert \varphi_{j}\Vert^{-2}
 (
 1
 -
 \frac{\alpha^{2}}{\alpha_{K}^{\frac{2n}{n-2}}}K_{j}\alpha_{j}^{\frac{4}{n-2}}
 )
 +
 b_{ \alpha}
 =
 \beta_{\alpha_{j}}+b_{\alpha}
\end{equation*} 
we find 
\begin{equation*}
\sum_{i}\beta_{\alpha_{i}}\alpha_{i}^{2}\Vert \varphi_{i}\Vert^{2}
=
 0,
\end{equation*}
cf.  \eqref{alpha_and_alpha_K}. Secondly,  since
$
\nu=-C_{v}v
$, we find $\Vert \nu\Vert^{2}=o(\Vert v \Vert)$, provided 
$$\Vert \nu \Vert =C_{v}\Vert v \Vert\leq C\varepsilon=o(1),$$ 
which we may  assume as $\varepsilon \longrightarrow \infty.$ 

\item 
The  purpose of introducing $C_{v}$ is to obtain
$\partial_{t} \Vert v \Vert^{2}\leq 0$
and in fact  $b_{v}$ does not depend on $C_{v}$.  Moreover from \eqref{A_alpha_a_lambda_v} and \eqref{b_estimates} it is clear, that 
$$
\partial_{t}u\geq - C\alpha^{i}\varphi_{i}-C_{v}
$$
for a universal  $C_{\kappa,q}>0$ independent of $C_{v}$. Hence 
$\partial_{t}u\geq -C_{v}u$
by fixing $C_{v}\gg C$.

\item Whereas the movement in $\alpha$ and $a$ are obviously well defined, we note, that
 \begin{equation*}
\begin{split}
\{\eta_{\lambda_{j}}^{\leq}\Pi_{i}(1-\eta_{a_{i}})^{m_{\lambda_{j},\lambda_{i}}}\neq 0\}
\subseteq &
\{\eta_{a_{j}}\neq 1\}
=
\{\vert \nabla K_{j}\vert \leq \frac{\kappa_{a}}{\lambda_{j}}\},
\end{split}
 \end{equation*}
whence by non degeneracy, cf. \eqref{nd},
 also the movement in $\lambda$ is well defined.
\item 
The union $\cup A_{\cdot,\kappa_{\cdot}}$ covers $V(q,2\varepsilon)$ for $\varepsilon>0$ sufficiently small. Indeed 
we have 
\begin{equation*}
\forall \;1\leq i,j \leq q\;:\;
\frac{1}{\lambda_{i}^{2}}\simeq \sum_{r\neq s}\varepsilon_{r,s}\simeq \frac{1}{\lambda_{j}^{2}}
\; \text{ and }\; 
\vert \nabla K_{i}\vert \lesssim \frac{1}{\lambda_{i}}
\end{equation*}
on
$
A^{c}=V(q,2\varepsilon)\setminus \cup A_{\cdot,\kappa_{\cdot}} 
$, 
whence $\frac{1}{\lambda_{i}^{2}}\simeq (\frac{1}{\lambda_{i}\lambda_{j}})^{\frac{n-2}{2}}\simeq \frac{1}{\lambda_{j}^{2}}$, a contradiction for $n \geq 5$.
\item 
In view of Lemmata  \ref{lem_alpha_derivatives_at_infinity},   \ref{lem_lambda_derivatives_at_infinity} and  \ref{lem_a_derivatives_at_infinity} we call 
\begin{enumerate}
\item[($\alpha$)] \quad
$P(\partial J(u)\varphi_{j})=(1-\frac{\alpha^{2}}{\alpha_{K}^{\frac{2n}{n-2}}}K_{j}\alpha_{j}^{\frac{4}{n-2}});$ 
\item[($\lambda$)] \quad 
$P(\partial J(u)\lambda_{j}\partial_{\lambda_{j}})=\frac{\Delta K_{j}}{K_{j}\lambda_{j}^{2}},\;\sum_{j\neq i}\lambda_{j}\partial_{\lambda_{j}}\varepsilon_{i,j};$
\item \quad
$P(\partial J(u)\frac{\nabla_{a_{j}}}{\lambda_{j}})=\frac{\nabla K_{j}}{K_{j}\lambda_{j}}$
\end{enumerate}
the principal terms in $\alpha, \lambda$ and $a$. The flow is designed in such a way, that whenever the principal terms are dominant in the expansion of the lemmata above, their corresponding movement in $\alpha,\lambda$ and $a$ will decrease energy. Notably we move $\lambda=\max \lambda_{i}$ as little as possible by the Laplacian of $K$.
\end{enumerate}
\end{remark}
From Lemma \ref{lem_shadow_flow} and  Definition \ref{def_flow}, see \eqref{A_alpha_a_lambda_v} in particular, we may generate a flow. 
\begin{lemma}\label{easy_properties}
On
$
X=\{0\leq u \in W^{1,2}(M) \mid u\neq 0,\; \Vert u \Vert=1\}
$
the evolution 
\begin{equation*}
\partial_{t}u=\eta \sum_{q}A_{q}(\alpha,a,\lambda,v)-(1-\eta )\nabla J(u),\; \nabla=\nabla_{L_{g_{0}}}
\end{equation*}
for  a  cut-off function $\eta\in C^{\infty}_{loc}(X,[0,1])$ satisfying 
\begin{equation*}
\eta\lfloor_{\cup_{q}V(q,\varepsilon_{q})}=1 \; \text{ and } \;\eta \lfloor_{X\setminus \cup_{q}V(q,2\varepsilon_{q})}=0  
\end{equation*}
induces a semi-flow
 \begin{equation*}
\Phi:\R_{\geq 0} \times X\longrightarrow X,
 \end{equation*}
i.e. the flow exists for all times, remains non negative and preserves $\Vert \cdot \Vert=1$, 
provided  
$$
\forall\; q\in \N \;:\;
\kappa_{v}\gg 1
,\;
\kappa_{\alpha}\gg 1
,\;
\kappa_{a}\gg \kappa_{\lambda}^{2}
,\;
\sqrt{\kappa_{\lambda}}\gg \kappa^{q^{2}}
,\; 
\kappa \gg 1 \; \text{ and } \; 
0<\varepsilon \ll 1,
$$
cf. \eqref{constants}.
\end{lemma}
\begin{proof}
 Since $\nabla J=\nabla_{L_{g_{0}}} J$ is the $L_{g_{0}}$-gradient, we may write
$ 
\partial_{t}u=f(u)
 $ 
and $f$ is a locally smooth vectorfield on $W^{1,2}(M,\R_{\geq 0}) \setminus\{0\}$. So we have short time existence and due to 
\begin{equation*}
\partial_{t} \Vert u \Vert^{2}
=
-2(1-\eta )\langle \nabla J(u),u \rangle_{L_{g_{0}}}
+
2\eta \langle \sum_{q}A_{q}(\alpha,a,\lambda,v),u\rangle_{L_{g_{0}}}
=0,
\end{equation*}
as $\langle A_{q}(\alpha,a,\lambda,v),u\rangle_{L_{g_{0}}}=0$ on $ V(q,2\varepsilon) $ by construction and 
\begin{equation*}
\langle \nabla J(u),u \rangle_{L_{g_{0}}}
=
\partial J(u)u=0
\end{equation*}
by scaling invariance of $J$,
the $L_{g_{0}}$-norm  is preserved.   Moreover, as Proposition \ref{prop_decreasing_the_energy} will show, there holds
$ 
\partial_{t}J(u)\leq 0,
$ 
whence in combination with the positivity of the Yamabe invariant we have 
\begin{equation}\label{uniform_energy_bound}
c_{K}
<
J(u)
=
\frac{\int L_{g_{0}}uud\mu_{g_{0}}}{(\int Ku^{\frac{2n}{n-2}}d\mu_{g_{0}})^{\frac{n-2}{n}}}
=
\frac{\Vert u \Vert^{2}}{k^{\frac{n-2}{n}}}
=
\frac{1}{k^{\frac{n-2}{n}}}
<
J(u_{0})
\end{equation}
along each flow line for its time of existence.  Hence $\nabla J$ and thus $f$  are uniformly bounded along each flow line and, as is easy to see, locally smooth. Therefore every flow line exists for all times. 
Moreover, as $\vert \phi_{k,i}\vert \lesssim \varphi_{i}$, we find
from \eqref{A_alpha_a_lambda_v}, \eqref{b_estimates} and (ii) of Lemma \ref{lem_shadow_flow},
that
\begin{equation*}
A_{q}(\alpha,a,\lambda,v)
\geq 
-C (\alpha^{i}\varphi_{i})-C_{v}v \geq C_{v}u
\; \text{ on } \; V(q,2\varepsilon) 
\; \text{ for some }\; C>0
\end{equation*}
 provided $C_{v}>C$, i.e. $C_{v}>0$ is sufficiently large. Hence
\begin{equation*}
\nabla J(u)=L_{g_{0}}^{-}(\partial J(u))=\frac{2L_{g_{0}}^{-}}{k^{\frac{n-2}{n}}}
(L_{g_{0}}u-\frac{r}{k}Ku^{\frac{n+2}{n-2}})
<
\frac{2u}{k^{\frac{n-2}{n}}}
\end{equation*}
by definition of the $L_{g_{0}}$-gradient and positivity of $L_{g_{0}}^{-}\simeq G_{g_{0}}>0$.  
We conclude $\partial_{t}u >-cu$ using \eqref{uniform_energy_bound}, 
so every initially non negative or positive  flow line becomes or remains  positive for all times
 to come. 
\end{proof}
We  point out, that the long time existence part is not critical, since the flow is based on a strong gradient, i.e. the gradient corresponding  to the metric on the variational space, and thus falls into the class of ordinary differential equations, cf. \cite{bc} or \cite{BenAyed_Ahmeou},  in contrast to 
Yamabe type flows as in \cite{Brendle} or \cite{may-cv}. 
\begin{proposition}\label{prop_decreasing_the_energy}
Under $u=\Phi(\cdot,u_{0})$ there holds
\begin{equation*}
\partial_{t}J(u)
\leq
-
c_{q}
(
\sum_{r\neq s}
\frac{\vert \nabla K_{r}\vert ^{2}}{\lambda_{r}^{2}}
+
\frac{1}{\lambda_{r}^{4}}
+
\vert 1
-
\frac{\alpha^{2}}{\alpha_{K}^{\frac{2n}{n-2}}}
K_{r}\alpha_{r}^{\frac{4}{n-2}} \vert^{2}
+
\varepsilon_{r,s}^{\frac{n+2}{n}}
+
\Vert v \Vert^{2})
\; \text{ on }\; V(q,2\varepsilon)
\end{equation*}
for some $ c_{q}>0$, 
provided  
$$
\kappa_{v}\gg 1
,\;
\kappa_{\alpha}\gg 1
,\;
\kappa_{a}\gg \kappa_{\lambda}^{2}
,\;
\sqrt{\kappa_{\lambda}}\gg \kappa^{q^{2}}
,\; 
\kappa \gg 1 \; \text{ and } \; 
0<\varepsilon \ll 1,
$$
cf. \eqref{constants}, while 
$
\partial_{t}J(u)
=
-\vert \nabla J(u)\vert^{2}
\; \text{ on } \; (\cup_{q} V(q,2\varepsilon))^{c}.
$ 
\end{proposition}

\begin{proof}
First consider the flow 
$ \partial_{t}u=A_{q}(\alpha,a,\lambda,v) $ 
on  $V(q,2\varepsilon)$ for some $q\in \N$. We then have  
\begin{equation*}
\begin{split}
I_{v}
= & 
\partial J(u) \partial_{t}v=\partial J(u)(-C_{v}v+b_{v}).
\end{split}
\end{equation*}
From \eqref{phi_k_i_star_estimate} and \eqref{precise_form_of_b} we have  
\begin{equation*}
\partial J(u)b_{v}
=
O(\Vert v \Vert (\sum_{i}\vert \beta_{\lambda_{i}}\vert + \vert \beta_{a_{i}}\vert ))
\left(
\sum_{i}(\vert \partial J(u)\phi_{2,i}\vert + \vert \partial J(u)\phi_{3,i}\vert)
+
O(\vert\partial J(u)\vert )
(\sum_{r\neq s}\frac{1}{\lambda_{r}^{2}}+\varepsilon_{r,s})
\right),
\end{equation*}
whence by virtue of Lemmata 
\ref{lem_lambda_derivatives_at_infinity},  \ref{lem_a_derivatives_at_infinity}
and Proposition \ref{prop_gradient_bounds} 
\begin{equation*}
\partial J(u)b_{v}
=
O
(
\sum_{r\neq s} \frac{\vert \nabla K_{r}\vert}{\lambda_{r}}+\frac{1}{\lambda_{r}^{2}}
+
\varepsilon_{r,s}
+
\vert 
1
-
\frac{\alpha^{2}}{\alpha_{K}^{\frac{2n}{n-2}}}
K_{r}\alpha_{r}^{\frac{4}{n-2}} 
\vert^{2}
+
\Vert v \Vert^{2}
) 
\Vert v \Vert
\end{equation*}
and thus
\begin{equation*}
\partial J(u)b_{v}
=
O(
\sum_{r}\vert 
1
-
\frac{\alpha^{2}}{\alpha_{K}^{\frac{2n}{n-2}}}
K_{r}\alpha_{r}^{\frac{4}{n-2}} 
\vert^{3}
+
\Vert v \Vert^{3})
+
O
(
\sum_{r\neq s} \frac{\vert \nabla K_{r}\vert}{\lambda_{r}}+\frac{1}{\lambda_{r}^{2}}
+
\varepsilon_{r,s}
)
\Vert v \Vert.
\end{equation*}
Moreover
by the well known positivity of 
$\partial^{2}J(\alpha^{i}\varphi_{i})>0$
on 
$H_{u}(q,2\varepsilon)=\langle \phi_{k,i}\rangle^{\perp_{L_{g_{0}}}}$  we find by expansion 
\begin{equation}\label{partial_J_u_v} 
\begin{split}
-\partial J(u)v 
\simeq &  
-\Vert v \Vert^{2}
+
O
(
\sum_{r\neq s} 
\frac{\vert \nabla K_{r}\vert^{2}}{\lambda_{r}^{2}}
+
\frac{1}{\lambda_{r}^{4}}
+
\varepsilon_{r,s}^{\frac{n+2}{n}}
),
\end{split}
\end{equation}
where we made use of  Lemma \ref{lem_testing_with_v}. 
Hence 
\begin{equation}\label{I_v_estimate_final} 
\begin{split}
I_{v}
\simeq &
-\Vert v \Vert^{2}
+
O(
\sum_{r}\vert 
1
-
\frac{\alpha^{2}}{\alpha_{K}^{\frac{2n}{n-2}}}
K_{r}\alpha_{r}^{\frac{4}{n-2}} 
\vert^{3}
+
\Vert v \Vert^{3})
+
O
(
\sum_{r\neq s} 
\frac{\vert \nabla K_{r}\vert^{2}}{\lambda_{r}^{2}}
+
\frac{1}{\lambda_{r}^{4}}
+
\varepsilon_{r,s}^{\frac{n+2}{n}}
).
\end{split}
\end{equation}
Secondly for 
\begin{equation*}
\begin{split}
I_{\alpha}
= &
\frac{\dot \alpha^{i}}{\alpha_{i}}\partial J(u)\alpha_{i}\varphi_{i}
=
\partial J(u)
(
- \eta_{\alpha}(1-\eta_{v})\Vert \varphi_{j}\Vert^{-2}\alpha^{j} 
(
 1
 -
 \frac{\alpha^{2}}{\alpha_{K}^{\frac{2n}{n-2}}}K_{j}\alpha_{j}^{\frac{4}{n-2}}
)
\varphi_{j} 
+
b_{\alpha}\alpha^{i}\varphi_{i}
)
\end{split} 
\end{equation*}
we have due to $\partial J(u)u=0$ by scaling invariance of $J$ and due to 
\eqref{b_estimates}, \eqref{partial_J_u_v}  
\begin{equation*}
b_{\alpha}\partial J(u)\alpha^{i}\varphi_{i}
= 
O(\sum_{r \neq s}\frac{1}{\lambda_{r}^{2}}+\varepsilon_{r,s}+\Vert v \Vert )
\partial J(u)v
=
O(\Vert v \Vert^{3} )
+
O
(
\sum_{r\neq s} 
\frac{\vert \nabla K_{r}\vert^{2}}{\lambda_{r}^{2}}
+
\frac{1}{\lambda_{r}^{4}}
+
\varepsilon_{r,s}^{\frac{n+2}{n}}
).
\end{equation*}
Moreover Lemma  \ref{lem_alpha_derivatives_at_infinity}  and Proposition \ref{prop_gradient_bounds} show
\begin{equation*}
\begin{split}
- \Vert \varphi_{j}\Vert^{-2}\alpha^{j} 
(
 1
 -
 \frac{\alpha^{2}}{\alpha_{K}^{\frac{2n}{n-2}}}K_{j}\alpha_{j}^{\frac{4}{n-2}}
)
\partial J(u)
\varphi_{j} 
\lesssim & 
-
\sum_{j}
\vert
1
 -
 \frac{\alpha^{2}}{\alpha_{K}^{\frac{2n}{n-2}}}K_{i}\alpha_{i}^{\frac{4}{n-2}}
 \vert^{2}   
 \\
 & + 
 O
(
\sum_{r\neq s} 
\frac{1}{\lambda_{r}^{4}}
+
\varepsilon_{r,s}^{2}
+
\vert 
1
-
\frac{\alpha^{2}}{\alpha_{K}^{\frac{2n}{n-2}}}
K_{r}\alpha_{r}^{\frac{4}{n-2}} 
\vert^{4}
+
\Vert v \Vert^{4}
)
\end{split} 
\end{equation*} 
and we obtain  
\begin{equation}\label{I_alpha_estimate_final}
\begin{split}
I_{\alpha}
\lesssim &
-
\eta_{\alpha} (1-\eta_{v})\sum_{j}\vert 1
 -
 \frac{\alpha^{2}}{\alpha_{K}^{\frac{2n}{n-2}}}K_{i}\alpha_{i}^{\frac{4}{n-2}}
 \vert^{2}  \\
& + 
O(
\sum_{r}\vert 
1
-
\frac{\alpha^{2}}{\alpha_{K}^{\frac{2n}{n-2}}}
K_{r}\alpha_{r}^{\frac{4}{n-2}} 
\vert^{4}
+
\Vert v \Vert^{3})
+
O
(
\sum_{r\neq s} 
\frac{\vert \nabla K_{r}\vert^{2}}{\lambda_{r}^{2}}
+
\frac{1}{\lambda_{r}^{4}}
+
\varepsilon_{r,s}^{\frac{n+2}{n}}
).
\end{split} 
\end{equation}
Therefore  combining \eqref{I_v_estimate_final} and \eqref{I_alpha_estimate_final} 
\begin{equation*}
\begin{split}
I_{\alpha}+I_{v}
\lesssim &
-
\Vert v \Vert^{2}
-
\eta_{\alpha} (1-\eta_{v})\sum_{j}\vert 1
 -
 \frac{\alpha^{2}}{\alpha_{K}^{\frac{2n}{n-2}}}K_{i}\alpha_{i}^{\frac{4}{n-2}}
 \vert^{2}   \\
& +
O(
\sum_{r}\vert 
1
-
\frac{\alpha^{2}}{\alpha_{K}^{\frac{2n}{n-2}}}
K_{r}\alpha_{r}^{\frac{4}{n-2}} 
\vert^{3}
+
\Vert v \Vert^{3})
+
O
(
\sum_{r\neq s} 
\frac{\vert \nabla K_{r}\vert^{2}}{\lambda_{r}^{2}}
+
\frac{1}{\lambda_{r}^{4}}
+
\varepsilon_{r,s}^{\frac{n+2}{n}}
)
\end{split}
\end{equation*} 
and recalling  the definitions of $\eta_{v}$ and $\eta_{\alpha}$ from Definition \ref{def_flow} we conclude
\begin{equation}\label{I_alpha_+_I_v_final_estimate}
\begin{split}
I_{\alpha}+I_{v}
\lesssim &
-
\Vert v \Vert^{2}
-
\eta_{\alpha} (1-\eta_{v})\sum_{j}\vert 1
 -
 \frac{\alpha^{2}}{\alpha_{K}^{\frac{2n}{n-2}}}K_{i}\alpha_{i}^{\frac{4}{n-2}}
 \vert^{2}   \\
& + 
(1-\eta_{\alpha})(1-\eta_{v})
O
(
\sum_{r\neq s} \frac{\vert \nabla K_{r}\vert^{2}}{\lambda_{r}^{2}}+\frac{1}{\lambda_{r}^{4}}
+
\varepsilon_{r,s}^{\frac{n+2}{n}}
).
\end{split} 
\end{equation}
We turn to the $\lambda$ and $a$ evolution.  By Proposition \ref{prop_gradient_bounds} and the  definitions of $\eta_{v}$ and $\eta_{\alpha}$ we find  up to some 
\begin{equation*}
(1-\eta_{\alpha})(1-\eta_{v})
O
(
\sum_{r\neq s} \frac{\vert \nabla K_{r}\vert^{2}}{\lambda_{r}^{2}}+\frac{1}{\lambda_{r}^{4}}
+
\varepsilon_{r,s}^{\frac{n+2}{n}}
)  
\end{equation*}
from  Lemma \ref{lem_lambda_derivatives_at_infinity} the relation
\begin{equation*}
\begin{split}
I_{\lambda}
=
&
\frac{\dot \lambda^{j}}{\lambda_{j}} \partial   J(u)
 \alpha_{j} \lambda_{j}\partial_{\lambda_{j}}\varphi_{j} 
= 
-
 (1-\eta_{\alpha})(1-\eta_{v})
 (
\eta_{\lambda_{j}}^{\leq} \Pi_{i}(1-\eta_{a_{i}})^{m_{\lambda_{j},\lambda_{i}}}\frac{\Delta K_{j}}{\vert \Delta K_{j}\vert}
 +
 \eta_{\lambda_{j}}^{\geq}m_{\lambda_{j}}
 )
 \\
&\quad\quad\quad\quad\quad\quad\quad\quad\quad\quad\;\;\,
(
\frac{\alpha^{j}\alpha_{j}}{(\alpha_{K}^{\frac{2n}{n-2}})^{\frac{n-2}{n}}}
(
\tilde c_{2}\frac{\Delta K_{j}}{K_{j}\lambda_{j}^{2}} 
-
\tilde b_{2}\sum_{j\neq i }\frac{\alpha_{i}}{\alpha_{j}}\lambda_{j}\partial_{\lambda_{j}}\varepsilon_{i,j}
+
\tilde  d_{1}
\begin{pmatrix}
\frac{H_{j}}{\lambda_{j}^{3 }} &\text{for }\; n=5\\
\frac{W_{j}\ln \lambda_{j}}{\lambda_{j}^{4}} &\text{for }\; n=6\\
0 &\text{for }\; n\geq 7
\end{pmatrix} 
)
) \\
= & 
 -
 \tilde c_{2} (1-\eta_{\alpha})(1-\eta_{v}) \sum_{j}
\eta_{\lambda_{j}}^{\leq} \Pi_{i}(1-\eta_{a_{i}})^{m_{\lambda_{j},\lambda_{i}}} 
\frac{\alpha_{j}^{2}}{(\alpha_{K}^{\frac{2n}{n-2}})^{\frac{n-2}{n}}}
(
\frac{\vert \Delta K_{j}\vert}{K_{j}\lambda_{j}^{2}} 
+
o(\frac{1}{\lambda_{j}^{2}})
+
O(\sum_{j\neq i}\varepsilon_{i,j})
) \\ 
& 
 +\tilde b_{2}
 (1-\eta_{\alpha})(1-\eta_{v}) 
 \sum_{j\neq i}\eta_{\lambda_{j}}^{\geq} m_{\lambda_{j}} 
(\frac{\alpha_{i}\alpha_{j}}{(\alpha_{K}^{\frac{2n}{n-2}})^{\frac{n-2}{n}}}\lambda_{j}\partial_{\lambda_{j}}\varepsilon_{i,j}
+
O(\frac{1}{\lambda_{j}^{2}}) )
\end{split}
\end{equation*}
and likewise  from Lemma \ref{lem_a_derivatives_at_infinity}
\begin{equation*}
\begin{split}
I_{a}
= &
\lambda_{j}\dot a^{j}\partial J(u)\alpha_{j}\frac{\nabla_{a_{j}}}{\lambda_{j}}\varphi_{j} 
=  
-(1-\eta_{\alpha})(1-\eta_{v})\eta_{a_{j}}\frac{\alpha^{j}\alpha_{j}}{(\alpha_{K}^{\frac{2n}{n-2}})^{\frac{n-2}{n}}}\frac{\nabla K_{j}}{\vert \nabla K_{j}\vert} 
\big(
\dot{c}_{3}\frac{\nabla K_{j}}{K_{j}\lambda_{j}}
+
O(
\frac{1}{\lambda_{j}^{3}}+\sum_{j\neq i}\varepsilon_{i,j})
\big).
\end{split}
\end{equation*} 
Thus recalling $(iii)$ of Remark \ref{rem_flow_first_properties}  and the definition of $\eta_{a_{j}}$
we have for some positive constants  $c_{i}>0$
\begin{equation}\label{dt_J_u_general_version}
\begin{split}
I_{\lambda}+I_{a}
\leq &
-
c_{1}(1-\eta_{\alpha})(1-\eta_{v})\sum_{j} \eta_{a_{j}}
\frac{\alpha_{j}^{2}}{(\alpha_{K}^{\frac{2n}{n-2}})^{\frac{n-2}{n}}}
(\frac{\vert \nabla K_{j}\vert }{K_{j}\lambda_{j}}
+
O(\sum_{j\neq i}\varepsilon_{i,j}))
\\
&
 -
 c_{2} (1-\eta_{\alpha})(1-\eta_{v}) \sum_{j}
\eta_{\lambda_{j}}^{\leq} \Pi_{i}(1-\eta_{a_{i}})^{m_{\lambda_{j},\lambda_{i}}} 
\frac{\alpha_{j}^{2}}{(\alpha_{K}^{\frac{2n}{n-2}})^{\frac{n-2}{n}}}
(
\frac{\vert \Delta K_{j}\vert}{K_{j}\lambda_{j}^{2}} 
+
O(\sum_{j\neq i}\varepsilon_{i,j})
) \\ 
&
+
c_{3}
 (1-\eta_{\alpha})(1-\eta_{v}) \sum_{j\neq i}\eta_{\lambda_{j}}^{\geq} m_{\lambda_{j}} 
(\frac{\alpha_{i}\alpha_{j}}{(\alpha_{K}^{\frac{2n}{n-2}})^{\frac{n-2}{n}}}\lambda_{j}\partial_{\lambda_{j}}\varepsilon_{i,j}
+
O(\frac{1}{\lambda_{j}^{2}}) )
\end{split}
\end{equation}
up to some $(1-\eta_{\alpha})(1-\eta_{v})
O
(
\sum_{r\neq s} \frac{\vert \nabla K_{r}\vert^{2}}{\lambda_{r}^{2}}+\frac{1}{\lambda_{r}^{4}}
+
\varepsilon_{r,s}^{\frac{n+2}{n}}
)$. 
Let us now suppose
\begin{equation*}
\exists \; 1\leq i \leq q\;:\;\frac{\vert \nabla K_{i}\vert}{\lambda_{i}}
=
\max_{j}\frac{\vert \nabla K_{j}\vert}{\lambda_{j}}
\geq \frac{\kappa_{a}}{2\underline \lambda^{2}} 
,
\; \text{ where }\; \underline \lambda=\min_{i}\lambda_{i}.
\end{equation*}
In particular we may assume  $\eta_{a_{i}}\geq \frac{1}{2}$, cf.  Definition \ref{def_flow}, and this  implies 
\begin{equation*}
\exists\; 1\leq i \leq q \;:\; 
\sum_{j}\eta_{a_{j}}\frac{\vert \nabla K_{j}\vert}{\lambda_{j}}
\simeq
\max_{j}\{\eta_{a_{j}}\frac{\vert \nabla K_{j}\vert}{\lambda_{j}}\}
\simeq 
\frac{\vert \nabla K_{i}\vert}{\lambda_{i}}
\simeq 
\sum_{j}\frac{\vert \nabla K_{j}\vert}{\lambda_{j}}.
\end{equation*}
We thus  infer from \eqref{dt_J_u_general_version}, that 
up to some 
\begin{equation*}
(1-\eta_{\alpha})(1-\eta_{v})
O
(
\sum_{r\neq s} \frac{\vert \nabla K_{r}\vert^{2}}{\lambda_{r}^{2}}+\frac{1}{\lambda_{r}^{4}}
+
\varepsilon_{r,s}^{\frac{n+2}{n}}
) 
\end{equation*}
and with possibly different constants $c_{i}>0$
\begin{equation}\label{I_lambda_I_a_first_estimate}
\begin{split}
I_{\lambda}+I_{a}
\leq &
-
c_{1}(1-\eta_{\alpha})(1-\eta_{v}) 
(
\sum_{j}\frac{\vert \nabla K_{j}\vert}{K_{j}\lambda_{j}}
+
\frac{\kappa_{a}+O(\sum_{r}m_{\lambda_{r}})}{\underline \lambda^{2}}
+
O(\sum_{r \neq s}\varepsilon_{r,s})
)
\\
&
+
c_{3}
 (1-\eta_{\alpha})(1-\eta_{v}) \sum_{j\neq i}\eta_{\lambda_{j}}^{\geq} m_{\lambda_{j}} 
\frac{\alpha_{i}\alpha_{j}}{(\alpha_{K}^{\frac{2n}{n-2}})^{\frac{n-2}{n}}}\lambda_{j}\partial_{\lambda_{j}}\varepsilon_{i,j} \\
\leq &
-
c_{1}(1-\eta_{\alpha})(1-\eta_{v})
(
\sum_{j}\frac{\vert \nabla K_{j}\vert}{K_{j}\lambda_{j}}
+
\frac{\kappa_{a}+O(\sum_{r}\kappa_{\lambda}m_{\lambda_{r}})}{\underline \lambda^{2}}
+
O(\sum_{r\neq s}\varepsilon_{r,s})
)
\\
&
+
c_{3}
 (1-\eta_{\alpha})(1-\eta_{v}) 
 \sum_{j\neq i} m_{\lambda_{j}} 
\frac{\alpha_{i}\alpha_{j}}{(\alpha_{K}^{\frac{2n}{n-2}})^{\frac{n-2}{n}}}\lambda_{j}\partial_{\lambda_{j}}\varepsilon_{i,j}
\end{split}
\end{equation}
recalling the definition of $\eta^{\geq }_{\lambda_{j}}$, cf. Definition \ref{def_flow}, for the last inequality. 
Note, that
 \begin{equation}\label{lambda_derivative_epsilon_i_j}
 -\lambda_{j}\partial_{\lambda_{j}}\varepsilon_{i,j}
 =
 \frac{n-2}{2}
 \frac
 {
 \frac{\lambda_{j}}{\lambda_{i}}
 -
 \frac{\lambda_{i}}{\lambda_{j}}
 +
 \lambda_{i}\lambda_{j}\gamma_{n}G_{g_{0}}^{\frac{2}{2-n}}(a_{i},a_{j})
 }
 {
 (
 \frac{\lambda_{j}}{\lambda_{i}}
 +
 \frac{\lambda_{i}}{\lambda_{j}}
 +
 \lambda_{i}\lambda_{j}\gamma_{n}G_{g_{0}}^{\frac{2}{2-n}}(a_{i},a_{j})
 )^{\frac{n}{2}}
 },
 \end{equation}
cf.  \eqref{eq:eij},
and recalling Definition \ref{def_flow}, there holds 
\begin{enumerate}[label=(\roman*)]
 \item \quad $m_{\lambda_{j}}\geq m_{\lambda_{i}}$ for $\lambda_{j}\geq \lambda_{i};$
\item \quad $m_{\lambda_{j}}\geq \kappa  m_{\lambda_{i}} $ for $\;\text{ for }\; \lambda_{j}\geq 2\lambda_{i}$.
\end{enumerate}
Therefore
 \begin{equation*}
\begin{split}
 \sum_{j\neq i} m_{\lambda_{j}} 
\alpha_{i}\alpha_{j}\lambda_{j}\partial_{\lambda_{j}}\varepsilon_{i,j}
\lesssim
-\sum_{i \neq  j}m_{\lambda_{i}}\varepsilon_{i,j} 
\end{split}
 \end{equation*}
 and, since $m_{\lambda_{r}}\geq \kappa \gg 1$, cf. Definition \ref{def_flow},  plugging this into \eqref{I_lambda_I_a_first_estimate} we conclude
 \begin{equation}\label{I_lambda_I_a_first_estimate_final}
\begin{split}
I_{\lambda}+I_{a}
\lesssim&
-
(1-\eta_{\alpha})(1-\eta_{v})
(
\sum_{j}\frac{\vert \nabla K_{j}\vert}{K_{j}\lambda_{j}}
+
\frac{1}{\lambda_{j}^{2}}
+
\sum_{j\neq i}\varepsilon_{i,j}
+
O(
\sum_{r\neq s} \frac{\vert \nabla K_{r}\vert^{2}}{\lambda_{r}^{2}}+\frac{1}{\lambda_{r}^{4}}
+
\varepsilon_{r,s}^{\frac{n+2}{n}}
)),
\end{split}
\end{equation}
 provided  
$\kappa_{a}\gg \sum_{r}\kappa_{\lambda}m_{\lambda_{r}}\gg 1$.
We now assume contrarily and in addition 
\begin{equation*}
\forall\; 1\leq j \leq q\;:\; \frac{\vert \nabla K_{j}\vert}{\lambda_{j}}
<
\frac{\kappa_{a}}{2\underline \lambda^{2}}
\; \text{ and } \; 
\sum_{r\neq s}  \varepsilon_{r,s}
<
\frac{1}{\kappa_{\lambda}\max_{r}m_{\lambda_{r}}\underline \lambda^{2}} . 
\end{equation*}  
Then from \eqref{dt_J_u_general_version} and recalling the definition of $\eta_{\lambda_{\cdot}}^{\geq }$
we find with possibly different constants $c_{i}>0$
\begin{equation*} 
\begin{split}
I_{\lambda}+I_{a}
\leq &
-
c_{1}(1-\eta_{\alpha})(1-\eta_{v})
(
\sum_{j} \eta_{a_{j}}
\frac{\vert \nabla K_{j}\vert }{K_{j}\lambda_{j}} 
+
O(\max_{r}m_{\lambda_{r}}\sum_{r\neq s}\varepsilon_{r,s} )
)
\\
& -
 c_{2} (1-\eta_{\alpha})(1-\eta_{v}) \sum_{j}
\eta_{\lambda_{j}}^{\leq} \Pi_{i}(1-\eta_{a_{i}})^{m_{\lambda_{j},\lambda_{i}}}  
\frac{\vert \Delta K_{j}\vert}{K_{j}\lambda_{j}^{2}} 
)
\end{split}
\end{equation*}
up to some $(1-\eta_{\alpha})(1-\eta_{v})
O
(
\sum_{r\neq s} \frac{\vert \nabla K_{r}\vert^{2}}{\lambda_{r}^{2}}+\frac{1}{\lambda_{r}^{4}}
+
\varepsilon_{r,s}^{\frac{n+2}{n}}
)$.
Note, that 
\begin{equation*}
\sum_{r\neq s}  \varepsilon_{r,s}
\leq 
\frac{1}{\kappa_{\lambda}\max_{r}m_{\lambda_{r}}\underline \lambda^{2}}  
\quad \Longrightarrow \quad
\sum_{r\neq s}\varepsilon_{r,s}
\leq 
\max_{r}m_{\lambda_{r}}
\sum_{r\neq s}\varepsilon_{r,s}
\leq 
\frac{1}{\kappa_{\lambda}\underline \lambda^{2}}
\end{equation*}
and in particular $\eta_{\underline \lambda}^{\leq}=1$. Hence up to the same error as above
\begin{equation*} 
\begin{split}
I_{\lambda}+I_{a}
\lesssim &
-
(1-\eta_{\alpha})(1-\eta_{v})
(
\sum_{j} \eta_{a_{j}}
\frac{\vert \nabla K_{j}\vert }{K_{j}\lambda_{j}} 
+
\frac{ \Pi_{i}(1-\eta_{a_{i}})^{m_{\underline \lambda ,\lambda_{i}}}   }{\underline{\lambda}^{2}}
+
O(\frac{1}{\kappa_{\lambda}\underline{\lambda}^{2}})
).
\end{split}
\end{equation*} 
Recalling Definition \ref{def_flow} there holds  
\begin{equation*}
\Pi_{i}(1-\eta_{a_{i}})^{m_{\underline \lambda,\lambda_{i}}}
=o(1)
\Longrightarrow
\exists  \; \underline \lambda \leq \lambda_{i}\leq 2 \underline \lambda
\; : \; 
(1-\eta_{a_{i}})=o(1)
\end{equation*}
and  in the latter case $\eta_{a_{i}}\simeq 1$, i.e. $ \frac{\vert \nabla K_{i}\vert}{\lambda_{i}}\geq \frac{\kappa_{a}}{2\lambda_{i}^{2}}\geq \frac{\kappa_{a}}{4\underline \lambda^{2}}$. Hence we deduce, that in any case  
\begin{equation*}
\begin{split}
I_{\lambda}+I_{a} 
\lesssim &
-
(1-\eta_{\alpha})(1-\eta_{v})
(\frac{1}{\underline \lambda^{2}}
+
O(\frac{1}{\kappa_{\lambda}\underline{\lambda}^{2}}
+
\sum_{r\neq s} \frac{\vert \nabla K_{r}\vert^{2}}{\lambda_{r}^{2}}+\frac{1}{\lambda_{r}^{4}}
+
\varepsilon_{r,s}^{\frac{n+2}{n}}
))
\end{split}
\end{equation*}
and \eqref{I_lambda_I_a_first_estimate_final} follows again, provided
$\kappa_{\lambda}\gg 1 $ is sufficiently large. 
We finally consider the remaining case 
\begin{equation*}
\forall \; 1\leq j \leq q \;:\;
\frac{\vert \nabla K_{j}\vert}{\lambda_{j}}
<
\frac{\kappa_{a}}{2\underline \lambda^{2}}
\;\text{ and }\;
\sum_{r\neq s}  \varepsilon_{r,s}
\geq 
\frac{1}{\kappa_{\lambda}\max_{r}m_{\lambda_{r}}\underline \lambda^{2}}.
\end{equation*}
Then recalling again the definition of $\eta_{\lambda_{\cdot}}^{\geq}$ we find from \eqref{dt_J_u_general_version} 
\begin{equation*}
\begin{split}
I_{\lambda}+I_{a}
\leq &
c_{3}
 (1-\eta_{\alpha})(1-\eta_{v}) 
(
 \sum_{j\neq i}\eta_{\lambda_{j}}^{\geq} m_{\lambda_{j}} 
\frac{\alpha_{i}\alpha_{j}}{(\alpha_{K}^{\frac{2n}{n-2}})^{\frac{n-2}{n}}}\lambda_{j}\partial_{\lambda_{j}}\varepsilon_{i,j} 
+
O
(
(1+\frac{\max_{r}m_{\lambda_{r}}}{\sqrt{\kappa_{\lambda}}})
\sum_{r\neq s}\varepsilon_{r,s}
)
)
\end{split}
\end{equation*} 
up to some 
\begin{equation*}
(1-\eta_{\alpha})(1-\eta_{v})
O(
\sum_{r\neq s} \frac{\vert \nabla K_{r}\vert^{2}}{\lambda_{r}^{2}}+\frac{1}{\lambda_{r}^{4}}
+
\varepsilon_{r,s}^{\frac{n+2}{n}}
).
\end{equation*}
Let us decompose for  some $\Lambda\gg 1 $
and with a slight abuse of notation
\begin{equation*}
q=q_{1}+q_{2}, \text{ where } 
q_{1}=\{1\leq i \leq q \mid \lambda_{i}>\Lambda \underline \lambda \}
\text{ and } q_{2}=q\setminus q_{1}.
\end{equation*}
In particular for $i,j\in q_{2}$ we have  
\begin{equation*}  
\Lambda^{-1}\leq \frac{\lambda_{i}}{\lambda_{j}}\leq \Lambda
\; \text{ and } \; 
\vert \nabla K_{i}\vert \leq \frac{\kappa_{a}\Lambda}{2\underline{\lambda}}
\xrightarrow{\varepsilon \to 0} 0.
\end{equation*}
Hence, if $a_{i},a_{j}$ for $i,j\in q_{2}$ were close to the same critical point, we find, cf. \ref{eq:eij},  the  contradiction
\begin{equation*}
0\xleftarrow{\,\varepsilon\to 0\,}
\varepsilon_{i,j}
\simeq
(
\frac{1}
{\Lambda+\Lambda^{2}\underline \lambda^{2}(\frac{\kappa_{a}\Lambda}{2\underline{\lambda}})^{2}}
)^{\frac{n-2}{2}}
\simeq 
\frac{1}{(\kappa_{a}\Lambda^{2})^{n-2}}
\hspace{4pt}\not \hspace{-4pt}\longrightarrow 0\; \text{ as }\; \varepsilon \longrightarrow 0.
\end{equation*}
Hence for $i,j \in q_{2}$ we may assume, that  $a_{i},a_{j}$ are close to different critical points of $K$, whence 
\begin{equation*}
\varepsilon_{i,j}\simeq (\frac{1}{\lambda_{i}\lambda_{j}})^{\frac{n-2}{2}}
\leq 
\frac{1}{\underline{\lambda}^{n-2}}.
\end{equation*} 
Moreover for $j\in q_{1}$ 
\begin{equation*}
\{\eta_{\lambda_{j}}^{\geq}<1\}
=
\{\sum_{r \neq s }\varepsilon_{r,s}<\frac{\kappa_{\lambda}}{\lambda_{j}^{2}}\}
\subseteq
\{\sum_{r \neq s }\varepsilon_{r,s}<\frac{\kappa_{\lambda}}{\Lambda^{2} \underline \lambda^{2}}\} .
\end{equation*}
Consequently 
\begin{equation*}
\begin{split}
I_{\lambda} + I_{a}
\leq &
c_{3}
 (1-\eta_{\alpha})(1-\eta_{v}) \sum_{q_{1}\ni j\neq i}m_{\lambda_{j}} 
\frac{\alpha_{i}\alpha_{j}}{(\alpha_{K}^{\frac{2n}{n-2}})^{\frac{n-2}{n}}}\lambda_{j}\partial_{\lambda_{j}}\varepsilon_{i,j}
 \\
& +
c_{3}
 (1-\eta_{\alpha})(1-\eta_{v}) \sum_{q_{2}\ni j\neq i \in q_{1}}\eta_{\lambda_{j}}^{\geq} m_{\lambda_{j}} 
\frac{\alpha_{i}\alpha_{j}}{(\alpha_{K}^{\frac{2n}{n-2}})^{\frac{n-2}{n}}}\lambda_{j}\partial_{\lambda_{j}}\varepsilon_{i,j}
\\
&+
(1-\eta_{\alpha})(1-\eta_{v}) 
O
(
\frac{\kappa_{\lambda}\max_{r}m_{\lambda_{r}}}{\Lambda^{2} \underline \lambda^{2}}
+
\frac{\max_{r}m_{\lambda_{r}}}{\underline \lambda^{n-2}} 
+
(1+\frac{\max_{r}m_{\lambda_{r}}}{\sqrt{\kappa_{\lambda}}})\sum_{r\neq s}\varepsilon_{r,s}
)
\end{split}
\end{equation*}
 up to some 
\begin{equation*}
(1-\eta_{\alpha})(1-\eta_{v})
O
(
\sum_{r\neq s} \frac{\vert \nabla K_{r}\vert^{2}}{\lambda_{r}^{2}}+\frac{1}{\lambda_{r}^{4}}
+
\varepsilon_{r,s}^{\frac{n+2}{n}}
)
\end{equation*}
and rearranging this we obtain up to the same error
\begin{equation*}
\begin{split}
I_{\lambda}+I_{a}
\leq &
c_{3}
 (1-\eta_{\alpha})(1-\eta_{v}) \sum_{q_{1}\ni j\neq i \in q_{1}}m_{\lambda_{j}} 
\frac{\alpha_{i}\alpha_{j}}{(\alpha_{K}^{\frac{2n}{n-2}})^{\frac{n-2}{n}}}\lambda_{j}\partial_{\lambda_{j}}\varepsilon_{i,j} \\
& +
c_{3}
 (1-\eta_{\alpha})(1-\eta_{v})
 \sum_{q_{1}\ni j\neq i \in q_{2}}
 \frac{\alpha_{i}\alpha_{j}}{(\alpha_{K}^{\frac{2n}{n-2}})^{\frac{n-2}{n}}}
 (
 m_{\lambda_{j}} \lambda_{j}\partial_{\lambda_{j}}
 +
 \eta_{\lambda_{i}}^{\geq} m_{\lambda_{i}} 
\lambda_{i} \partial_{\lambda_{i}}
 )
\varepsilon_{i,j}
\\
&+
(1-\eta_{\alpha})(1-\eta_{v}) 
O
(
\frac{\kappa_{\lambda}\max_{r}m_{\lambda_{r}}}{\Lambda^{2} \underline \lambda^{2}}
+
\frac{\max_{r}m_{\lambda_{r}}}{\underline \lambda^{n-2}} 
+
(1+\frac{\max_{r}m_{\lambda_{r}}}{\sqrt{\kappa_{\lambda}}})\sum_{r\neq s}\varepsilon_{r,s}
) .
\end{split}
\end{equation*}
Recalling \eqref{lambda_derivative_epsilon_i_j}
and from  Definition \ref{def_flow} 
\begin{enumerate}[label=(\roman*)]
 \item \quad $m_{\lambda_{j}}\geq m_{\lambda_{i}}$ for $\lambda_{j}\geq \lambda_{i};$
\item \quad $m_{\lambda_{j}}\geq \kappa  m_{\lambda_{i}} $ for $\;\text{ for }\; \lambda_{j}\geq 2\lambda_{i}$,
\end{enumerate}
we find
\begin{equation*}
 \sum_{q_{1}\ni j\neq i \in q_{1}}m_{\lambda_{j}} 
\frac{\alpha_{i}\alpha_{j}}{(\alpha_{K}^{\frac{2n}{n-2}})^{\frac{n-2}{n}}}\lambda_{j}\partial_{\lambda_{j}}\varepsilon_{i,j}
\lesssim
-
\sum_{q_{1}\ni j\neq i \in q_{1}}m_{\lambda_{i}} 
\varepsilon_{i,j}
\end{equation*}
and  using $\lambda_{j}\geq \lambda_{i}$ for $j\in q_{1}$ and $i\in q_{2}$
\begin{equation*}
\sum_{q_{1}\ni j\neq i \in q_{2}}
 \frac{\alpha_{i}\alpha_{j}}{(\alpha_{K}^{\frac{2n}{n-2}})^{\frac{n-2}{n}}}
 (
 m_{\lambda_{j}} \lambda_{j}\partial_{\lambda_{j}}
 +
 \eta_{\lambda_{i}}^{\geq} m_{\lambda_{i}} 
\lambda_{i} \partial_{\lambda_{i}}
 )
\varepsilon_{i,j}
\lesssim 
\sum_{q_{1}\ni j\neq i \in q_{2}}
m_{\lambda_{i}}
\varepsilon_{i,j}.
\end{equation*}
Therefore and recalling $\varepsilon_{i,j}\lesssim \frac{1}{\underline{\lambda}^{n-2}}$ for $i,j\in q_{2}$ 
 \begin{equation*}
\begin{split}
I_{\lambda}+I_{a}
\lesssim & 
 -(1-\eta_{\alpha})(1-\eta_{v}) 
(\sum_{i\neq j}  m_{\lambda_{i}}\varepsilon_{i,j}
+
O
(
\frac{\kappa_{\lambda}\max_{r}m_{\lambda_{r}}}{\Lambda^{2} \underline \lambda^{2}}
+
\frac{\max_{r}m_{\lambda_{r}}}{\underline \lambda^{n-2}} 
+
(1+\frac{\max_{r}m_{\lambda_{r}}}{\sqrt{\kappa_{\lambda}}})\sum_{r\neq s}\varepsilon_{r,s}
)
)
\end{split}
\end{equation*}
up to some 
\begin{equation*}
(1-\eta_{\alpha})(1-\eta_{v})
O
(
\sum_{r\neq s} \frac{\vert \nabla K_{r}\vert^{2}}{\lambda_{r}^{2}}+\frac{1}{\lambda_{r}^{4}}
+
\varepsilon_{r,s}^{\frac{n+2}{n}}
).
\end{equation*}
Consequently  \eqref{I_lambda_I_a_first_estimate_final} follows again and thus in any case 
from $\kappa^{q^{2}}\geq \max_{r}m_{\lambda_{r}}$ and 
upon choosing 
$$\sqrt{\Lambda} \gg \sqrt{\kappa_{\lambda}}\gg \kappa^{q^{2}}.$$ 
We therefore conclude combining \eqref{I_alpha_+_I_v_final_estimate} 
and \eqref{I_lambda_I_a_first_estimate_final}, that on $V(q,\varepsilon)$ for $\varepsilon>0$ sufficiently small 
\begin{equation*}
\begin{split}
I_{\alpha}+I_{\lambda}+I_{a}+I_{v}
\lesssim &
-
\Vert v \Vert^{2}
-
\eta_{\alpha} (1-\eta_{v})\sum_{j}\vert 1
 -
 \frac{\alpha^{2}}{\alpha_{K}^{\frac{2n}{n-2}}}K_{i}\alpha_{i}^{\frac{4}{n-2}}
 \vert^{2}   \\
& -
(1-\eta_{\alpha})(1-\eta_{v})
(
\sum_{j}\frac{\vert \nabla K_{j}\vert}{K_{j}\lambda_{j}}
+
\frac{1}{\lambda_{j}^{2}}
+
\sum_{j\neq i}\varepsilon_{i,j}
+
O(
\sum_{r\neq s} \frac{\vert \nabla K_{r}\vert^{2}}{\lambda_{r}^{2}}+\frac{1}{\lambda_{r}^{4}}
+
\varepsilon_{r,s}^{\frac{n+2}{n}}
)),
\end{split}
\end{equation*} 
which recalling the definitions of $\eta_{v},\eta_{\alpha}$, cf. Definition \ref{def_flow}  simplifies to 
\begin{equation}\label{I_alpha_+_I_v_+I_lambda+I_afinal_estimate}
\begin{split}
I_{\alpha} & +I_{\lambda}+I_{a}+I_{v} \\
\lesssim  &
-
\Vert v \Vert^{2}
-
\eta_{\alpha} (1-\eta_{v})\sum_{j}\vert 1
 -
 \frac{\alpha^{2}}{\alpha_{K}^{\frac{2n}{n-2}}}K_{i}\alpha_{i}^{\frac{4}{n-2}}
 \vert^{2}   
 -
(1-\eta_{\alpha})(1-\eta_{v})
(
\sum_{j}\frac{\vert \nabla K_{j}\vert}{K_{j}\lambda_{j}}
+
\frac{1}{\lambda_{j}^{2}}
+
\sum_{j\neq i}\varepsilon_{i,j}
) \\
\lesssim & 
-
\Vert v \Vert^{2}
-
\sum_{j}
(
\vert 1
 -
 \frac{\alpha^{2}}{\alpha_{K}^{\frac{2n}{n-2}}}K_{i}\alpha_{i}^{\frac{4}{n-2}}
 \vert^{2}    
 +
\frac{\vert \nabla K_{j}\vert^{2}}{K_{j}\lambda_{j}^{2}}
+
\frac{1}{\lambda_{j}^{4}}
+
\sum_{j\neq i}\varepsilon_{i,j}^{\frac{n+2}{n}}
) 
.
\end{split} 
\end{equation} 
As for the gluing  with the gradient flow on some  $V(q,2\varepsilon)\setminus V(q,\varepsilon)$ 
via $\eta$,  i.e. 
\begin{equation*}
\partial_{t}u=\eta A_{q}(\alpha,a,\lambda,v)-(1-\eta)\nabla J(u),
\end{equation*}
we remark, that with a suitably small constant $ c_{q}>0 $ from  \eqref{I_alpha_+_I_v_+I_lambda+I_afinal_estimate} we now have
\begin{equation*}
\eta \partial J(u)A_{q}(\alpha,a,\lambda,v)
\leq
- c_{q} \eta 
(
\sum_{r\neq s}
\frac{\vert \nabla K_{r}\vert }{\lambda_{r}^{2}}
+
\frac{1}{\lambda_{r}^{4}}
+
\vert 1
-
\frac{\alpha^{2}}{\alpha_{K}^{\frac{2n}{n-2}}}
K_{r}\alpha_{r}^{\frac{4}{n-2}} \vert^{2}
+
\varepsilon_{r,s}^{\frac{n+2}{n}}
+
\Vert v \Vert^{2}
),
\end{equation*}
on $ V(q,2\varepsilon) $, but also 
\begin{equation*}
(1-\eta)\vert \nabla J(u)\vert^{2} 
\gtrsim
c_{q}(1-\eta)
(
\sum_{r\neq s}
\frac{\vert \nabla K_{r}\vert^{2} }{\lambda_{r}^{2}}
+
\frac{1}{\lambda_{r}^{4}}
+
\vert 1
-
\frac{\alpha^{2}}{\alpha_{K}^{\frac{2n}{n-2}}}
K_{r}\alpha_{r}^{\frac{4}{n-2}} \vert^{2}
+
\varepsilon_{r,s}^{\frac{n+2}{n}}
+
\Vert v \Vert^{2})
\;\text{ on }\; 
V(q,2\varepsilon_{q})
\end{equation*} 
due to  Proposition \ref{prop_gradient_bounds} and Lemma \ref{lem_v_part_gradient}.
Thence the proposition follows.  
 \end{proof} 
 
Let us show, that a flow line, which at least up to a sequence in time concentrates for some $q$ eventually for every $\varepsilon>0$ in  
$V(q,\varepsilon)$, then the whole flow line will eventually stay for every $\varepsilon>0$ in  
$V(q,\varepsilon)$. In particular such a flow line will be eventually governed by the prescribed movements 
$(\alpha),(a),(\lambda)$ and $(v)$ as in Definition \ref{def_flow}, i.e. the patching with the gradient flow will be irrelevant.  
\begin{lemma}\label{lem_capturing_the_flow} 
If for a flow line $u=\Phi(\cdot,u_{0})$ 
\begin{equation*}
u_{t_{k}}=\Phi(t_{k},u_{0})\in V(q,\varepsilon_{k}) \text{ with }\varepsilon_{k}\longrightarrow 0,
\end{equation*}
then for every $\varepsilon>0$ there exists $T_{\varepsilon}>0$ such, that
\begin{equation*}
u=\Phi(t,u_{0}) \in V(q,\varepsilon)\text{ for all } t>T_{\varepsilon}.
\end{equation*}
\end{lemma}
\begin{proof}
If the statement was false, there would exist
$ 
0<t_{1}<t_{1}^{\prime}<t_{2}<t_{2}^{\prime}<\ldots<\infty
$ 
such, that 
\begin{equation*}
u\in V(q,\varepsilon)\setminus V(q,\frac{\varepsilon}{2})
\; \text{ during } \; 
[t_{i},t_{i}^{\prime}]
,\;
u_{t_{i}}\in \partial V(q,\varepsilon)
\; \text{ and } \; 
u_{t_{i}^{\prime}}\in \partial V(q,\frac{\varepsilon}{2})
\end{equation*}
for some arbitrarily small $\varepsilon>0$. 
Thus the flow has during $[t_{i},t_{i}^{\prime}]$ to travel a distance 
\begin{equation*}
c_{0,\varepsilon}=d(\partial V(q,\varepsilon),\partial V(q,\frac{\varepsilon}{2}))>0
\end{equation*}
with bounded speed
$
\vert \partial_{t}u\vert \leq c_{1,\varepsilon}
$
and energy decay
$ %
 \partial_{t}J(u)
\leq - c_{2,\varepsilon}
$ 
due to proposition \ref{prop_decreasing_the_energy}. 
So the time 
 for this travelling is lower bounded, i.e. $\vert t_{i}^{\prime}-t_{i}\vert \geq \frac{c_{0,\varepsilon}}{c_{1,\varepsilon}}$,
 and thus we consume at least a quantity of energy
\begin{equation*}
J(u_{t_{i}})-J(u_{t_{i}^{\prime }})
=
-\int^{t_{i}^{\prime}}_{t_{i}}\partial_{t}J(u)
\geq
 \frac{c_{0,\varepsilon}c_{2,\varepsilon}}{c_{1,\varepsilon}}.
\end{equation*}
Clearly this leads to a contradiction, as  the lower bounded energy is never increased. 
\end{proof}

\section{Non compact flow lines}
 Since every flow line can be considered as a Palais-Smale sequence, when restricted to a sequence in time, every non compact zero weak limit flow line has by Proposition  \ref{blow_up_analysis} to enter every $V(q,\varepsilon)$ and by lemma \ref{lem_capturing_the_flow} to remain therein eventually. 
Let us study such a flow line 
$u=\alpha^{i}\varphi_{a_{i},\lambda_{i}}$,
which then satisfies 
\begin{enumerate}
 \item[($\alpha$)] \quad 
 $
 \frac{\dot\alpha_{j}}{\alpha_{j}} 
 =
 -
 \eta_{\alpha} (1-\eta_{v})
 \Vert \varphi_{j}\Vert^{-2} 
 (
 1
 -
 \frac{\alpha^{2}}{\alpha_{K}^{\frac{2n}{n-2}}}K_{j}\alpha_{j}^{\frac{4}{n-2}}
 )
 +
 b_{\alpha}
 ;
 $
  \item[($a$)] \quad
 $\lambda_{j}\dot a_{j}
 =
(1-\eta_{\alpha})(1-\eta_{v})\eta_{a_{j}}\frac{\nabla K_{j}}{\vert \nabla K_{j}\vert}
;$
 \item[($\lambda$)] \quad 
 $
 \frac{\dot \lambda_{j}}{\lambda_{j}}
 =
 -
 (1-\eta_{\alpha})(1-\eta_{v})
 (
\eta_{\lambda_{j}}^{\leq} \Pi_{i}(1-\eta_{a_{i}})^{m_{\lambda_{j},\lambda_{i}}}\frac{\Delta K_{j}}{\vert \Delta K_{j}\vert}
 +
 \eta_{\lambda_{j}}^{\geq}m_{\lambda_{j}}
 );
 $
 \item[($v$)] \quad 
 $
 \partial_{t}v=b_{v}-C_{v}v
 $
\end{enumerate}
for all times to come.  First note, that from $ (v)$ we have 
$\partial_{t}\Vert v \Vert^{2}\lesssim -\Vert v \Vert^{2}$, while from $(a)$ and $(\lambda)$
\begin{equation*}
\partial_{t} \vert \frac{1}{\lambda_{i}^{2}}\vert^{2}=O(\frac{1}{\lambda_{i}^{4}})
,\; 
\partial_{t}\vert \frac{\nabla K_{i}}{\lambda_{i}}\vert^{2}
=
O(
\frac{\vert \nabla K_{i}\vert^{2}}{\lambda_{i}^{2}}
+
\frac{1}{\lambda_{i}^{4}}
)
\; \text{ and } \; 
\partial_{t} \sum_{i\neq j}\varepsilon_{i,j}^{\frac{n+2}{n}} 
=
O(\sum_{i\neq j}\frac{\vert \nabla K_{i}\vert^{2}}{\lambda_{i}^{2}}+\frac{1}{\lambda_{i}^{4}}
+
\varepsilon_{i,j}^{\frac{n+2}{n}}) 
\end{equation*}
and finally from $(\alpha), (a)$ and \eqref{b_estimates}, cf.  also \eqref{alpha_and_alpha_K},
\begin{equation*}
\begin{split}
\partial_{t}\sum_{i}\vert 1-\frac{\alpha^{2}}{\alpha_{K}^{\frac{2n}{n-2}}}K_{i}\alpha_{i}^{\frac{4}{n-2}}\vert^{2}
= &
O(\sum_{i}\vert 1-\frac{\alpha^{2}}{\alpha_{K}^{\frac{2n}{n-2}}}K_{i}\alpha_{i}^{\frac{4}{n-2}}\vert)\Vert v \Vert
\\
& +
O(\sum_{i\neq j}\frac{\vert \nabla K_{i}\vert^{2}}{\lambda_{i}^{2}}+\frac{1}{\lambda_{i}^{4}}
+
\vert 1-\frac{\alpha^{2}}{\alpha_{K}^{\frac{2n}{n-2}}}K_{i}\alpha_{i}^{\frac{4}{n-2}}\vert^{2}
+
\varepsilon_{i,j}^{\frac{n+2}{n}}). 
\end{split}
\end{equation*}
Recalling the definition of $\eta_{v}$, cf. Definition \ref{def_flow}, we then have for $\kappa_{v}$ sufficiently large 
\begin{equation*}
\partial_{t}\Vert v \Vert^{2}
\ll
\partial_{t}
\sum_{r\neq s}
\frac{\vert \nabla K_{r}\vert^{2}}{\lambda_{r}^{2}}
+
\frac{1}{\lambda_{r}^{4}}
+
\vert 1-\frac{\alpha^{2}}{\alpha_{K}^{\frac{2n}{n-2}}K_{r}\alpha_{r}^{\frac{4}{n-2}}} \vert^{2}
+
\varepsilon_{r,s}^{\frac{n+2}{n}} 
\; \text{ on } \; 
\{\eta_{v}\geq 0\}
\end{equation*}
and may thus assume, that 
 eventually, hence from now on and for all times to come
\begin{equation}\label{v_control}  
\eta_{v}\equiv 0.
\end{equation}
We turn to describing the movement in $a_{i}$ and $\lambda_{i}$. 
Clearly we may assume
 $$
\forall \;i : \lambda_{i}\xrightarrow{t\to \infty}  \infty
 $$
 due to Lemma \ref{lem_capturing_the_flow}
 and so at least for a time sequence 
 $t_{k} \longrightarrow \infty$
\begin{equation*}
\partial_{t}\overline \lambda\lfloor_{t=t_{k}} >0,
\; \text{ where } \; 
\overline \lambda=\max_{i}\lambda_{i}
.
\end{equation*}
Hence  $(\lambda)$ and $\forall\; i \;:\;m_{\overline{\lambda},\lambda_{i}}=1$, cf. Definition \ref{def_flow},  show, 
that necessarily 
\begin{equation*}
\forall\;i : \eta_{a_{i}}<1
\; \text{ at }\; 
t=t_{k}
\; \text{ and hence }\;
\forall\; i\;:\; \frac{\vert \nabla K_{i}\vert}{\lambda_{i}}\leq \frac{\kappa_{a}}{\lambda_{i}^{2}}.
\end{equation*}
Since we assume $K$ to be Morse, we conclude, that at least for a sequence in  time  
\begin{equation*}
\forall \; i\;:\;a_{i}\longrightarrow x_{i}\in \{\vert \nabla K \vert =0\}.
\end{equation*}
On the other hand due to $(a)$ all $a_{i}$ move exclusively along the gradient of $K$, whence necessarily
\begin{equation*}
a_{i}\xrightarrow{t\to \infty} x_{i}\in \{\vert \nabla K \vert =0\}
\end{equation*}
and $a_{i}$ has to move along the stable manifold of $x_{i}$ with respect to the positive gradient flow for $K$, hence
\begin{equation*}
\forall \;i\;:\;
a_{i}\in W_{s}(x_{i})=W_{s}^{\nabla K}(x_{i}).
\end{equation*}
But  the only possibility for $\lambda_{i}$ to increase is $\Delta K_{i}<0$, cf. $(\lambda)$, whence necessarily  as a first consequence 
\begin{equation*}
\forall \; i \;:\; 
W_{s}(x_{i}) \ni a_{i}\xrightarrow{t\to \infty} x_{i}\in \{\vert \nabla K \vert =0\}\cap \{\Delta K<0 \}.
\end{equation*}
In particular we may assume $\Delta K_{i}<0$ from now on. Secondly, since for $\lambda_{j}\geq \lambda_{k}$ 
\begin{equation*}
\eta^{\leq }_{\lambda_{j}}\leq \eta^{\leq}_{\lambda_{k}}
,\; 
\eta^{\geq }_{\lambda_{j}}\geq \eta^{\geq}_{\lambda_{k}}
\; \text{ and }\; \forall i\;:\; 
m_{\lambda_{j},\lambda_{i}}\geq m_{\lambda_{k},\lambda_{i}}
\; \text{ and  }\; 
m_{\lambda_{j}}\geq m_{\lambda_{k}}
,
\end{equation*}
cf. Definition \ref{def_flow}, we derive from $(\lambda)$ using $\frac{\Delta K_{i}}{\vert \Delta K_{i}\vert}=-1$ and $\eta_{v}= 0$
\begin{equation*}
\begin{split}
\frac{\dot \lambda_{j}}{\lambda_{j}}
= &
 (1-\eta_{\alpha})
 (
\eta_{\lambda_{j}}^{\leq} \Pi_{i}(1-\eta_{a_{i}})^{m_{\lambda_{j},\lambda_{i}}}
 -
 \eta_{\lambda_{j}}^{\geq}m_{\lambda_{j}}
 ) 
 \leq  
  (1-\eta_{\alpha})
 (
\eta_{\lambda_{k}}^{\leq} \Pi_{i}(1-\eta_{a_{i}})^{m_{\lambda_{k},\lambda_{i}}}
 -
 \eta_{\lambda_{k}}^{\geq}m_{\lambda_{k}}
 )
 =
 \frac{\dot \lambda_{k}}{\lambda_{k}},
\end{split}
\end{equation*}
whence $\partial_{t}\frac{\lambda_{j}}{\lambda_{k}}\leq 0$
for $\lambda_{j}\geq \lambda_{k}$ and we may  therefore assume from now on 
\begin{equation}\label{lambda_bounds}
1\leq \frac{\overline \lambda}{\underline \lambda}\leq C=C(u_{0})
\; \text{ for }\; \overline \lambda =\max_{i}\lambda_{i}
\;\text{ and }\; 
\underline{\lambda}=\min_{i}\lambda_{i}.
\end{equation}
Thirdly, as we had said,  $\overline{\lambda}$ has to grow at
times $t=t_{k}$, and then we have
\begin{equation}\label{nabla_K_i_smallness_sequential} 
\forall \;i\;:\; \eta_{a_{i}}<1
\;
\text{ and}
\; 
\frac{\vert \nabla K_{i}\vert}{\lambda_{i}}
\leq \frac{\kappa_{a}}{\lambda_{i}^{2}},
 \end{equation}
 whereas generally  there holds 
due to $(a)$ and $(\lambda)$
\begin{equation*}
\begin{split}
\frac{\partial_{t}}{2}(\lambda_{i}^{2}\vert \nabla K_{i}\vert^{2})
= &
\frac{\dot \lambda_{i}}{\lambda_{i}}\lambda_{i}^{2}\vert \nabla K_{i}\vert^{2}
+
\lambda_{i}\vert \nabla K_{i}\vert
\langle {2}K_{i},\frac{\nabla K_{i}}{\vert \nabla K_{i}\vert},\lambda_{i}\dot a_{i}\rangle  
\\
\leq & 
(1-\eta_{\alpha})
[
(1-\eta_{a_{i}})\lambda_{i}^{2}\vert \nabla K_{i}\vert^{2}
+
\eta_{a_{i}}
\lambda_{i}\vert \nabla K_{i}\vert
\langle {2}K_{i},\frac{\nabla K_{i}}{\vert \nabla K_{i}\vert},\frac{\nabla K_{i}}{\vert \nabla K_{i}\vert}\rangle  
],
\end{split}
\end{equation*}
where we used $m_{\lambda_{i},\lambda_{i}}=1$.
Since $K$ is Morse and $a_{i}$ is close to $x_{i}$ and moves along $W_{s}(x_{i})$, we have 
\begin{equation}\label{a_along_stable} 
\langle {2}K_{i},\frac{\nabla K_{i}}{\vert \nabla K_{i}\vert},\frac{\nabla K_{i}}{\vert \nabla K_{i}\vert}\rangle  
<
-\delta\; \text{ for some }\; \delta>0
\end{equation}
and consequently 
\begin{equation}\label{staying_in_A_a}
\begin{split}
\frac{\partial_{t}}{2}(\lambda_{i}^{2}\vert \nabla K_{i}\vert^{2})
\leq & 
(1-\eta_{\alpha})
[
(1-\eta_{a_{i}})\lambda_{i}^{2}\vert \nabla K_{i}\vert^{2}
-
\delta\,\eta_{a_{i}}
\lambda_{i}\vert \nabla K_{j}\vert
].
\end{split}
\end{equation}
In particular \eqref{nabla_K_i_smallness_sequential} and \eqref{staying_in_A_a} imply, that  
we may assume from now on and for all times to come
\begin{equation}\label{nabla_K_control}
\forall\;i\;:\; \eta_{a_{i}}<1-\epsilon 
\; \text{ and }\; 
\frac{d_{g_{0}}(a_{i},x_{i})}{\lambda_{i}}\simeq \frac{\vert \nabla K_{i}\vert}{\lambda_{i}}
<
\frac{\kappa_{a}}{\lambda_{i}^{2}}
\end{equation}
for some fixed $\epsilon>0$. 
From \eqref{lambda_bounds} and \eqref{nabla_K_control} we then may exclude tower bubbling, i.e. 
\begin{equation*}
a_{i}\longrightarrow x_{i}=x_{j}\longleftarrow a_{j}.
\end{equation*}
Indeed in the latter case 
\begin{equation*}
d^{2}_{g_{0}}(a_{i},a_{j})\leq (d_{g_{0}}(a_{i},x_{i})+d_{g_{0}}(a_{j},x_{j})^{2}
\lesssim 
(\lambda_{i}^{-1}+\lambda_{j}^{-1})^{2}
\lesssim 
\lambda_{i}^{-2}+\lambda_{j}^{-2} 
\end{equation*} 
and  $\lambda_{i}\simeq \lambda_{j}$, 
whence 
\begin{equation*}
0\xleftarrow{\,t\to \infty\,} \varepsilon_{i,j}
\simeq 
(\frac{\lambda_{i}}{\lambda_{j}}+\frac{\lambda_{j}}{\lambda_{i}}+\lambda_{i}\lambda_{j}d^{2}_{g_{0}}(a_{i},a_{j}))^{\frac{2-n}{2}}
\simeq 
(\frac{\lambda_{i}}{\lambda_{j}}+\frac{\lambda_{j}}{\lambda_{i}})^{\frac{2-n}{2}}
\hspace{4pt}\not \hspace{-4pt}
\longrightarrow 0 
\; \text{ as} \; 
\varepsilon \longrightarrow 0
,
\end{equation*}
a contradiction.  Hence we may assume $x_{i}\neq x_{j}$, thus 
$d(a_{i},a_{j}) \hspace{4pt}\not \hspace{-4pt}\longrightarrow 0 $ and therefore
\begin{equation}\label{interaction_control}
\forall \; i \;:\; \sum_{r\neq s}
\varepsilon_{r,s}\simeq\sum_{r\neq s} (\lambda_{r}\lambda_{s})^{\frac{2-n}{2}}
\simeq \frac{1}{\lambda_{i}^{n-2}}
\; \text{ and }\; 
\sum_{r\neq s}\varepsilon_{r,s}^{\frac{n+2}{2n}}
=
o(\frac{1}{\lambda_{i}^{2}})
\end{equation}
for all times to come due to $\lambda_{r}\simeq \lambda_{s}\simeq \lambda_{i}$, cf. \eqref{lambda_bounds}. 
In particular $(a)$ and $(\lambda)$ therefore simplifies to 
\begin{enumerate} 
\item[$(a)$]\quad
$
\lambda_{j}\dot a_{j}=(1-\eta_{\alpha})\eta_{a_{j}}\frac{\nabla K_{j}}{\vert \nabla K_{j}\vert};
$
\item[$(\lambda)$]\quad
$
 \frac{\dot \lambda_{j}}{\lambda_{j}}
 =
 (1-\eta_{\alpha})
\Pi_{i}(1-\eta_{a_{i}})^{m_{\lambda_{j},\lambda_{i}}}
\gtrsim 
\epsilon^{q} (1-\eta_{\alpha}) 
$
\end{enumerate} 
due to \eqref{nabla_K_control},  \eqref{interaction_control}  and, cf. Definition \ref{def_flow}, 
$m_{\lambda_{j},\lambda_{i}}\in [0,1]$ and 
\begin{enumerate}[label=(\roman*)]
 \item \quad 
 $ \eta^{\leq}_{\lambda_{j}}=1  $
on 
$ \;\{\sum_{r\neq s}\varepsilon_{r,s}\leq \frac{\kappa_{\lambda}^{-1}}{\lambda_{j}^{2}}\} \;$;
 \item \quad 
$\eta^{\geq}_{\lambda_{j}}=0$ 
on 
$ \;\{\sum_{r\neq s}\varepsilon_{r,s}\leq \frac{\sqrt{\kappa_{\lambda}}}{\lambda_{j}^{2}}\} .$
\end{enumerate} 
We turn our attention to  the movement in  $\alpha_{i}$. Since we may assume by now $\eta_{v}=0$, hence
\begin{equation*}
\Vert v \Vert 
\lesssim
\sum_{r\neq s} \frac{\vert \nabla K_{r}\vert}{\lambda_{r}}+\frac{1}{\lambda_{r}^{2}}
+
\vert 1
-
\frac{\alpha^{2}}{\alpha_{K}^{\frac{2n}{n-2}}}
K_{r}\alpha_{r}^{\frac{4}{n-2}} \vert
+
\varepsilon_{r,s}^{\frac{n+2}{2n}}
\lesssim 
\sum_{r} \frac{1}{\lambda_{r}^{2}}
+
\vert
1
-
\frac{\alpha^{2}}{\alpha_{K}^{\frac{2n}{n-2}}}
K_{r}\alpha_{r}^{\frac{4}{n-2}} \vert,
\end{equation*}
cf. Definition \ref{def_flow}, \eqref{nabla_K_control} and \eqref{interaction_control}, and due to \eqref{b_estimates}, we have
\begin{enumerate}
 \item[$(\alpha)$] \quad
 $
 \frac{\dot \alpha_{j}}{\alpha_{j}}
 =
 -
\frac{ \eta_{\alpha} }{\bar{c_{0}}^{2}} 
 (
 1
 -
 \frac{\alpha^{2}}{\alpha_{K}^{\frac{2n}{n-2}}}K_{j}\alpha_{j}^{\frac{4}{n-2}}
 ) 
+
o(
\sum_{r}
\vert
1
-
\frac{\alpha^{2}}{\alpha_{K}^{\frac{2n}{n-2}}}
K_{r}\alpha_{r}^{\frac{4}{n-2}} \vert 
)
+
O( 
 \sum_{r} \frac{1}{\lambda_{r}^{2}}
 ),
 $
\end{enumerate}
where we made use of $\Vert \varphi_{j}\Vert=\bar c_{0}+O(\frac{1}{\lambda_{j}^{2}})$.
Note, that due to $(a)$ we have   
\begin{equation*}
\partial_{t}(\frac{\alpha^{2}}{\alpha_{K}^{\frac{2n}{n-2}}}K_{j}\alpha_{j}^{\frac{2n}{n-2}}-\alpha_{j}^{2})
=
\sum_{i}\alpha_{i}\partial_{\alpha_{i}}(\frac{\alpha^{2}}{\alpha_{K}^{\frac{2n}{n-2}}})
\frac{\dot \alpha_{i}}{\alpha_{i}}
K_{j}\alpha_{j}^{\frac{2n}{n-2}} 
+
(\frac{\alpha^{2}}{\alpha_{K}^{\frac{2n}{n-2}}}
K_{j}
\alpha_{j}\partial_{\alpha_{j}}\alpha_{j}^{\frac{2n}{n-2}}
-
\alpha_{j}\partial_{\alpha_{j}}\alpha_{j}^{2}
)
\frac{\dot \alpha_{j}}{\alpha_{j}}, 
\end{equation*}
up to some 
$O(\sum_{r}\frac{\vert \nabla K_{r}\vert}{\lambda_{r}})$.  
Recalling \eqref{alpha_and_alpha_K} we then find, 
that  up to some 
\begin{equation*}
o(
\sum_{r}
\vert
1
-
\frac{\alpha^{2}}{\alpha_{K}^{\frac{2n}{n-2}}}
K_{r}\alpha_{r}^{\frac{4}{n-2}} \vert 
)
+
O( 
 \sum_{r} 
 \frac{\vert \nabla K_{r}\vert}{\lambda_{r}}
 +
 \frac{1}{\lambda_{r}^{2}}
 )
\end{equation*}
there holds  
$
 \sum_{i}\alpha_{i}\partial_{\alpha_{i}}\alpha^{2}\frac{\dot \alpha_{i}}{\alpha_{i}}
 =
 0 
 $
and secondly
$$
  \sum_{i}\alpha_{i}\partial_{\alpha_{i}}\alpha_{K}^{\frac{2n}{n-2}}\frac{\dot \alpha_{i}}{\alpha_{i}}
  =
\frac{2n}{n-2}\frac{\eta_{\alpha}}{\bar{c_{0}}^{2}}  \sum_{i}K_{i}\alpha_{i}^{\frac{2n}{n-2}}
(\frac{\alpha^{2}}{\alpha_{K}^{\frac{2n}{n-2}}}K_{i}\alpha_{i}^{\frac{4}{n-2}}-1),
 $$ 
i.e. 
 \begin{equation*}
\begin{split}
\frac{\alpha^{2}}{\alpha_{K}^{\frac{2n}{n-2}}}
  \sum_{i}\alpha_{i}\partial_{\alpha_{i}}\alpha_{K}^{\frac{2n}{n-2}}\frac{\dot \alpha_{i}}{\alpha_{i}}
  = &
  \frac{2n}{n-2}\frac{\eta_{\alpha}}{\bar{c_{0}}^{2}}  \sum_{i}
  (\frac{\alpha^{2}}{\alpha_{K}^{\frac{2n}{n-2}}}K_{i}\alpha_{i}^{\frac{2n}{n-2}}-\alpha_{i}^{2}+\alpha_{i}^{2}) 
(\frac{\alpha^{2}}{\alpha_{K}^{\frac{2n}{n-2}}}K_{i}\alpha_{i}^{\frac{4}{n-2}}-1)
= 0.
\end{split}
 \end{equation*}
Hence we obtain up to the same error
\begin{equation*}
\begin{split}
\partial_{t}(\frac{\alpha^{2}}{\alpha_{K}^{\frac{2n}{n-2}}}K_{j}\alpha_{j}^{\frac{2n}{n-2}}-\alpha_{j}^{2})
= &
\frac{\eta_{\alpha}}{\bar{c_{0}}^{2}}
\left(
\frac{2n}{n-2}\frac{\alpha^{2}}{\alpha_{K}^{\frac{2n}{n-2}}}K_{j}\alpha_{j}^{\frac{2n}{n-2}} 
-
2\alpha_{j}^{2}
\right)
(\frac{\alpha^{2}}{\alpha_{K}^{\frac{2n}{n-2}}}K_{j}\alpha_{j}^{\frac{4}{n-2}}-1) \\
 = &
 \frac{4}{n-2}\frac{\eta_{\alpha}}{\bar{c_{0}}^{2}}
 \frac{\alpha^{2}}{\alpha_{K}^{\frac{2n}{n-2}}}K_{j}\alpha_{j}^{\frac{2n}{n-2}} 
 (\frac{\alpha^{2}}{\alpha_{K}^{\frac{2n}{n-2}}}K_{j}\alpha_{j}^{\frac{4}{n-2}}-1)
 \gtrsim 
 \eta_{\alpha}
  (\frac{\alpha^{2}}{\alpha_{K}^{\frac{2n}{n-2}}}K_{j}\alpha_{j}^{\frac{2n}{n-2}}-\alpha_{j}^{2}).
\end{split}
\end{equation*}
Consequently we have
\begin{equation}\label{alpha_growth}
\partial _{t}\sum_{i}\vert \frac{\alpha^{2}}{\alpha_{K}^{\frac{2n}{n-2}}}K_{j}\alpha_{j}^{\frac{2n}{n-2}}-\alpha_{j}^{2}\vert^{2} \geq 0
\; \text{ and  }\; 
\partial_{t}\frac{\vert \nabla K_{r}\vert^{2}}{\lambda_{r}^{2}},
\partial_{t}\frac{1}{\lambda_{r}^{4}}\leq   0
\end{equation}
as long as $u\in \{\eta_{\alpha}>\frac{1}{4}\}$ due to $(\lambda)$ 
and $(a)$, cf. \eqref{a_along_stable}. Thence necessarily 
\begin{equation}\label{alpha_control} 
\eta_{\alpha}<\frac{1}{2}
\end{equation}
for all times to come. Indeed we may assume 
\begin{enumerate}[label=(\roman*)]
 \item \quad 
 $\{\eta_{\alpha}>\frac{1}{4}\}
 =
 \{
\sum_{j}\vert 1-\frac{\alpha^{2}}{\alpha_{K}^{\frac{2n}{n-2}}}K_{j}\alpha_{j}^{\frac{4}{n-2}}\vert 
> \tilde \kappa_{\alpha,\frac{1}{4}}
(
\sum_{r\neq s}\frac{\vert \nabla K_{r}\vert}{\lambda_{r}}+\frac{1}{\lambda_{r}^{2}}
+
\varepsilon_{r,s}^{\frac{n+2}{2n}}
)
\}$
\item  \quad 
$
\{\eta_{\alpha}>\frac{1}{2}\}
 =
 \{
\sum_{j}\vert 1-\frac{\alpha^{2}}{\alpha_{K}^{\frac{2n}{n-2}}}K_{j}\alpha_{j}^{\frac{4}{n-2}}\vert 
> \tilde \kappa_{\alpha,\frac{1}{2}}
(
\sum_{r\neq s}\frac{\vert \nabla K_{r}\vert}{\lambda_{r}}+\frac{1}{\lambda_{r}^{2}}
+
\varepsilon_{r,s}^{\frac{n+2}{2n}}
)
\}
$
\end{enumerate}
with 
\begin{equation*}
\kappa_{\alpha\frac{1}{4}},\kappa_{\alpha,\frac{1}{2}}\in (\sqrt{\kappa_{\alpha}},\kappa_{\alpha})
\; \text{ and } \; 
\frac{\kappa_{\alpha,\frac{1}{2}}}{\kappa_{\alpha,\frac{1}{4}}}\longrightarrow \infty
\; \text{ as }\; 
\kappa_{\alpha}\longrightarrow \infty.
\end{equation*}  
We then find from \eqref{alpha_growth}, 
 since $\delta<\alpha_{i}<\delta^{-1}$ on $V(q,\varepsilon)\cap [\Vert \cdot \Vert=1]$,
that
\begin{equation*}
\begin{split}
\exists \; c_{\delta}>0 \;  
\forall\;  u_{0}\in \{\eta_{\alpha}>\frac{1}{2}\}\;:\; 
\sum_{j}\vert 1-\frac{\alpha^{2}}{\alpha_{K}^{\frac{2n}{n-2}}}K_{j}\alpha_{j}^{\frac{4}{n-2}}\vert 
\geq
c_{\delta} \,\kappa_{\alpha,\frac{1}{2}}
(\sum_{r}\frac{\vert \nabla K_{r}\vert}{\lambda_{r}}+\frac{1}{\lambda_{r}^{2}})   
\end{split}   
\end{equation*} 
as long as 
$u\in \{\eta_{\alpha}>\frac{1}{4}\}$, and thus by virtue of \eqref{interaction_control}
\begin{equation*}
\begin{split}
\sum_{j}\vert 1-\frac{\alpha^{2}}{\alpha_{K}^{\frac{2n}{n-2}}}K_{j}\alpha_{j}^{\frac{4}{n-2}}\vert 
\geq
\frac{c_{\delta} }{2}\kappa_{\alpha,\frac{1}{2}}
(\sum_{r\neq s}\frac{\vert \nabla K_{r}\vert}{\lambda_{r}}+\frac{1}{\lambda_{r}^{2}}
+
\varepsilon_{r,s}^{\frac{n+2}{2n}})  . 
\end{split}   
\end{equation*}
Consequently and, since we may  assume 
$$\frac{c_{\delta} }{2}\kappa_{\alpha,\frac{1}{2}} >
\kappa_{\alpha,\frac{1}{4}},$$
provided $\kappa_{\alpha}$ is sufficiently large, we will stay in $\{\eta_{\alpha}>\frac{1}{4}\}$
for all times, whence by virtue of \eqref{alpha_growth}   
\begin{equation*}
\sum_{i}\vert \frac{\alpha^{2}}{\alpha_{K}^{\frac{2n}{n-2}}}K_{j}\alpha_{j}^{\frac{4n}{n-2}}-1\vert
\simeq 
\sqrt{
\sum_{i}\vert \frac{\alpha^{2}}{\alpha_{K}^{\frac{2n}{n-2}}}K_{j}\alpha_{j}^{\frac{2n}{n-2}}-\alpha_{j}^{2}\vert^{2}
}\hspace{4pt} \not  \hspace{-4pt} \longrightarrow 0
\; \text{ as} \; t  \longrightarrow \infty, 
\end{equation*}
a contradiction. We thus conclude from \eqref{lambda_bounds}, $(\lambda)$, \eqref{nabla_K_control}, \eqref{alpha_control} and 
\eqref{v_control}, that eventually
\begin{equation*}
\forall\; i \; : \; 
\lambda_{i}\simeq \lambda_{j}
,\; 
e^{-Ct}
\leq 
\frac{1}{\lambda_{i}},d_{g_{0}}(a_{i},x_{i}),
\vert 1 - \frac{\alpha^{2}}{\alpha_{K}^{\frac{2n}{n-2}}}K_{i}\alpha_{i}^{\frac{2n}{n-2}}\vert,
\Vert v \Vert
\leq 
e^{-ct}
\end{equation*}
for some $0<c<C<\infty$.  We have thus derived the non trivial part of 

\begin{proposition} 
\label{prop_tower_bubble_exclusion}
A zero weak limit  flow line is non compact, if and only if eventually 
\begin{enumerate}[label=(\roman*)]
 \item \quad $\Vert v \Vert $ decays exponentially and
 \item \quad $\vert \alpha^{2}K_{i}\alpha_{i}^{\frac{4}{n-2}}-\alpha_{K}^{\frac{2n}{n-2}}\vert$ decays exponentially and
 \item \quad $\lambda_{i}$ increases exponentially and
 \item \quad $a_{i}\longrightarrow x_{i}\in \{\vert \nabla K \vert =0\}\cap\{ \nabla K<0\}$ along $W_{s}(x_{i})$ exponentially fast, where $x_{i}\neq x_{j}$
\end{enumerate}
for all $i\neq j=1,\ldots,q$. Conversely flow lines satisfying $(i)$-$(iv)$ exist. 
\end{proposition} 
 
\begin{proof}
By what we have seen above, every non compact zero weak limit flow line has to satisfy 
$(i)$-$(iv)$ above eventually and clearly every flow line satisfying $(i)$-$(iv)$ is non compact with zero weak limit. 
Hence we are left with showing their existence.  
Let us choose for simplicity as initial data 
\begin{equation*}
v=0,\; a_{i}=x_{i} \; \text{ and }\; 
\frac{\alpha^{2}}{\alpha_{K}^{\frac{2n}{n-2}}}K_{i}\alpha_{i}^{\frac{4}{n-2}}=1
\end{equation*}
for $x_{i}\in \{\vert \nabla K \vert =0\}\cap \{\Delta K<0\}$ and $x_{i}\neq x_{j}$ for $i\neq j$. 
Recalling $b_{v}=O(\Vert v\Vert)$, cf. Lemma \ref{lem_shadow_flow}, we have
\begin{equation*}
\partial \Vert v \Vert^{2}=-C_{v}\Vert v \Vert^{2}+O(\Vert v \Vert^{2}), \; C_{v} \gg 1
\end{equation*}
 cf. $(v)$ and hence $v=0$ is preserved. Secondly due to $(a)$ also $a_{i}=x_{i}$ is preserved. Thirdly 
 \begin{equation*}
\partial_{t}(\frac{\alpha^{2}}{\alpha_{K}^{\frac{2n}{n-2}}}K_{i}\alpha_{i}^{\frac{4}{n-2}})=0
 \end{equation*}
follows from $\frac{\dot \alpha_{j}}{\alpha_{j}}=b_{\alpha}$ and $\dot a_{i}=0$, cf. $(\alpha)$, hence also
$\frac{\alpha^{2}}{\alpha_{K}^{\frac{2n}{n-2}}}K_{i}\alpha_{i}^{\frac{4}{n-2}}=1$ is preserved. 
In particular 
\begin{equation*}
\eta_{v},\eta_{\alpha},\eta_{\alpha_{i}}=0
\end{equation*}
are preserved, as long as we do not leave some $V(q,\varepsilon)$, upon which 
$(\alpha),(a),(\lambda)$ and $(v)$ are valid. Thus
\begin{equation*}
 \frac{\dot \lambda_{j}}{\lambda_{j}}
 =
 (
\eta_{\lambda_{j}}^{\leq} \Pi_{i}
 -
 \eta_{\lambda_{j}}^{\geq}m_{\lambda_{j}}
 )
\end{equation*}
due to $(\lambda)$. Since $a_{i}=x_{i}$ and $x_{i}\neq x_{j}$ for $i\neq j$ by assumption, there holds for
$\underline \lambda=\min_{i}\lambda_{i}$
\begin{equation*}
\sum_{r\neq s}\varepsilon_{r,s}
\lesssim (\frac{1}{\lambda_{r}\lambda_{s}})^{\frac{n-2}{2}}\lesssim \frac{1}{\underline{\lambda}^{n-2}}
=
o(\frac{1}{\underline{\lambda}^{2}}), 
\end{equation*}
whence $\eta^{\leq}_{\underline{\lambda}}=1$ and $\eta^{\geq}_{\underline{\lambda}}=0$, 
so $\partial_{t}\underline \lambda =1$.  Hence the above $V(q,\varepsilon)$ will never be left and 
$\underline \lambda \nearrow \infty$.  
\end{proof}

 \begin{proof}[Proof of Theorem \ref{thm}] 
 Consider the flow introduced in Lemma \ref{easy_properties}, i.e. 
  \begin{equation*}
\Phi: \R_{\geq 0} \times X\longrightarrow X
\; \text{ with } \; 
X=\{0\leq u \in W^{1,2}(M) \mid u\neq 0,\; \Vert u \Vert=1\}, 
\end{equation*}
which decreases the energy $J$ according  to Proposition \ref{prop_decreasing_the_energy}.  
Then by energy reasoning every flow line
$u(t)=\Phi(t,\cdot)$ 
induces upon choice of a subsequence in time a Palais-Smale  sequence 
$u_{t_{k}}=\Phi(t_{k}, \cdot)$ 
. 
Indeed $J$ is by positivity of the Yamabe invariant strictly positive, while by virtue of Propositions
\ref{prop_gradient_bounds} and \ref{prop_decreasing_the_energy} every flow line consumes energy as long as 
$\vert \partial J\vert>0$. Then Proposition \ref{blow_up_analysis} and the comment following it show, that
upon a subsequence
$u_{t_{k}}$ is of zero weak limit, if and only if $u_{t_{k}}$ concentrates in the sense
\begin{equation*}
\forall\; \varepsilon>0\;\exists \;N\in \N \; \forall \; k\geq N\; : \;
u_{t_{k}}\in V(q,\varepsilon),
\end{equation*}
in which case $u$ is of zero weak limit itself, as Lemma \ref{lem_capturing_the_flow} shows. Hence  according to  Proposition \ref{prop_tower_bubble_exclusion} the full flow line concentrates 
simply  with limiting profile and energy
\begin{equation}\label{limiting_profile}
\sum_{i}\frac{c}{K_{i}^{\frac{n-2}{4}}}\delta_{x_{i}}
\; \text{ and }\; 
\lim_{t\to \infty}J(u(t))=J_{x_{1},\ldots,x_{q}}
=
c(\sum_{i}K_{i}^{\frac{2-n}{2}})^{\frac{2}{n}}
\end{equation}
respectively, where $ c>0 $ is a dimensional constant and  
$$x_{i},\ldots,x_{q}\in \{ \vert \nabla K\vert=0 \}\cap \{ \Delta K<0 \}$$
are distinct.
Hence zero weak limit sequences along a flow line are classified with respect to their end configuration, which corresponds one to one to subsets of $
\{\vert \nabla K \vert =0\}\cap \{\Delta K<0\}$ on the one hand and
to finite energy and zero weak limit subcritical blow-up solutions on the other, cf. \cite{MM2}.
And of course flow lines of the latter type do exist by Proposition \ref{prop_tower_bubble_exclusion}. 

\

So let us consider 
$$x_{i},\ldots,x_{q}\in \{\vert \nabla K \vert =0\} \cap \{\Delta K<0\}
\; \text{ with } \;  x_{i}\neq x_{j}$$ 
and denote correspondingly by
\begin{equation*}
 u_{\tau,x_{1},\ldots,x_{q}}\in \{\partial J_{\tau}=0\}
\end{equation*}
the unique,   zero weak limit subcritical blow-up solution from \cite{MM2} of  the same limiting 
profile and energy  
\begin{equation*}
\sum_{i}\frac{c}{K_{i}^{\frac{n-2}{4}}}\delta_{x_{i}}
\; \text{ and }\; 
\lim_{\tau \to 0}J(u_{\tau,x_{1,\ldots,x_{q}}})=J_{x_{1},\ldots,x_{q}}.
\end{equation*} 
Then by virtue of Proposition 3.1 in \cite{MM2} there exists 
$\varepsilon>0$ such, that for all
for $0 <\tau \ll \varepsilon$
\begin{equation*}
\{ u_{\tau,x_{1},\ldots,x_{q}}\}
=
\{\partial J_{\tau}=0\} 
\cap
V(q,\varepsilon)
\cap 
\{d(a_{i},x_{i})\ll 1\}, 
\end{equation*}
i.e. uniqueness as a solution on some $V(q,\varepsilon)
\cap 
\{d(a_{i},x_{i})\ll 1\}$. 
\begin{figure}[h!]
\begin{center}
\includegraphics[scale=0.8]{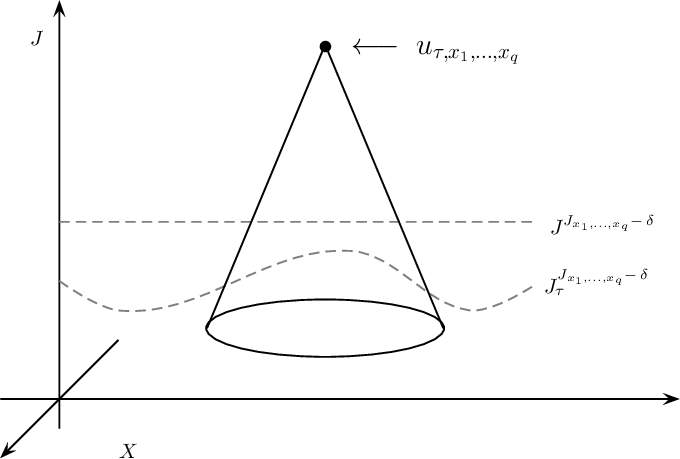}
 \caption{Attaching a cell suspended at a subcritical solution}
\label{Fig_suspension}
 \end{center}
\end{figure}
Hence 
\begin{equation*}
\{\partial J_{\tau}=0\} 
\cap
V(q,\varepsilon)
\cap 
\{d(a_{i},x_{i})\ll 1\}
\end{equation*}
for some $\varepsilon>0$ and any $0<\tau \ll \varepsilon $
contains exactly one element as a subcritical solution
with
\begin{equation*}
m_{x_{1},\ldots,x_{m}}=m(J_{\tau},u_{\tau,x_{1},\ldots,x_{q}})  
=
q-1+\sum^{q}_{i=1}(n-m(K,x_{i})) 
\end{equation*}
as Morse index. 
And we have a homotopy equivalence by attaching a cell, cf. Figure \ref{Fig_suspension},  
\begin{equation*}
J_{\tau}^{J_{x_{1},\ldots,x_{q}}- \,\delta} 
\cup 
(\;
V(q,\varepsilon) \cap \{\vert a_i -x_i \vert \leq \varepsilon\} 
\;)
\simeq 
J_{\tau}^{J_{x_{1},\ldots,x_{q}}- \,\delta} \; \sharp \;C,
\quad
\dim C=m_{x_{1},\ldots,x_{q}}
\end{equation*}
along the unstable manifold of and suspended at $u_{\tau,x_{1},\ldots,x_{q}}$ with energy
$$
J_{\tau}(u_{\tau,x_{1},\ldots,x_{q}})
=
J_{x_{1},\ldots,x_{q}}
+
o_{\frac{1}{\tau}}(1).
$$
Since  
$
J_{\tau}^{s}=\{J_{\tau}\leq s\} \subset \{J\leq s\}=J^{s}
$
due to H\"older's inequality, we then find 
\begin{equation*}
\begin{split}
H_{k} &
\left(  
J^{J_{x_{1},\ldots,x_{q}}- \,\delta}  
\cup  
\left( \;
V(q,\varepsilon)
\cap 
\{d(a_{i},x_{i})\ll 1\}
\; \right) 
,
J^{J_{x_{1},\ldots,x_{q}}- \,\delta}
\right)   
\\
= &
H_{k}
\left(
J^{J_{x_{1},\ldots,x_{q}}- \,\delta} 
\cup
(\;J_{\tau}^{J_{x_{1},\ldots,x_{q}}- \,\delta}  
\cup 
( \;
V(q,\varepsilon)
\cap 
\{d(a_{i},x_{i})\ll 1\}
\; )
 )
,
J^{J_{x_{1},\ldots,x_{q}}- \,\delta}
\right) \\
= &
H_{k}
\left(
J^{J_{x_{1},\ldots,x_{q}}- \,\delta} 
\cup
(
J_{\tau}^{J_{x_{1},\ldots,x_{q}}- \,\delta} \; \sharp \;C
)
,
J^{J_{x_{1},\ldots,x_{q}}- \,\delta}
\right) \\
= &
H_{k}
\left(
J^{J_{x_{1},\ldots,x_{q}}- \,\delta} 
\; \sharp \;C
,
J^{J_{x_{1},\ldots,x_{q}}- \,\delta}
\right)
=
\delta_{k,m_{x_1,\ldots,x_q}}
\end{split}
\end{equation*}
for the $k$-th relative singular homology 
$ H_{k}(A,B) $ with $ k=m_{x_{1},\ldots,x_{q}} $ 
of the pair 
\begin{equation*}
\begin{split}
J^{J_{x_{1},\ldots,x_{q}}- \,\delta}
=
B
\subset
A=
J^{J_{x_{1},\ldots,x_{q}}- \,\delta}   \cup 
(\;V(q,\varepsilon) \cap \{\vert a_i-x_i\} \leq \varepsilon\}),
\end{split}
\end{equation*}
see \cite{Chang-book}. 
Thus we observe  a change of topology of the sublevel sets of $J$ on
\begin{equation}\label{neigbourhood_of_cp_at_infty}
V(q,\varepsilon)
\cap 
\{d(a_{i},x_{i})\ll 1\},    
\end{equation}
while thanks to Proposition \ref{prop_gradient_bounds} we know, that for $\varepsilon>0$ sufficiently small
$
\{\partial J=0\} \cap V(q,\varepsilon) = \emptyset.
$

\

In other words this change of topology happens on \eqref{neigbourhood_of_cp_at_infty} as a \textit{neighbourhood} of the limiting profile \eqref{limiting_profile} of a non compact, energy decreasing, zero weak limit flow line of 
$\Phi$. And in fact this limiting profile  does correspond to a
 \textit{critical point at infinity}, as $J$ exhibits a correct Morse structure on this neighbourhood, cf. Proposition \ref{morse_lemma} and Lemma \ref{non_degenerate_implies_strongly_critical_2}. Hence we may justly  
\begin{enumerate}[label=(\roman*)]
\item say, that this change of topology is induced by this critical point at infinity
\item associate to this critical point at infinity the index 
\begin{equation}\label{index_definition}
\begin{split}
\ind(J,u_{\infty,x_{1},\ldots,x_{q}})
= &
m_{x_{1},\ldots,x_{q}} 
=
m(J_{\tau},u_{\tau,x_{1},\ldots,x_{q}})
=
(q-1)
+
\sum^{q}_{i=1}(n-m(K,x_{i})).
\end{split}
 \end{equation}
\end{enumerate}
This completes the proof. 
\end{proof}  

To establish the Morse structure at infinity, we first require a further orthogonalization.
\begin{lemma}\label{bar_v}
For every  $ \alpha^{i}\varphi_{i}\in V(q,\varepsilon)\cap \{ \vert a_{i}-x_{i} \vert\leq \varepsilon \} $ there exists a unique minimizer for
$ \bar v \in H$ 
\begin{equation*}
\begin{split}
J(\alpha^{i}\varphi_{i}+\bar v)
= &
\min_{v\in H_{\alpha^{i}\varphi_{i}}}
J(\alpha^{i}\varphi_{i}+v),
\end{split}
\end{equation*}
provided $\varepsilon >0$ is sufficiently small, 
and, if $x_{i}\neq x_{j}$ for $i\neq j$, there holds
$
\Vert \bar v \Vert
=
O(
\sum_{r}\frac{\vert \nabla K_{r}\vert}{\lambda_{r}}
+
\sum_{r}\frac{1}{\lambda_{r}^{2}}).
$
\end{lemma}
\begin{proof}
Existence follows from  uniform positivity of $ \partial^{2}J $ on $ H $, cf. \cite{Rey}.  Moreover by Lemma 4.1 in \cite{MM1}
\begin{equation}\label{gradient_on_bar_v}
\partial J_{\tau}(\alpha^{i}\varphi_{i})\bar v
= 
O(
[
\sum_{r}\frac{\vert \nabla K_{r}\vert}{\lambda_{r}}
+
\sum_{r}\frac{1}{\lambda_{r}^{2}}
]
\Vert \bar v\Vert 
),  
\end{equation}
since $ \tau,\theta=0 $, $ n\geq 5 $ and the blow-up points $ a_{i} $ are far from each other, cf. \eqref{eq:eij}. Expanding
\begin{equation*}
\begin{split}
0
= &
\partial J(\alpha^{i}\varphi_{i}+\bar v)\bar v
=
\partial J(\alpha^{i}\varphi_{i})\bar v
+
\partial^{2}J(\alpha^{i}\varphi_{i})\bar v^{2}+o(\Vert \bar v \Vert^{2}),
\end{split}
\end{equation*}
the claimed estimate follows by positivity of the second variation on $ H $  and absorption. 
\end{proof}
Clearly $ \bar v=\bar v_{\alpha_{i}, a_{i},\lambda_{i}} \in H_{\alpha^{i}\varphi}$ and we may represent
every $ u\in V(q,\varepsilon)\cap \{ \vert a_{i}-x_{i} \vert \leq \varepsilon \}  $ uniquely as
\begin{equation*}
\begin{split}
u
= &
\alpha^{i}\varphi_{i}+\bar v + \tilde v.
\end{split}
\end{equation*}
Then by construction
\begin{equation*}
\begin{split}
J(u)
= &
J(\alpha^{i}\varphi_{i}+\bar v +\tilde v)
=
J(\alpha^{i}\varphi_{i}+\bar v)
+
\partial^{2}J(\alpha^{i}\varphi_{i}+\bar v)\tilde v^{2}
+
o(\Vert \tilde v \Vert^{2})
\end{split}
\end{equation*}
and by positivity of $ \partial^{2} J $ on $ H $ and smallness of $ \bar v $ we find 
\begin{equation}\label{tilde_v_positive}
\begin{split}
J(u)
= &
J(\alpha^{i}\varphi_{i}+\bar v)
+
O^{+}(\Vert \tilde v \Vert^{2})
\; \text{ with }\;
0<c\leq O^{+}(1)\leq C<\infty,
\end{split}
\end{equation} 
which is to say, that the $ \tilde v $-direction is a positive, i.e. energy increasing one.  Moreover
\begin{equation}\label{bar_v_negligible}
\begin{split}
J(\alpha^{i}\varphi_{i}+\bar v)
= &
J(\alpha^{i}\varphi_{i})
+
o_{\varepsilon}(\sum_{r}\frac{1}{\lambda_{r}^{2}}),
\end{split}
\end{equation}
as follows by expansion using Lemma \ref{bar_v}  and \eqref{gradient_on_bar_v}. 
\begin{proposition}\label{morse_lemma}
For $x_{i}\neq x_{j}$ for $i\neq j$ and $\varepsilon >0$ sufficiently small there holds 
\begin{equation*}
\begin{split} 
J(u)
=
\hat c_{0}
(\sum_{i} \frac{1}{K^{\frac{n-2}{2}}(x_{i})})^{\frac{2}{n}}
\biggr(
1
& -
\sum_{i=1}^{q-1} (1+o_{\varepsilon}(1))\tilde \alpha_i^2 \\
& -
\sum_{i=1}^{q}(1+o_{\varepsilon}(1))
(\sum_{j}\vert \tilde a_{i,j}^{+} \vert^{2}-\sum_{j}\vert \tilde a_{i,j}^{-} \vert^{2})
+
\sum_{i=1}^{q}\frac{1}{\lambda_{i}^{2}} 
\biggr) 
+
O^{+}(\Vert \tilde v \Vert^{2}),
\end{split}
\end{equation*}
for
$u \in  V(q,\varepsilon)\cap \{ \vert a_{i}-x_{i} \vert\leq \varepsilon \}$,
where upon rescaling
\begin{enumerate}[label=(\roman*)]
\item $ \tilde \lambda_i=\lambda_i $
\item $ \tilde a_{i,j}^{+},\tilde a_{i,j}^{-} $ are local coordinates of $ a_{i} $ in a Morse chart around $ x_{i} $, upon which
\begin{equation*}
\begin{split}
K(a_{i})
= &
K(x_{i})+\sum_{j}\tilde a_{i,j}^{+}-\sum_{j}\tilde a_{i,j}^{-}
\end{split}
\end{equation*}
\item $ \tilde \alpha_i $ are the eigenvectors negative eigenvalues of
$$
\mathcal{A}_{i,j}
=
-\frac{4}{n-2}
(
\delta_{i, j}
-
\frac{K^{\frac{2-n}{4}}(x_i)K^{\frac{2-n}{4}}(x_i)}{\sum_{r}K^{\frac{2-n}{2}}(x_i)}
)
$$
\end{enumerate}
\end{proposition}
\begin{remark}
At this point, cf. \eqref{index_definition}, it is hardly surprising, that for the number of negative directions
$$
(q-1)-\sharp(\tilde a_{i,j}^{+})
=
(q-1)-\sum_{i=1}^{1}\text{coindex}K(x_{i})
=
\ind(J,u_{\infty,x_{1},\ldots,x_{q}}).
$$
\end{remark}
\begin{proof}[Proof of Proposition \ref{morse_lemma}] 
We clearly have to study $ J $ at $\alpha^{i}\varphi_{i}$ only. From Proposition 5.1 in \cite{MM1} we have
\begin{equation*}
J(\alpha^{i}\varphi_{i})
=
\frac
{\hat c_{0}\sum_{i} \alpha_{i}^{2}}
{(\sum_{i} K_{i}\alpha_{i}^{\frac{2n}{n-2}})^{\frac{n-2}{n}}}
\left(
1
-
\hat c_{2}\sum_{i}\frac{\Delta K_{i}}{K_{i}\lambda_{i}^{2}} \frac{\alpha_{i}^{2} }{\sum_{j}\alpha_{j}^{2}} 
\right) 
+ 
o_{\varepsilon}(\sum_{r}\frac{1}{\lambda_{r}^{2}})
\end{equation*}
with positive constants $\hat c_{0},\hat c_{2}$, 
noting, that 
\begin{enumerate}[label=(\roman*)]
\item there is no $ O(\vert \partial J(u) \vert^{2}) $ in the remainder in case $ u=\alpha^{i}\varphi_{i} $
\item $ n\geq 5 $ and $ \tau=0 $
\item the blow-up points are far from each other, since 
\item $\vert a_{i}-x_{i} \vert\leq \varepsilon$ and $ x_{i}\neq x_{j} $ for $ i\neq j $, in particular 
$ \vert \nabla K_{i} \vert=o_{\varepsilon}(1)$   
\end{enumerate}
Moreover and, since $ \vert 1-\frac{rK_{i}\alpha_{i}^{\frac{4}{n-2}}}{k} \vert \leq \varepsilon $, we obtain
\begin{equation*}
J(\alpha^{i}\varphi_{i})
=
\frac
{\hat c_{0}\sum_{i} \alpha_{i}^{2}}
{(\sum_{i} K(a_{i})\alpha_{i}^{\frac{2n}{n-2}})^{\frac{n-2}{n}}}
\left(
1
-
\hat c_{2}
\frac{\sum_{i}\frac{(1+o_{\varepsilon}(1))\Delta K(x_{i})}{K^{\frac{n}{2}}(x_{i})\lambda_{i}^{2}}}
{\sum_{j}K^{\frac{2-n}{2}}(x_{j})},  
\right) 
\end{equation*}
recalling the non degeneracy assumption \eqref{nd}. Passing to the Morse charts and expanding we get
\begin{equation*}
J(\alpha^{i}\varphi_{i})
=
\frac
{\hat c_{0}\sum_{i} \alpha_{i}^{2}}
{(\sum_{i} K(x_{i})\alpha_{i}^{\frac{2n}{n-2}})^{\frac{n-2}{n}}}
\left(
1
-
\hat c_{1}
\frac
{
\sum_{i}\frac{(1+o_{\varepsilon}(1))}{K^{\frac{n}{2}}(x_{i})}
(\sum_{j}\vert \tilde a_{i,j}^{+} \vert^{2}-\sum_{j}\vert \tilde a_{i,j}^{-} \vert^{2})
}
{\sum_{j}K^{\frac{2-n}{2}}(x_{j})}
-
\hat c_{2}
\frac{\sum_{i}\frac{(1+o_{\varepsilon}(1))\Delta K(x_{i})}{K^{\frac{n}{2}}(x_{i})\lambda_{i}^{2}}}
{\sum_{j}K^{\frac{2-n}{2}}(x_{j})}
\right) 
,
\end{equation*}
where $ \hat c_{1}=\frac{n-2}{2} $.  
Finally consider the scaling invariant function
\begin{equation*}
\begin{split}
f(\alpha_{i})
= &
\frac
{\sum_{i} \alpha_{i}^{2}}
{(\sum_{i} K(x_{i})\alpha_{i}^{\frac{2n}{n-2}})^{\frac{n-2}{n}}},
\end{split}
\end{equation*}
whose restriction to 
$$
X_{\alpha}=\{ \sum_{i}K(x_{i})\alpha^{\frac{2n}{n-2}}_{i}=1 \}$$
reflects the restriction of $J$ to $ X $. Then $f\lfloor_{X_{\alpha}} $   
has a unique, strict and non degenerate maximum in 
\begin{equation*}
\begin{split}
\alpha_{i}=\frac{\Theta}{K^{\frac{2-n}{4}}(x_{i})} \; \text{ and there } \;
\frac{1}{2}\partial^2_{\alpha_{i},\alpha_{j}}f(\alpha)
= &
-\frac{4}{n-2}
\frac
{
\delta_{i, j}
-
\frac{\alpha_{i}\alpha_{j}}{\sum_{r}\alpha_{r}^{2}}
}
{(\sum_{r}K(x_{r})\alpha_{r}^{\frac{2n}{n-2}})^{\frac{n-2}{n}}}.
\end{split}
\end{equation*}
In particular 
$
\Theta^{\frac{2n}{n-2}}=\sum_i K^{\frac{n-2}{2}}(x_i)
$
and due smallness of 
$\vert \partial J\vert $ on $V(q,\varepsilon)$ necessarily
$$
\alpha_{i}=\frac{\Theta}{K^{\frac{2-n}{4}}(x_{i})}
+
o_\varepsilon(1) 
$$
Denoting hence by 
$(\tilde \alpha_i)\in \R^{q}$ for $i=1,\ldots, q-1$ the eigenvector 
with negative eigenvalue $\lambda_{\tilde \alpha_i}$ of 
$$
\mathcal{A}_{i,j}
=
-\frac{4}{n-2}
(
\delta_{i, j}
-
\frac
{K^{\frac{2-n}{4}}(x_i)K^{\frac{2-n}{4}}(x_j)}
{\sum_{r}K^{\frac{2-n}{2}}(x_r)}
)
$$
and $(\alpha_q)=(K^{\frac{2-n}{4}}(x_{1}),\ldots,K^{\frac{2-n}{4}}(x_{q}))$
with $\lambda_{\alpha_{q}}=0$, we conclude with 
\begin{equation*}
\begin{split}
J(\alpha^{i}\varphi_{i})
=
\hat c_{0}
(\sum_{i} \frac{1}{K^{\frac{n-2}{2}}(x_{i})})^{\frac{2}{n}}
\biggr(
1
& -
\sum_i \left(1+o_{\varepsilon}(1)\right)\lambda_{\tilde \alpha_i}\tilde \alpha_i^2 \\
& -
\hat c_{1}
\frac
{
\sum_{i}\frac{(1+o_{\varepsilon}(1))}{K^{\frac{n}{2}}(x_{i})}
(\sum_{j}\vert \tilde a_{i,j}^{+} \vert^{2}-\sum_{j}\vert \tilde a_{i,j}^{-} \vert^{2})
}
{\sum_{j}K^{\frac{2-n}{2}}(x_{j})}
-
\hat c_{2}
\frac{\sum_{i}\frac{(1+o_{\varepsilon}(1))\Delta K(x_{i})}{K^{\frac{n}{2}}(x_{i})\lambda_{i}^{2}}}
{\sum_{j}K^{\frac{2-n}{2}}(x_{j})}
\biggr).
\end{split} 
\end{equation*}
Recalling $ \Delta K(x_i)<0 $,  \eqref{tilde_v_positive} and \eqref{bar_v_negligible}, the proposition follows.  
\end{proof} 

Let us conclude with a discussion of the inaccuracy in \cite{bab1}, namely,  that the deformation constructed in its Appendix 2, cf. also \cite{[BCCH4]}, leaves the variational space
  \begin{equation*}
X=\{0\leq u \in W^{1,2}(M) \mid \Vert u \Vert=1\}
\end{equation*}
by not preserving non negativity $u\geq 0$ and not preserving the normalisation $\Vert u \Vert=1$. 
While the first violation is not an issue as exposed in \cite{[BCCH4]}, the latter has to be addressed. 
Let us discuss some possibilities. 
\begin{enumerate}[label=(\roman*)]
 \item \textit{Naive renormalisation}
The most simple approach would be to let flow and renormalise afterwards, which however might lead to a lack of well definedness as a flow on $X=\{ \Vert \cdot \Vert =1 \} $
.
 
\item \textit{Brute force normalisation} 
The  construction in \cite{bab1} is an adaptation at \textit{infinity}, i.e. flow lines of type  
\begin{equation*}
u=\alpha^{i}\varphi_{a_{i},\lambda_{i}}+v.
\end{equation*} 
Leaving the $v$-part aside, the constructed  vectorfield prescribes a movement in $a_{i}$ and $\lambda_{i}$  keeping the scaling parameters $\alpha_{i}$ invariant. Hence one might adjust the $\alpha_{i}$ dynamically, e.g. along 
\begin{equation*}
\frac{\dot\alpha_{i}}{\alpha_{i}}=\beta_{\alpha},\; \beta_{\alpha}=\beta_{\alpha}(a_{i},\lambda_{i})
\end{equation*}
in order to preserve $\Vert u \Vert =1$. This however adds error terms of type
\begin{equation*}
o(\sum_{r}
\vert
1
-
\frac{\sum_{s}\alpha_{s}^{2}}{\sum_{s}K(a_{s})\alpha_{s}^{\frac{2n}{n-2}}}K(a_{r})\alpha_{r}^{\frac{4}{n-2}} 
\vert
) ,
\end{equation*}
when verifying energy decreasing, i.e. $\partial_{t}J(u)\leq 0$, and hence  necessitates a dynamical control of these quantities, which is not available from \cite{bab1} or \cite{[BCCH4]}. We perform this argument here.

\item \textit{Geometric normalisation}  The most intuitive way to adjust the vectorfield $w_{0}$ constructed in \cite{bab1} in order to preserve $\Vert u \Vert=1$ is passing from $\partial_{t}u=w_{0}(u)$ to
\begin{equation}\label{perturbed_flow}
\partial_{t}u=w(u)=w_{0}(u)-\frac{\langle w_{0}(u),u\rangle}{\Vert u \Vert^{2}}u.
\end{equation}
Of course the flow lines for $w_{0}$ and $w$ are then not evidently related and passing from $w_{0}$ to $w$
perturbs the movements in $a_{i}$ and $\lambda_{i}$. But then the statement of Proposition A2 in \cite{bab1}, that away from the \textit{critical points at infinity}
$\overline{\lambda}=\max_{i}\lambda_{i}$ is non increasing, requires justification.

\end{enumerate}

Since there has been a variety of scientific research relying on \cite{bab1} and in particular its Appendix 2, we would like to point out, that in our opinion and based on the availability of better estimates on $ v$  the errors induced by the necessity to  normalise the flow are not critical at least in  low dimensions $n=3,4,5$. 
Also note, that for instance \cite{bab} describing the positivity and norm preserving gradient flow, is not affected.  

\

\begin{center}
\hline
\textbf{Acknowledgment} \\
 M.Mayer has been supported by the Italian MIUR Department of Excellence grant CUP E83C18000100006. 
\end{center}

\end{document}